\documentclass[11pt]{article}

\usepackage[T1]{fontenc}
\usepackage{lmodern}
\usepackage{amsmath,amssymb,amsthm,mathtools}
\usepackage{enumitem}
\usepackage{microtype}
\usepackage[hidelinks]{hyperref}

\theoremstyle{plain}
\newtheorem{theorem}{Theorem}[section]
\newtheorem*{maintheorem}{Main Theorem}
\newtheorem{proposition}[theorem]{Proposition}
\newtheorem{lemma}[theorem]{Lemma}

\theoremstyle{definition}
\newtheorem{definition}[theorem]{Definition}
\newtheorem{convention}[theorem]{Convention}
\theoremstyle{remark}
\newtheorem{remark}[theorem]{Remark}
\numberwithin{equation}{section}

\newcommand{\Z}{\mathbb Z}
\newcommand{\Q}{\mathbb Q}

\newcommand{\Zp}{\mathbb Z_p}
\newcommand{\Qp}{\mathbb Q_p}
\newcommand{\Fitt}{\operatorname{Fitt}}
\newcommand{\Isol}{\operatorname{Isol}}
\newcommand{\tor}{\operatorname{tor}}
\newcommand{\Ker}{\operatorname{Ker}}
\newcommand{\im}{\operatorname{Im}}

\newcommand{\spec}{\operatorname{spec}}

\title{The Conjugacy Problem for Recursively Presented Virtually Soluble Groups of Finite Pr\"ufer Rank}
\author{M. Brescia \and E. Ingrosso \and M. Trombetti}
\date{}

\usepackage{xcolor}

\begin{document}
\maketitle

\begin{abstract}
We solve a long-standing conjugacy problem for finitely presented soluble groups of finite Pr\"ufer rank, recorded as Problem~9.52 in the Kourovka Notebook, in the stronger setting of recursively presented finitely generated virtually soluble groups. Robinson proved that, in every recursively presented finitely generated virtually soluble minimax group, conjugacy to any fixed element is decidable, while explicitly lefting open whether the procedure could be made uniform in the target element. We establish this uniformity. More specifically, we prove that every recursively presented finitely generated virtually soluble group of finite Pr\"ufer rank has decidable Conjugacy Problem.
\end{abstract}

\section{Introduction}

The Word and Conjugacy Problems are two of the basic decision problems in combinatorial group theory. A solution of the Conjugacy Problem always gives a solution of the Word Problem, since an element is conjugate to the identity if and only if it is the identity. The converse fails even for finitely presented groups (see \cite{Bokut67}), and Collins \cite{Collins69} even showed that the Conjugacy Problem of a finitely presented group can realize arbitrary recursively enumerable degrees of unsolvability. {The solvability of the Conjugacy Problem also need not pass to finite extensions, as shown by Collins and Miller \cite{CollinsMiller1977}.}

{Within soluble group theory, finitely generated minimax groups form a natural extension of the class of polycyclic groups: their defining series may have quasicyclic factors in addition to cyclic factors. The Conjugacy Problem is decidable for virtually polycyclic groups, by the conjugacy-separability results of Remeslennikov and Formanek \cite{Remeslennikov1969,Formanek1976}. However, conjugacy separability cannot provide a general solution beyond the polycyclic case: Wehrfritz \cite{Wehrfritz1973} constructed a finitely presented soluble minimax group which is not conjugacy separable. On the other hand, virtually torsion-free minimax groups admit faithful rational linear representations \cite{KL20}, providing a natural starting point for a different approach.}

A finitely generated group is \emph{recursively presented} if, on a fixed finite generating set, it has a recursively enumerable set of defining relators. For soluble groups of finite {Pr\"ufer} rank, the Word Problem is much better behaved. Cannonito and Robinson \cite{CR84} proved that a finitely generated soluble group of finite {Pr\"ufer} rank has decidable Word Problem if and only if it is recursively presented. Robinson later proved that, in a recursively presented finitely generated virtually soluble minimax group $G$, if $g$ is a fixed element of $G$, one can decide whether an input element is conjugate to $g$ (see {\cite{Robinson1986} and }\cite[Theorem~9.3.5]{LR04}). In his argument, auxiliary data depending on the fixed $g$ may be incorporated into the resulting algorithm without being computed uniformly from $g$. When $g$ is allowed to vary as part of the input, these data can no longer be treated as fixed constants. {Robinson explicitly stated that the question of a uniform algorithm was open \cite{Robinson1986}.} The uniformity problem is therefore to show that all target-dependent information required by the decision procedure can itself be obtained effectively from the target. This is the issue behind Problem~9.52 of the Kourovka Notebook \cite{Kourovka21} for finitely presented soluble groups of finite rank. Our argument yields the following slightly stronger statement, allowing recursive presentations and virtually soluble groups.

\begin{maintheorem}
Let $G$ be a finitely generated virtually soluble group of finite {Pr\"ufer} rank, equipped with a fixed recursive presentation on a finite generating set. For this fixed group $G$, there exists an algorithm which, on input two words representing elements $x,y\in G$, decides whether there exists $g\in G$ such that $x^g=y$.
\end{maintheorem}

{The theorem is uniform in the two input elements for each fixed group $G$; no uniformity in the recursive presentation of $G$ is claimed. Once conjugacy is decidable, a conjugating element can also be found, when one exists, by enumerating words and using the Word Problem.}

\smallskip

The result already includes all recursively presented virtually torsion-free virtually soluble minimax groups. Beyond the virtually polycyclic setting, solvability of conjugacy was previously known for particular families, including the Baumslag--Solitar groups covered by Anshel and Stebe \cite{AnshelStebe1974} and certain HNN extensions of finitely generated abelian groups treated by Sokolov \cite{Sokolov2005}. These special cases do not cover arbitrary virtually torsion-free minimax groups. In the virtually torsion-free case of our proof, one may take the cover to be the identity $H=G$. Hence $K=K_0=1$, $J_i=W_i=0$ in the notation below, and $E=H/K_0=G$. The conjugacy-lifting algorithm of Sections~\ref{sec:3}--\ref{sec:5} then proves the result directly, without the final reduction in Section~\ref{sec:6}. For a general minimax group, by contrast, conjugate elements may have chosen lifts that are not conjugate in $H$. Keeping track of the kernel throughout the lifting procedure and passing from $E$ to $G$ therefore requires additional work.

The proof proceeds by passing to the virtually torsion-free minimax cover $H\twoheadrightarrow G$ supplied by Kropholler and Lorensen \cite{KL20}, and then solving a sequence of lifting problems. Inside $H$ we choose a torsion-free nilpotent normal subgroup $N_0$ such that $H/N_0$ is virtually abelian. Conjugacy in this quotient can be decided by integral linear algebra. Once conjugacy has been settled modulo $N_0$, it remains to determine whether a conjugacy found in the quotient can be lifted to the whole group. We do this successively along a central series of $N_0$, correcting the conjugating element one layer at a time. At each stage, the remaining obstruction lies in an abelian minimax group and is expressed as an image-membership problem for a homomorphism, or, in its natural general form, as a derivation equation. We decide this membership problem by passing to the rational completion and to the tensor $p$-completions $X_p$. In this way, each lifting step is reduced to rational linear algebra, effective $p$-adic decision procedures, and finite calculations in nilpotent groups.

\smallskip

The layout of the paper is as follows. Section~\ref{sec:2} collects the structural and computational tools used throughout the argument, including the finite-rank structure of the groups involved, rational and tensor $p$-completions, coordinate systems for torsion-free nilpotent minimax groups, and the required effective $p$-adic computations. Section~\ref{sec:3} fixes the virtually torsion-free cover of $G$ and makes effective the subgroup, coordinate, and membership data attached to the nilpotent subgroup $N_0$ and its successive abelian factors. Section~\ref{sec:4} develops the effective kernel and image computations needed in the lifting argument, both in the rational completion and in the tensor $p$-completions. Section~\ref{sec:5} first decides conjugacy in the virtually abelian quotient and then lifts a conjugacy successively through the fixed series of the nilpotent normal subgroup, using the results of Section~\ref{sec:4} to decide the obstruction at each stage. Finally, Section~\ref{sec:6} transfers the resulting decision procedure from the virtually torsion-free minimax cover back to the original group and proves the Main Theorem.

\section{Preliminaries}\label{sec:2}

We collect here the structural and computational results used throughout the paper.  Proposition~\ref{prop:fatr} relates the finite-rank hypothesis to the minimax condition for finitely generated soluble groups.  We then recall rational and tensor $p$-completions of nilpotent minimax groups.  Lemma~\ref{lem:tensor} records the functoriality and exactness properties of tensor $p$-completion that will be used repeatedly, while Lemma~\ref{lem:coordinates} provides a fixed coordinate system for the rational completion of a torsion-free nilpotent minimax group, together with rational polynomials describing multiplication and powers.

For torsion-free abelian minimax groups, we give an explicit description of their tensor $p$-completions in Lemma~\ref{lem:abelian-q-completion} and an effective lifting procedure in Lemma~\ref{lem:factorlift}.  The $p$-adic calculations require exact arithmetic in finitely generated subfields of $\Qp$: Lemma~\ref{lem:field} provides this arithmetic, and Lemma~\ref{lem:additive} gives the corresponding effective linear operations on subgroups of finite-dimensional $\Qp$-vector spaces. Finally, Theorem~\ref{thm:KL} supplies the virtually torsion-free cover of $G$ on which the subsequent constructions are carried out.

\smallskip

A group $G$ has \textnormal(\emph{Pr\"ufer\textnormal) rank $r$} if every finitely generated subgroup of $G$ can be generated by at most $r$ elements, and $r$ is the smallest integer with this property; it has \emph{finite \textnormal(Pr\"ufer\textnormal) rank} if such an $r$ exists.  For an abelian group~$A$, write $r_0(A)$ for its torsion-free rank and
\[
 r_p(A)=\dim_{\mathbf F_p}A[p],
 \qquad
 A[p]=\{a\in A:pa=0\}.
\]
Following \cite[Section~5.1]{LR04}, its total and reduced ranks are
\[
 r_{\rm tot}(A)=r_0(A)+\sum_p r_p(A),
 \qquad
 \widehat r(A)=r_0(A)+\sup_p r_p(A).
\]
For abelian groups, finite Pr\"ufer rank is equivalent to $\widehat r(A)<\infty$.  Finite total rank is stronger in general.  A soluble group is a \emph{FATR-group} if it has a finite series whose factors are abelian groups of finite total rank.

For a prime $p$, write $C_{p^\infty}=\Z[1/p]/\Z$ for the quasicyclic $p$-group.  A group is \emph{minimax} if it has a finite series whose factors are finite, infinite cyclic, or quasicyclic.  We shall also use the standard term \emph{rational group} for a group isomorphic to a subgroup of the additive group of rational numbers $(\Q,+)$, which of course is not minimax. The following standard equivalence for finitely generated soluble groups is recalled in \cite[Section~9.3]{LR04}.

\begin{proposition}\label{prop:fatr}
For a finitely generated soluble group $S$, the following are equivalent:
\begin{enumerate}[label=\textup{(\roman*)}]
\item $S$ has finite Pr\"ufer rank\textnormal;
\item $S$ is minimax\textnormal;
\item $S$ is an FATR-group.
\end{enumerate}
\end{proposition}

In order to carry out our work, we need to recall the notions of rational and $p$-adic completions. First note that, throughout, groups are written multiplicatively, while abelian factors are written additively. Moreover, we use the following standard conventions
\[
 a^b=b^{-1}ab,
 \qquad
 [a,b]=a^{-1}b^{-1}ab=a^{-1}a^b
\] for $a,b$ elements of a group. Let $U$ be a torsion-free nilpotent group. By  \cite[Theorem~2.1.1]{LR04}, there exists a nilpotent group $U_{\Q}$ containing $U$ such that roots can be extracted uniquely in $U_{\Q}$ and every element of $U_{\Q}$ has a positive power lying in $U$. Moreover, $U_{\Q}$ is unique up to isomorphism. We call $U_{\Q}$ the \emph{rational completion}, or \emph{Mal'cev completion}, of $U$. Thus, $U_{\Q}$ is torsion-free and radicable. Note that if $U$ is also abelian, then $U_{\Q}\simeq U\otimes_{\Z}\Q$. More generally, for an arbitrary abelian group $A$, we write $A_{\Q}:=A\otimes_{\Z}\Q$. Thus, when $A$ is torsion-free, this notation agrees with the rational completion defined above. If $B\leq A$, then, since $\Q$ is a flat $\Z$-module, the natural map $B_{\Q}\to A_{\Q}$ is injective, and we identify $B_{\Q}$ with its image in $A_{\Q}$.

If $U$ is torsion-free nilpotent and $Y\leq U$, the \emph{rational closure} of $Y$ in $U_{\Q}$ is
\begin{equation}
\overline Y^{\Q}
=
\Isol_{U_{\Q}}(Y)
=
\{z\in U_{\Q}:z^m\in Y\text{ for some }m\geq1\}.
\label{eq:4.1}
\end{equation}
Under the natural embedding $Y_{\Q}\leq U_{\Q}$, one has
$Y_{\Q}=\overline Y^{\Q}$. A subgroup $R\leq U_{\Q}$ is called
\emph{rational} if $R=\Isol_{U_{\Q}}(R)$.

\smallskip

We now define the $p$-adic objects which will be used below. Fix a prime $p$. Let $\Qp$ denote the field of $p$-adic numbers. Every non-zero element $x\in\Qp$ has a unique expansion
$x=\sum_{n=N}^{\infty}a_np^n$, where $N\in\Z$, $a_n\in\{0,1,\ldots,p-1\}$, and $a_N\neq0$. The $p$-adic valuation of $x$ is $v_p(x)=N$, and we put $v_p(0)=+\infty$. In particular, $v_p(p)=1$. The $p$-adic topology on $\Qp$ is the topology for which the sets $x+p^n\Zp$, with $n\in\Z$, form a neighbourhood basis at $x$, where $\Zp=\{z\in\Qp:v_p(z)\geq0\}$. Thus a sequence $(x_k)$ converges to $x$ if and only if $v_p(x_k-x)\to+\infty$. We also write $\Z_{(p)}=\{a/b\in\Q:p\nmid b\}$. Under the natural embedding $\Q\leq\Qp$, one has $\Z_{(p)}=\Q\cap\Zp$, and both $\Z$ and $\Z_{(p)}$ are dense in $\Zp$.

For a finitely generated group $J$, its {\it pro-$p$ completion} is
$\widehat J_p=\varprojlim_R J/R$, where $R$ runs over the normal subgroups of $J$ such that $J/R$ is a finite $p$-group. For each such $R$, let $K_R$ denote the kernel of the canonical projection $\widehat J_p\to J/R$. We endow $\widehat J_p$ with the inverse-limit topology, in which the subgroups $K_R$ form a neighbourhood basis of the identity, and hence the cosets $xK_R$, with $x\in\widehat J_p$, form a basis of open sets. Now let $X$ be nilpotent of finite torsion-free rank, and let $\mathcal H(X)$ be the directed set of its finitely generated subgroups.  If $J\leq K$ belong to $\mathcal H(X)$, the natural map $\widehat J_p\to\widehat K_p$ is injective by \cite[Lemma~5.1(ii)]{KL20}.  We therefore  define the \emph{tensor $p$-completion} of $X$ by putting
\begin{equation}
 X_p=\varinjlim_{J\in\mathcal H(X)}\widehat J_p.
 \label{eq:tensor-p-definition}
\end{equation}
This is the definition from \cite[Section~5.1]{KL20}.  We endow $X_p$ with the direct-limit topology: a subset $U\subseteq X_p$ is open exactly when $U\cap\widehat J_p$ is open in $\widehat J_p$ for every $J\in\mathcal H(X)$.   If~$X$ is finitely generated, then $X_p=\widehat X_p$.  If~$X$ is torsion-free abelian, then
\begin{equation}
 X_p\simeq X\otimes_{\Z}\Zp.
 \label{eq:abelian-tensor-p}
\end{equation}
Thus, for example, $(\Z^d)_p=\Zp^d$ and
$\Z[1/p]_p=\Z[1/p]\otimes_{\mathbb Z}\Zp=\Qp$. 


For nilpotent minimax groups, tensor $p$-completion has the following functoriality and exactness properties.

\begin{lemma}\label{lem:tensor}
Let $X$ be a nilpotent minimax group, and let $A\leq X$.  For every prime $p$ the following hold.
\begin{enumerate}[label=\textup{(\roman*)}]
\item The natural homomorphism $A_p\to X_p$ is injective.
\item If $A\trianglelefteq X$, then $X_p/A_p\simeq (X/A)_p$.
\item Every homomorphism $f:X\to Y$ between nilpotent groups of finite torsion-free rank extends uniquely to a continuous homomorphism $f_p:X_p\to Y_p$.
\end{enumerate}
\end{lemma}

\begin{proof}
Statement \textup{(i)} follows from the injectivity of the maps between pro-$p$ completions of finitely generated nilpotent subgroups (see \cite[Lemma~5.1(ii)]{KL20}), after passage to the direct limits defining tensor completion.  Statement \textup{(ii)} is \cite[Lemma~5.4(iii)]{KL20}, and \textup{(iii)} follows from \cite[Proposition~5.3]{KL20}.
\end{proof}

In view of Lemma~\ref{lem:tensor}\textup{(i)}, whenever $A\leq X$ are nilpotent minimax groups we shall identify $A_p$ with its image in $X_p$.

Let $U$ be a torsion-free nilpotent group, and let $V\leq U$. The {\it isolator} of~$V$ in $U$ is defined as follows:
$$
 \Isol_U(V)=\{u\in U: u^n\in V\text{ for some }n\geq1\}.
$$
A subgroup $V\le U$ is said {\it isolated} if it coincides with its isolator in $U$.

For a torsion-free nilpotent minimax group $U$, set
\[
 U_i=\Isol_U(\gamma_i(U)),\qquad i=1,\ldots,c+1,
\]
where $c$ is the nilpotency class and $U_{c+1}=1$. Now, \cite[Theorem~2.3.1(i)]{LR04} shows that the isolator of every subgroup of $U$ is a subgroup, and hence is isolated. In particular, since the subgroups $\gamma_i(U)$ are fully invariant in $U$, it follows directly from the definition of the isolator that each $U_i$ is characteristic in $U$. Moreover, \cite[Theorem~2.3.5]{LR04} gives
$$
 [\Isol_U(A),\Isol_U(B)]
 \leq \Isol_U([A,B])
$$
for all subgroups $A,B\leq U$.  Since
$[\gamma_i(U),\gamma_j(U)]\leq\gamma_{i+j}(U)$, we also obtain
\begin{equation}
 [U_i,U_j]\leq U_{i+j}.
 \label{eq:2.1}
\end{equation}
In particular,
$[U_i,U_i]\leq U_{2i}\leq U_{i+1}$, so $U_i/U_{i+1}$ is abelian.
Since $U_{i+1}$ is isolated in $U$, the quotient $U_i/U_{i+1}$ is
torsion-free.  Thus every $U_i/U_{i+1}$ is torsion-free
abelian of finite rank.

\medskip

The next result is fundamental and provides a coordinate system for the rational completion of  a torsion-free nilpotent minimax group. 

\begin{lemma}\label{lem:coordinates}
Let $U$ be a torsion-free nilpotent minimax group, and let $h$ be its Hirsch length. There exist a finitely generated subgroup $U^{\ast}\leq U$, elements $u_1,\ldots,u_h\in U^{\ast}$, and, for $k=1,\ldots,h$, polynomials
\begin{align*}
P_k&\in\Q[X_1,\ldots,X_h,Y_1,\ldots,Y_h],\\
R_k&\in\Q[X_1,\ldots,X_h,T],
\end{align*}
with the following properties\textnormal:
\begin{enumerate}[label=\textup{(\roman*)}]
\item $U=\Isol_U(U^{\ast})$.

\item Every element $x\in U^{\ast}$ has a unique expression $x=u_1^{a_1}\ldots u_h^{a_h}$, where
$$
\operatorname{cr}(x)=(a_1,\ldots,a_h)\in\Z^h
$$
is the \emph{coordinate vector} of $x$.

\item If $x,y\in U^{\ast}$, then
\begin{align}
\operatorname{cr}(xy)
&=
\bigl(
P_1(\operatorname{cr}(x),\operatorname{cr}(y)),\ldots,
P_h(\operatorname{cr}(x),\operatorname{cr}(y))
\bigr),
\label{eq:2.3}\\
\operatorname{cr}(x^n)
&=
\bigl(
R_1(\operatorname{cr}(x),n),\ldots,R_h(\operatorname{cr}(x),n)
\bigr)
\qquad (n\in\Z).
\label{eq:2.5}
\end{align}

\item Every element $x\in U_{\Q}$ has a unique expression
$$
x=u_1^{a_1}\ldots u_h^{a_h},
\qquad
a_1,\ldots,a_h\in\Q.
$$
With respect to these coordinates, the same polynomials $P_k$ give multiplication in $U_{\Q}$, while, for $\lambda\in\Q$, the polynomials $R_k$ give the coordinates of the $\lambda$-th power.
\end{enumerate}
\end{lemma}
\begin{proof}
For \textup{(i)}, for every $i$ choose a finite set $S_i\subseteq U_i$ whose images form a~\hbox{$\Q$-basis} of $(U_i/U_{i+1})\otimes_{\Z}\Q$, and put
$U^{\ast}=\langle S_1,\ldots, S_c\rangle$. Clearly, $U_i\leq\Isol_U(U^{\ast})$ for every $i$, and hence  $U=\Isol_U(U^{\ast})$. This proves (i).

\smallskip

For (ii), each $U_i\cap U^{\ast}$ is isolated in $U^{\ast}$, and hence
$$
U^{\ast}=U_1\cap U^{\ast}\geq U_2\cap U^{\ast}\geq\ldots\geq U_{c+1}\cap U^{\ast}=\{1\}
$$
is a central series of the finitely generated torsion-free
nilpotent group $U^{\ast}$ whose factors are torsion-free (and obviously finitely generated and abelian). It is therefore possible to find elements $u_1,\ldots,u_h\in U^{\ast}$ such that every element of
$U^{\ast}$ has a unique expression
$$
 u_1^{a_1}\ldots u_h^{a_h},
 \qquad
 a_1,\ldots,a_h\in\Z,
$$
and such that, for every $i$, the subgroup $U_i\cap U^{\ast}$ is
generated by a final segment of $u_1,\ldots,u_h$ (this is a Mal'cev basis as defined in \cite[Section~4.2.2]{CMZ17}). This proves~\textup{(ii)}.

\smallskip

For \textup{(iii)}, apply  \cite[Theorem~4.9]{CMZ17} to $u_1,\ldots,u_h$. It gives rational polynomials $P_k$ and $R_k$ such that multiplication and integral powers in $U^{\ast}$ are given by \eqref{eq:2.3} and \eqref{eq:2.5}. This proves \textup{(iii)}.

\smallskip

For \textup{(iv)}, note first that $U_{\Q}$ is also a Mal'cev completion of $U^{\ast}$, because $U=\Isol_U(U^{\ast})$. Hence the $R$-completion construction of \cite[Section~4.3]{CMZ17}, together with \cite[Theorems~4.11 and~4.13]{CMZ17}, gives the unique expressions
$$
x=u_1^{a_1}\ldots u_h^{a_h},
\qquad
a_1,\ldots,a_h\in\Q,
$$
for $x\in U_{\Q}$, with multiplication and rational powers given by the same polynomials $P_k$ and $R_k$. This proves \textup{(iv)}.
\end{proof}

Retain the notation of Lemma~\ref{lem:coordinates}. We introduce vector notation for the coordinate
polynomials. If $\mathbf a,\mathbf b$ are tuples of admissible
coordinates and $t$ is a scalar, put
\begin{align}
 P(\mathbf a,\mathbf b)
 &=
 \bigl(
 P_1(\mathbf a,\mathbf b),\ldots,
 P_h(\mathbf a,\mathbf b)
 \bigr),
 \label{eq:P-vector}\\
 R(\mathbf a,t)
 &=
 \bigl(
 R_1(\mathbf a,t),\ldots,R_h(\mathbf a,t)
 \bigr).
 \label{eq:R-vector}
\end{align}
Thus, for example, if $x,y\in U_{\Q}$ have coordinate vectors
$\mathbf a$ and $\mathbf b$, respectively, then $xy$ has coordinate
vector $P(\mathbf a,\mathbf b)$.

Now, fix a prime $p$. Since the coefficients of the polynomials
$P_k$ and $R_k$ belong to $\Q$, the same polynomials may be
evaluated on elements of $\Qp$. We define a multiplication on
$\Qp^h$ by $\mathbf a\ast\mathbf b=P(\mathbf a,\mathbf b)$. The group identities satisfied by the coordinate polynomials over
$\Q$ remain valid over~$\Qp$. Hence~$\Qp^h$, with this multiplication,
is a torsion-free nilpotent group. We denote it by $U_{\Q,p}$. (This is the $\mathbb Q_p$-completion of $U^\ast$ as described in \cite[Section~4.3]{CMZ17}.) Again the same polynomials $P_k$ give multiplication in $U_{\Qp}$, while, for $\lambda\in\Qp$, the polynomials $R_k$ give the coordinates of the $\lambda$-th power. We also extend the scope of the coordinate vectors to $\Qp$. 

In $U_{\Q,p}$, we can find a natural copy of the tensor $p$-completion of $U$. In fact, if $N$ is a
torsion-free nilpotent group of finite torsion-free rank, then its tensor $p$-completion $N_p$ is naturally isomorphic to the closure of
$N$ in its $\Qp$-completion (see the beginning of Section 5.1 of \cite{KL20}). Thus,
$$
 U_p=\overline{U}^{\,U_{\Q,p}}\leq U_{\mathbb Q,p},
$$ where the overline denotes the closure.

We need the following observation. First, conjugation is also given by the coordinate
polynomials. For $\mathbf a,\mathbf b\in\Qp^h$, put
\begin{equation}
 C(\mathbf a,\mathbf b)
 =
 P\bigl(R(\mathbf b,-1),P(\mathbf a,\mathbf b)\bigr).
 \label{eq:2.6}
\end{equation}
If $x$ and $y$ have coordinate vectors $\mathbf a$ and $\mathbf b$,
respectively, then $x^y=y^{-1}xy$ has coordinate vector
$C(\mathbf a,\mathbf b)$. 

Similarly, the action of every automorphism of $U$ on $U^\ast$ can be described in these coordinates by a polynomial map. Let
$\alpha\in\operatorname{Aut}(U)$, and suppose that the coordinate vector
\[
 \mathbf a^{(j)}
 :=
 \operatorname{cr}(\alpha(u_j))
 \in\Q^h
\]
is known for every $j=1,\ldots,h$. For integral
$\mathbf x=(x_1,\ldots,x_h)$ one has
\[
 \alpha\bigl(u_1^{x_1}\ldots u_h^{x_h}\bigr)
 =
 \alpha(u_1)^{x_1}\ldots\alpha(u_h)^{x_h}.
\]
The coordinate vector of $\alpha(u_j)^{x_j}$ is $R(\mathbf a^{(j)},x_j)$. Hence the coordinate vector of the whole product is obtained by
successive applications of $P$. Explicitly, define {\small
\begin{align}
 A_{\alpha}(\mathbf x)
 \!:=\!
 P\Bigl(
   \ldots
   P\bigl(
      P(
        R(\mathbf a^{(1)},x_1),
        R(\mathbf a^{(2)},x_2)
      ),
      R(\mathbf a^{(3)},x_3)
   \bigr)
   \ldots,
   R(\mathbf a^{(h)},x_h)
 \Bigr).
 \label{eq:external-aut}
\end{align}}
Then $A_{\alpha}(\mathbf x)$ is the coordinate vector of
$\alpha(u_1^{x_1}\ldots u_h^{x_h})$. 

Finally, let $g\in U_p$ have coordinate vector
$\mathbf a\in\Qp^h$, and let $z\in\Zp$. Choose integers
$n_k\in\Z$ such that $n_k\longrightarrow z$ in $\Zp$. For every $k$, the coordinate vector of $g^{n_k}$ is $R(\mathbf a,n_k)$. Since each $R_i$ is a polynomial over $\Qp$,
\[
 R(\mathbf a,n_k)\longrightarrow R(\mathbf a,z).
\]
Since $U_p$ is closed in $U_{\Q,p}$, the sequence $g^{n_k}$ converges
to an element of $U_p$. We define
\[
 g^z:=\lim_{k\to\infty}g^{n_k}.
\]
Its coordinate vector is $R(\mathbf a,z)$. In particular, the limit is
independent of the chosen sequence $(n_k)$. This is the meaning of a
$\Zp$-power throughout the paper. Moreover, continuous homomorphisms between tensor $p$-completions commute with $\Z_p$-powers, since, if $\varphi:U_p\to V_p$ is such a homomorphism, $g\in U_p$, $z\in\Z_p$, and $n_k\to z$ in $\Z_p$, then
$$
\varphi(g^z)  
=
\varphi\left(\lim_{k\to\infty}g^{n_k}\right)
=
\lim_{k\to\infty}\varphi(g)^{n_k}
=
\varphi(g)^z.
$$

\medskip

Let $B$ be a torsion-free abelian minimax group. Choose a free abelian subgroup $L\leq B$ with $r:=r_0(L)=r_0(B)$ such that $B/L$ is divisible. Then
\begin{equation}
 B/L\simeq\bigoplus_{q\in\spec(B)} C_{q^\infty}^{\,r_q},
 \label{eq:2.7}
\end{equation}
where $\spec(B)$ is a finite set of primes and each $r_q$ is a non-negative integer (see, for instance, \cite[Section~5.1]{LR04}). For each $q\in\spec(B)$, put $C_{q'}$ to be the $q'$-component of $B/L$, so that
\[
 B/L\simeq C_{q^\infty}^{\,r_q}\oplus C_{q'},
\]
and $C_{q'}$ has no non-trivial elements of $q$-power order. For each $q\in\spec(B)$ and $1\leq j\leq r_q$, choose a sequence $c_{q,j,n}$, $n\geq1$, in the $j$-th copy of $C_{q^\infty}$ such that $|c_{q,j,n}|=q^n$ and $q c_{q,j,n+1}=c_{q,j,n}$, and choose $b_{q,j,n}\in B$ whose image under this decomposition is $c_{q,j,n}$ in the $j$-th copy of $C_{q^\infty}$ and zero elsewhere. Then
\begin{equation}
 qb_{q,j,n+1}-b_{q,j,n}\in L.
 \label{eq:2.8}
\end{equation}
Now, define $v_{q,j,n}:=q^n b_{q,j,n}\in L$ for every $n\geq1$. Via the canonical embedding
\[
 L\longrightarrow L\otimes_{\Z}\Z_q,
 \qquad
 \ell\longmapsto \ell\otimes1,
\]
we regard each $v_{q,j,n}$ as an element of
$L\otimes_{\Z}\Z_q\simeq\Z_q^r$ (which is regarded with the product topology).

\begin{lemma}\label{lem:abelianpadic}
For every $j$, the sequence $\bigl(v_{q,j,n}\bigr)_{n\geq1}$ is Cauchy in
$L\otimes_{\Z}\Z_q$, and therefore converges to an element $v_{q,j}\in L\otimes_{\Z}\Z_q$. Moreover, for every $n\geq1$, we have $v_{q,j}-v_{q,j,n}\in q^n(L\otimes_{\Z}\Z_q)$.
\end{lemma}
\begin{proof}
By \eqref{eq:2.8}, $qb_{q,j,n+1}-b_{q,j,n}\in L$. Therefore
\begin{align*}
 v_{q,j,n+1}-v_{q,j,n}
 &=
 q^{n+1}b_{q,j,n+1}-q^n b_{q,j,n}\\
 &=
 q^n\bigl(qb_{q,j,n+1}-b_{q,j,n}\bigr)
 \in q^nL.
\end{align*}
After embedding $L$ in $L\otimes_{\Z}\Z_q$, this gives
\[
 v_{q,j,n+1}-v_{q,j,n}
 \in
 q^n(L\otimes_{\Z}\Z_q).
\]

Let $m>n$. Then
\[
 v_{q,j,m}-v_{q,j,n}
 =
 \sum_{r=n}^{m-1}
 \bigl(v_{q,j,r+1}-v_{q,j,r}\bigr).
\]
For every $r\geq n$,
\[
 v_{q,j,r+1}-v_{q,j,r}
 \in q^r(L\otimes_{\Z}\Z_q)
 \subseteq
 q^n(L\otimes_{\Z}\Z_q).
\]
Consequently
\[
 v_{q,j,m}-v_{q,j,n}
 \in
 q^n(L\otimes_{\Z}\Z_q).
\]
Thus, for every $N\geq1$ and every $m,n\geq N$,
\[
 v_{q,j,m}-v_{q,j,n}
 \in
 q^N(L\otimes_{\Z}\Z_q).
\]
This is precisely the required Cauchy condition. Since
$q^n(L\otimes_{\Z}\Z_q)$ is closed in $L\otimes_{\Z}\Z_q$, letting
$m\to\infty$ in
$$
v_{q,j,m}-v_{q,j,n}\in q^n(L\otimes_{\Z}\Z_q)
$$
gives
$$
v_{q,j}-v_{q,j,n}\in q^n(L\otimes_{\Z}\Z_q),
$$
as required.
\end{proof}

Before stating the next lemma, we fix the ambient $q$-adic vector
space in which all the groups involved will be compared. We start by noticing that
$$
L\otimes_{\Z}\Q
\simeq
B\otimes_{\Z}\Q,
$$
and we put
$$
V_q:=
(B\otimes_{\Z}\Q)\otimes_{\Q}\Q_q.
$$  Thus, $B\otimes_{\Z}\Q$ is the rational completion of $B$, while $V_q$ is the $q$-adic Mal'cev completion of $B$. Note that since $L\simeq\Z^r$, we have
$V_q\simeq\Q_q^r$.

The natural map
$$
B\longrightarrow V_q,
\qquad
b\longmapsto (b\otimes1)\otimes1,
$$
is injective, since $B$ is torsion-free.  There is also a natural injective map
$$
B\otimes_{\Z}\Z_q\longrightarrow V_q,\qquad
b\otimes z
\longmapsto
(b\otimes1)\otimes z.
$$ Thus $B$, $L\otimes_{\Z}\Z_q$, and
$B\otimes_{\Z}\Z_q$ are all to be regarded as subgroups of the same
ambient group $V_q$.  
Now, using the natural identification
$B_q\simeq B\otimes_{\Z}\Z_q$ already noticed in \eqref{eq:abelian-tensor-p}, we regard $B_q$ as a subgroup of $V_q$ through the embedding above. The next lemma describes $B_q$ as a subgroup of $V_q$.

\begin{lemma}\label{lem:abelian-q-completion}
We have
$B_q=
(L\otimes_{\Z}\Z_q)
+
\sum_{j=1}^{r_q}\Q_q v_{q,j}
$. If $q\notin\spec(B)$, then $r_q=0$, so that the sum is empty and
$B_q=L\otimes_{\Z}\Z_q$. Moreover, 
$B_q=\overline{B}^{\,V_q}$.
\end{lemma}
\begin{proof}
Put
$$
D_q:=
(L\otimes_{\Z}\Z_q)
+
\sum_{j=1}^{r_q}\Q_q v_{q,j}
\leq V_q.
$$ First,
$$
L\otimes_{\Z}\Z_q\leq B\otimes_{\Z}\Z_q=B_q,
$$
because $L\leq B$ and tensoring the injection $L\hookrightarrow B$
with $\Z_q$ preserves injectivity. Fix $1\leq j\leq r_q$.  We claim that
$$
q^{-m}v_{q,j}\in B_q
$$
for every $m\geq1$. By Lemma \ref{lem:abelianpadic} $v_{q,j,n}\longrightarrow v_{q,j}$ in the $q$-adic topology of $L\otimes_{\Z}\Z_q$. Hence, there exists an integer $N\geq m$ such that
$$
v_{q,j}-v_{q,j,N}
\in
q^m(L\otimes_{\Z}\Z_q).
$$
Thus there exists
$\varepsilon\in L\otimes_{\Z}\Z_q$ such that
$v_{q,j}-v_{q,j,N}=q^m\varepsilon$. Since
$v_{q,j,N}=q^N b_{q,j,N}$, and $q$ is invertible inside $V_q$, we obtain,
\begin{align*}
q^{-m}v_{q,j}
=
q^{-m}v_{q,j,N}+\varepsilon=
q^{N-m}b_{q,j,N}+\varepsilon.
\end{align*}
Now
$$
q^{N-m}b_{q,j,N}\in B\leq B_q
\quad\textnormal{
and
}\quad
\varepsilon\in L\otimes_{\Z}\Z_q\leq B_q.
$$
Therefore
$q^{-m}v_{q,j}\in B_q$, as claimed.

Let now $\lambda\in\Q_q$.  There exists $m\geq1$ such that
$q^m\lambda\in\Z_q$. Since $B_q=B\otimes_{\Z}\Z_q$ is a $\Z_q$-module,
$$
\lambda v_{q,j}
=
(q^m\lambda)(q^{-m}v_{q,j})
\in B_q.
$$
Thus
$\Q_qv_{q,j}\leq B_q$
for every $j$, and therefore $D_q\leq B_q$.

\smallskip

We now prove the converse inclusion
$B_q\leq D_q$. We first show that
$B\leq D_q$. Clearly,
$L\leq L\otimes_{\Z}\Z_q\leq D_q$. Fix $n\geq1$ and take $b_{q,j,n}$. Since $\bigl(v_{q,j,n}\bigr)_n$ converges to $v_{q,j}$, there exists $N\geq n$ such that
$$
v_{q,j}-v_{q,j,N}
\in
q^n(L\otimes_{\Z}\Z_q).
$$
Hence
$$
q^{-n}(v_{q,j}-v_{q,j,N})
\in
L\otimes_{\Z}\Z_q.
$$

On the other hand, the relations
$q(b_{q,j,s+1}+L)=b_{q,j,s}+L$
for $s\geq1$ imply, by iteration, that
$q^{N-n}b_{q,j,N}-b_{q,j,n}\in L$. Since
$v_{q,j,N}=q^N b_{q,j,N}$, we have
\begin{align*}
q^{-n}v_{q,j}
&=
q^{-n}v_{q,j,N}
+
q^{-n}(v_{q,j}-v_{q,j,N})\\
&=
q^{N-n}b_{q,j,N}
+
q^{-n}(v_{q,j}-v_{q,j,N}).
\end{align*}
Therefore
$$
b_{q,j,n}
=
q^{-n}v_{q,j}
-
\bigl(q^{N-n}b_{q,j,N}-b_{q,j,n}\bigr)
-
q^{-n}(v_{q,j}-v_{q,j,N}).
$$
The first term belongs to $\Q_qv_{q,j}$, while the second belongs to
$L$ and the third belongs to $L\otimes_{\Z}\Z_q$.  Hence
$b_{q,j,n}\in D_q$.

Now let $p\in\spec(B)$ with $p\neq q$.  By construction,
$p^n b_{p,j,n}\in L$. Since $p\neq q$, the integer $p$ is a unit in $\Z_q$, so $p^{-n}\in\Z_q$.
Consequently
$$
b_{p,j,n}
=
p^{-n}(p^n b_{p,j,n})
\in
L\otimes_{\Z}\Z_q
\leq D_q.
$$ Thus $L$ and all the chosen elements belong to $D_q$. Since $L$ and these elements generate~$B$, we obtain $B\leq D_q$.

Moreover, $D_q$ is a $\Z_q$-submodule of $V_q$: the group
$L\otimes_{\Z}\Z_q$ is a $\Z_q$-module, and every
$\Q_qv_{q,j}$ is stable under multiplication by elements of $\Z_q$.
Hence the inclusion $B\leq D_q$ induces a homomorphism
$$
B\otimes_{\Z}\Z_q\longrightarrow D_q
$$
whose composition with the inclusion $D_q\leq V_q$ is the canonical
embedding
$$
B\otimes_{\Z}\Z_q\hookrightarrow V_q.
$$
It follows that the image of $B\otimes_{\Z}\Z_q$ in $V_q$ is contained
in $D_q$.  Under our identification
$B_q=B\otimes_{\Z}\Z_q$, this gives
$B_q\leq D_q$.

Combining the two inclusions, we obtain
$$
B_q
=
(L\otimes_{\Z}\Z_q)
+
\sum_{j=1}^{r_q}\Q_qv_{q,j}.
$$


Finally, we show that $B_q=\overline{B}^{\,V_q}$. As observed above, the rational completion of the torsion-free
abelian group $B$ is
$B_{\Q}=B\otimes_{\Z}\Q$. Moreover,
$$
V_q=
(B\otimes_{\Z}\Q)\otimes_{\Q}\Q_q
$$
is the corresponding $q$-adic Mal'cev completion.   The third classical case listed at the beginning of
\cite[Section~5.1]{KL20}, immediately before Lemma~5.1, states that,
for a torsion-free nilpotent group $N$, its tensor $q$-completion
$N_q$ is naturally isomorphic to the closure of $N$ in its $q$-adic
Mal'cev completion.  Applying this statement to $N=B$ gives a natural
isomorphism
$B_q\simeq\overline{B}^{\,V_q}$.

On the other hand, earlier in the proof we used the natural
isomorphism
$B_q\simeq B\otimes_{\Z}\Z_q,
$ together with the
canonical embedding
$B\otimes_{\Z}\Z_q\hookrightarrow V_q$, to realize $B_q$ as the subgroup $D_q$ of $V_q$. The realization of $B_q$ inside $V_q$ constructed above and the
realization of $B_q$ as $\overline{B}^{\,V_q}$ both extend the same
canonical embedding
$$
B\longrightarrow V_q,
\qquad
b\longmapsto (b\otimes1)\otimes1.
$$
By the uniqueness in the universal property of the tensor $q$-completion \cite[Proposition~5.3]{KL20}, the two embeddings of $B_q$ into $V_q$ coincide. Hence the subgroup of $V_q$ obtained above is precisely $\overline{B}^{\,V_q}$.  With this identification we may therefore write $B_q=\overline{B}^{\,V_q}$.
\end{proof}

Now, let $Y$ be a torsion-free nilpotent minimax group, and let $\pi$ be an epimorphism from $Y$ to $B$. Let
$\pi_q:Y_q\longrightarrow B_q$ be the homomorphism induced by $\pi$ on the tensor $q$-completions. Choose a basis
$\ell_1,\ldots,\ell_d$ of $L$.  Choose lifts
$\widehat\ell_k\in Y$, $1\leq k\leq d$ and
$\widehat b_{q,j,n}\in Y$, $1\leq j\leq r_q$, $n\geq1$
such that
$\pi(\widehat\ell_k)=\ell_k$ and $\pi(\widehat b_{q,j,n})=b_{q,j,n}$. Fix $z_k\in\Z_q$ and $\lambda_j\in\Q_q$ and let
$$
b=
\sum_{k=1}^d z_k\ell_k+
\sum_{j=1}^{r_q}\lambda_jv_{q,j},
$$
which is an element of $B_q$ by Lemma \ref{lem:abelian-q-completion}.
Moreover, for each $j=1,\ldots,r_q$, write uniquely
$$
v_{q,j}
=
\sum_{k=1}^d a_{j,k}(\ell_k\otimes1),
\qquad
a_{j,k}\in\Z_q,
$$
with respect to the $\Z_q$-basis
$\ell_1\otimes1,\ldots,\ell_d\otimes1$ of $L\otimes_{\Z}\Z_q$. Apply Lemma~\ref{lem:coordinates} to $Y$, and fix the resulting
elements $y_1,\ldots,y_h$ and coordinate polynomials.

Next lemma will allow us to lift $b$ effectively back to $Y$.

\begin{lemma}\label{lem:factorlift}
Suppose there exists a finitely generated subfield \hbox{$F\leq\Q_q$} such that\textnormal:

\begin{itemize}
    \item $a_{j,k}, z_k,\lambda_j\in F$ are given by effective representations in $F$\textnormal;
    
    \item from the effective representations of elements of $F$, one can compute effectively their sums, differences, products and quotients whenever the denominator is non-zero, and decide equality in $F$\textnormal;

    \item from the effective representation of any $\alpha\in F$, one can  effectively compute its $q$-adic valuation $v_q(\alpha)\in\Z\cup\{\infty\}$\textnormal;

    \item for every $j$ and $N$, the integers $c_{j,k,N}$ determined by
    $$
    q^N b_{q,j,N}
    =
    \sum_{k=1}^d c_{j,k,N}\ell_k
    $$
    can be computed effectively from $j$ and $N$\textnormal;

    \item for every $j$ and $N$, the coordinates, with respect to $y_1,\ldots,y_h$, of a lift $\widehat b_{q,j,N}\in Y$ of $b_{q,j,N}$ can be computed effectively from $j$ and $N$, and the coordinates of the lifts $\widehat\ell_1,\ldots,\widehat\ell_d\in Y$ are given effectively.
\end{itemize}
Then a lift $y\in Y_q$ of $b$ can be constructed effectively from the given data using products, inverses and $\Z_q$-powers in $Y_q$.
\end{lemma}

\begin{proof}
We first lift the summand of $b$ which belongs to $L\otimes_{\Z}\Z_q$. For \hbox{$z=(z_1,\ldots,z_d)\in\Z_q^d$} define
$\sigma_L(z):=\widehat\ell_1^{\,z_1}\ldots\widehat\ell_d^{\,z_d}\in Y_q$. Since $\pi_q(\widehat\ell_k)=\ell_k$ and homomorphisms between the $q$-completions commute with $\Z_q$-powers, we have $\pi_q(\widehat\ell_k^{\,z_k})=z_k\ell_k$. Consequently
$\pi_q(\sigma_L(z))=\sum_{k=1}^d z_k\ell_k$.

We now lift the terms $\lambda_jv_{q,j}$. Fix $j\in\{1,\ldots,r_q\}$. If $\lambda_j\neq0$, put $s_j:=\max\{0,-v_q(\lambda_j)\}$, while if $\lambda_j=0$, put $s_j:=0$. Then $\mu_j:=q^{s_j}\lambda_j\in\Z_q$. Put $N_j:=\max\{1,s_j\}$. By Lemma~\ref{lem:abelianpadic},
$$
v_{q,j}-v_{q,j,N_j}
\in
q^{N_j}(L\otimes_{\Z}\Z_q)
\leq
q^{s_j}(L\otimes_{\Z}\Z_q).
$$
Hence there exists $\varepsilon_j\in L\otimes_{\Z}\Z_q$ such that
$v_{q,j}-v_{q,j,N_j}=q^{s_j}\varepsilon_j$. After identifying $\ell_k$ with $\ell_k\otimes1$, write uniquely
$$
\varepsilon_j
=
\sum_{k=1}^d e_{j,k}\ell_k,
\qquad
e_{j,k}\in\Z_q.
$$
Regard all these elements inside the ambient $\Q_q$-vector space $V_q$. After multiplication by $q^{-s_j}$ in $V_q$, we have $q^{-s_j}v_{q,j}=q^{-s_j}v_{q,j,N_j}+\varepsilon_j$. Since $v_{q,j,N_j}=q^{N_j}b_{q,j,N_j}$
and $N_j\geq s_j$, this becomes
$$
q^{-s_j}v_{q,j}
=
q^{N_j-s_j}b_{q,j,N_j}
+
\sum_{k=1}^d e_{j,k}\ell_k.
$$

Define
$$
w_j
:=
\widehat b_{q,j,N_j}^{\,q^{N_j-s_j}}\,
\sigma_L(e_{j,1},\ldots,e_{j,d})
\in Y_q.
$$
Applying $\pi_q$ gives
$$
\pi_q(w_j)
=
q^{N_j-s_j}b_{q,j,N_j}
+
\sum_{k=1}^d e_{j,k}\ell_k=
q^{-s_j}v_{q,j}.
$$
Since $\mu_j=q^{s_j}\lambda_j\in\Z_q$, the $\Z_q$-power $w_j^{\mu_j}$ is defined in $Y_q$. Using the compatibility of $\pi_q$ with $\Z_q$-powers, we obtain
\begin{align*}
\pi_q(w_j^{\mu_j})
&=
\mu_j\,\pi_q(w_j)\\
&=
(q^{s_j}\lambda_j)(q^{-s_j}v_{q,j})\\
&=
\lambda_jv_{q,j}.
\end{align*}
Thus $w_j^{\mu_j}$ is a lift in $Y_q$ of the summand $\lambda_jv_{q,j}$.

Now, if we put
$$
y
:=
\sigma_L(z_1,\ldots,z_d)
\prod_{j=1}^{r_q}w_j^{\mu_j}
\in Y_q,
$$
we immediately have that $\pi_q(y)=b= \sum_{k=1}^d z_k\ell_k+
\sum_{j=1}^{r_q}\lambda_jv_{q,j}$,
as required.

\smallskip

We now prove the effective assertion. Since equality and $q$-adic valuation in $F$ are computable, the integers $s_j$ and $N_j$ chosen above are effectively computable from $\lambda_j$. By Lemma~\ref{lem:abelianpadic},
$$
v_{q,j}-v_{q,j,N_j}
\in
q^{N_j}(L\otimes_{\Z}\Z_q)
\leq
q^{s_j}(L\otimes_{\Z}\Z_q).
$$
Write
$$
v_{q,j,N_j}
=
\sum_{k=1}^d c_{j,k,N_j}\ell_k,
\qquad
c_{j,k,N_j}\in\Z.
$$
By hypothesis, the integers $c_{j,k,N_j}$ can be computed effectively. The coefficients of
$\varepsilon_j=q^{-s_j}(v_{q,j}-v_{q,j,N_j})$
are
$e_{j,k}=
q^{-s_j}(a_{j,k}-c_{j,k,N_j})$.
These elements can be computed exactly in $F$. Since
$$
v_{q,j}-v_{q,j,N_j}
\in
q^{s_j}(L\otimes_{\Z}\Z_q),
$$
we have $e_{j,k}\in\Z_q$ for every $k$. Likewise, $\mu_j=q^{s_j}\lambda_j$ is computable and belongs to $\Z_q$.

Therefore all the exponents $z_k,e_{j,k},\mu_j$ are effectively computable. By hypothesis, the coordinates of $\widehat\ell_1,\ldots,\widehat\ell_d$ are given effectively, and, for every~$j$, the coordinates of $\widehat b_{q,j,N_j}$ can be computed effectively. Since multiplication and $\Z_q$-powers in $Y_q$, with respect to the fixed coordinate system $y_1,\ldots,y_h$, are given by known polynomials, the coordinates of $\sigma_L(z_1,\ldots,z_d),w_j,w_j^{\mu_j}$ can therefore be computed effectively. Hence the element $y$ above can be constructed effectively from the given data using products, inverses and $\Z_q$-powers in $Y_q$.
\end{proof}

We next record the exact arithmetic used in the $p$-adic linear calculations.

\begin{lemma}\label{lem:field}
Fix a prime $p$ and elements $\alpha_1,\ldots,\alpha_s\in\Qp$. Put $F=\Q(\alpha_1,\ldots,\alpha_s)$.
Assume that the following exact data and approximation procedure are given\textnormal:
\begin{enumerate}[label=\textup{(\roman*)}]
\item a subset $\beta_1,\ldots,\beta_t$ of the $\alpha_i$ which is a transcendence basis of $F/\Q$\textnormal;
\item a basis $e_1,\ldots,e_d$ of the finite extension
$F/\Q(\beta_1,\ldots,\beta_t)$, with each $e_k$ given as a monomial in the $\alpha_i$, together with the multiplication table of this basis and the coordinates of the $\alpha_i$ in it\textnormal;
\item an algorithm which, on input $(i,n)$, returns an element $a_{i,n}$ of $\Q$ satisfying
$v_p(\alpha_i-a_{i,n})\geq n$.
\end{enumerate}
Then equality in $F$ is decidable and the field operations in $F$ are computable. Moreover, for every $a\in F$, the value $v_p(a)\in\Z\cup\{\infty\}$ is computable, and for every $a\in F$ and $n\geq1$ one can effectively compute $a_n\in\Q$ such that
\begin{equation}
v_p(a-a_n)\geq n.
\label{eq:2.12}
\end{equation}
\end{lemma}

\begin{proof}
Put $K=\Q(\beta_1,\ldots,\beta_t)$. Represent an element of $K$ by a quotient of two polynomials in $\beta_1,\ldots,\beta_t$. Two such quotients are equal if and only if the corresponding cross-products are the same polynomial, so equality in $K$ is decidable and the four field operations in $K$ are computable. Represent an element of $F$ by its coordinate vector in the basis $e_1,\ldots,e_d$. Addition is coordinatewise and multiplication is obtained from the fixed multiplication table. If $0\neq x\in F$, the coordinates of $x^{-1}$ are the unique solution of the $d\times d$ linear system expressing $xy=1$ and Gaussian elimination over $K$ effectively computes this solution. Hence equality in $F$ is decidable and all field operations in $F$ are computable.

We shall use repeatedly the following elementary observation. Suppose that $0\neq x\in\Qp$ and that rational numbers $x_n$ can be computed effectively with $v_p(x-x_n)\geq n$ for every $n\geq1$. Then $v_p(x)$ is computable. Indeed, compute $x_1,x_2,\ldots$ successively and stop at the first index $m$ such that $v_p(x_m)<m$. This search terminates since, if $r=v_p(x)$, then for every $m>r$ we have $v_p(x-x_m)\geq m>r$, whence $v_p(x_m)=r<m$. For the first such $m$ we have $v_p(x-x_m)\geq m>v_p(x_m)$, and hence, since  $x=x_m+(x-x_m)$, we get $v_p(x)=v_p(x_m)$.

We first show that $v_p(\beta_j)$ can be computed effectively. Write $\beta_j=\alpha_{i(j)}$. By \textup{(iii)}, for every $n\geq1$ we can compute $b_n:=a_{i(j),n}\in\Q$ such that
$$
v_p(\beta_j-b_n)\geq n.
$$
Since $\beta_j\neq0$, the preceding observation shows that $v_p(\beta_j)$ is computable.

More generally, equality in $F$ allows us to determine which of the $\alpha_i$ are zero, while \textup{(iii)} and the preceding observation compute $v_p(\alpha_i)$ for every non-zero $\alpha_i$.

We next observe that rational approximations can be propagated through the field operations needed below. Let $u,v\neq0$ and suppose that $v_p(u)$ and $v_p(v)$ are known. Given a prescribed precision $n\geq1$, choose integers $M_u$ and $M_v$ such that
$$
M_u>v_p(u),\qquad M_v>v_p(v),
$$
and
$$
M_u+v_p(v)\geq n,\qquad v_p(u)+M_v\geq n.
$$
If $u',v'\in\Q$ satisfy
$$
v_p(u-u')\geq M_u,\qquad v_p(v-v')\geq M_v,
$$
then $v_p(u')=v_p(u)$ and $v_p(v')=v_p(v)$. Moreover,
$$
uv-u'v'=(u-u')v+u'(v-v'),
$$
and therefore
$$
v_p(uv-u'v')\geq\min\{M_u+v_p(v),\,v_p(u)+M_v\}\geq n.
$$
Thus sufficiently accurate rational approximations to $u$ and $v$ effectively yield a rational approximation to their product with any prescribed $p$-adic precision. Addition is handled similarly, since
$$
(u+v)-(u'+v')=(u-u')+(v-v').
$$
Starting from the approximations supplied by \textup{(iii)} for the $\alpha_i$, repeated use of these observations therefore gives, effectively, rational approximations to every fixed polynomial, and in particular to every fixed monomial, in the $\alpha_i$, to arbitrary prescribed $p$-adic precision.

Now, consider
a non-zero polynomial value
$$
c=r(\beta_1,\ldots,\beta_t),
\qquad
0\neq r\in\Q[X_1,\ldots,X_t].
$$
By the preceding paragraph, one computes rational numbers $c_n$ such that $v_p(c-c_n)\geq n$. The preceding observation therefore shows that $v_p(c)$ is computable. Thus the valuation of every non-zero polynomial value in $\beta_1,\ldots,\beta_t$ is computable.

Every element of $K$ has the form
$$
a=
\frac{f(\beta_1,\ldots,\beta_t)}
{g(\beta_1,\ldots,\beta_t)},
\qquad
f,g\in\Q[X_1,\ldots,X_t],
\quad
g(\beta_1,\ldots,\beta_t)\neq0.
$$
Applying the preceding paragraph to $g(\beta_1,\ldots,\beta_t)$ gives the integer
$$
\sigma=v_p(g(\beta_1,\ldots,\beta_t)).
$$
To approximate its inverse, choose an integer $N>\sigma$ and a rational number $g_N$ such that
$$
v_p(g(\beta_1,\ldots,\beta_t)-g_N)\geq N.
$$
Then $v_p(g_N)=\sigma$. Moreover,
$$
g(\beta_1,\ldots,\beta_t)^{-1}-g_N^{-1}
=
\frac{g_N-g(\beta_1,\ldots,\beta_t)}
{g(\beta_1,\ldots,\beta_t)g_N},
$$
and therefore
$$
\begin{aligned}
v_p\left(g(\beta_1,\ldots,\beta_t)^{-1}-g_N^{-1}\right)
&=
v_p\bigl(g_N-g(\beta_1,\ldots,\beta_t)\bigr)
-v_p\bigl(g(\beta_1,\ldots,\beta_t)\bigr)
-v_p(g_N)\\
&\geq N-2\sigma.
\end{aligned}
$$
Hence, given $n\geq1$, by choosing for instance $N\geq n+2|\sigma|+1$, we obtain effectively a rational number $t_n$ such that
$$
v_p\left(g(\beta_1,\ldots,\beta_t)^{-1}-t_n\right)\geq n.
$$
Thus the inverse of every non-zero polynomial value in $\beta_1,\ldots,\beta_t$ can be approximated rationally to arbitrary prescribed $p$-adic precision.

If $f=0$, then $a=0$, and we may take $a_n=0$ for every $n$. Suppose therefore that $f\neq0$, and put
$$
\rho=v_p(f(\beta_1,\ldots,\beta_t)),
\qquad
\sigma=v_p(g(\beta_1,\ldots,\beta_t)).
$$
By the preceding paragraphs, both $\rho$ and $\sigma$ are computable. Moreover,
$$
v_p\bigl(g(\beta_1,\ldots,\beta_t)^{-1}\bigr)=-\sigma.
$$
Given a prescribed precision $n\geq1$, choose integers $M_f$ and $M_g$ such that
$$
M_f>\rho,
\qquad
M_g>-\sigma,
$$
and
$$
M_f-\sigma\geq n,
\qquad
\rho+M_g\geq n.
$$
These integers can be chosen effectively from $n,\rho$ and $\sigma$. For example, one may take
$$
M_f=\max\{1,\rho+1,n+\sigma\},
\qquad
M_g=\max\{1,-\sigma+1,n-\rho\}.
$$
Using the approximation procedures above, compute rational numbers $f'$ and $t'$ such that
$$
v_p\bigl(f(\beta_1,\ldots,\beta_t)-f'\bigr)\geq M_f,
\qquad
v_p\left(g(\beta_1,\ldots,\beta_t)^{-1}-t'\right)\geq M_g.
$$
Since $M_f>\rho$ and $M_g>-\sigma$, we have
$$
v_p(f')=\rho,
\qquad
v_p(t')=-\sigma.
$$
Applying the product estimate to $f(\beta_1,\ldots,\beta_t)$ and $g(\beta_1,\ldots,\beta_t)^{-1}$ gives
$$
v_p(a-f't')
\geq
\min\{M_f-\sigma,\rho+M_g\}
\geq n.
$$
Thus, setting $a_n=f't'$, we obtain effectively a rational number $a_n$ such that
$$
v_p(a-a_n)\geq n.
$$
Hence every element of $K$ admits effectively computable rational approximations to arbitrary prescribed $p$-adic precision.

If $a\neq0$, the preceding observation applied to the rational approximations just constructed shows that $v_p(a)$ is computable.

Finally, write an arbitrary element of $F$ uniquely as
$$
x=\sum_{k=1}^d a_k e_k,
\qquad
a_k\in K.
$$
By \textup{(ii)}, each $e_k$ is a monomial in the $\alpha_i$. Consequently, from \textup{(iii)} and the effective product estimate above, rational approximations to each $e_k$ can be computed to arbitrary prescribed $p$-adic precision.  Moreover, since the valuations of the non-zero $\alpha_i$ have already been computed, the valuation of each $e_k$ is computable directly from its monomial expression.

Equality in $K$ decides whether each coefficient $a_k$ is zero.  If $a_k\neq0$, the preceding paragraph computes both $v_p(a_k)$ and rational approximations to $a_k$ of arbitrary prescribed precision. Thus, for every non-zero term $a_ke_k$ and every prescribed $n$, we can choose effectively precisions $M_k$ and $N_k$ for approximations $a_k'$ and $e_k'$ such that
$$
v_p(a_k-a_k')\geq M_k,
\qquad
v_p(e_k-e_k')\geq N_k,
$$
with
$$
M_k>v_p(a_k),
\qquad
N_k>v_p(e_k),
$$
and
$$
M_k+v_p(e_k)\geq n,
\qquad
v_p(a_k)+N_k\geq n.
$$
The product estimate then gives
$$
v_p(a_ke_k-a_k'e_k')\geq n.
$$
For a zero coefficient $a_k$ we simply take the corresponding approximation to be zero.

Since there are only finitely many terms, summing the rational approximations $a_k'e_k'$ and using the $p$-adic triangle inequality gives a rational number~$x_n$ satisfying $v_p(x-x_n)\geq n$. Thus rational approximations to every element of~$F$ can be computed effectively to arbitrary prescribed $p$-adic precision.

Finally, equality in $F$ decides whether $x=0$. If $x=0$, then $v_p(x)=\infty$, while if $x\neq0$, the preceding observation applied to the rational approximations just constructed computes $v_p(x)$. Hence $v_p(x)$ is computable for every $x\in F$.
\end{proof}

With this exact arithmetic available, we can now make effective the linear-algebraic operations on the mixed $\Qp$--$\Zp$ subgroups that occur in tensor $p$-completions.

\begin{lemma}\label{lem:additive}
Let $F\leq\Q_p$ be a finitely generated subfield containing $\Q$. Assume that elements of $F$ are given in an exact finite presentation for which equality is decidable, the field operations are computable, and the $p$-adic valuation $v_p:F\to\Z\cup\{\infty\}$ is computable. Put $\mathcal O=F\cap\Z_p$ and $V=\Q_p^d$.

We consider subgroups $B\leq V$ given by finite data of the form
\begin{equation}
B=D+L,
\qquad
D=\operatorname{span}_{\Q_p}\{d_1,\ldots,d_r\},
\qquad
L=\sum_{j=1}^s\Z_p\ell_j,
\label{eq:2.13}
\end{equation}
where $d_i,\ell_j\in F^d$, and each $d_i$ and $\ell_j$ is given by its coordinates with respect to the standard basis of $F^d$, with each coordinate given in the prescribed exact presentation of $F$.

Then there are terminating algorithms for the following operations\textnormal:
\begin{enumerate}[label=\textup{(\roman*)}]
\item given $x\in F^d$ by its coordinates with respect to the given basis, deciding whether $x\in B$\textnormal; when membership holds, the algorithm computes coefficients $a_1,\ldots,a_r\in F$ and $z_1,\ldots,z_s\in\mathcal O$ such that
$$
x=\sum_{i=1}^r a_i d_i+\sum_{j=1}^s z_j\ell_j\textnormal;
$$

\item sums and intersections of subgroups of $V$ given in the form
\eqref{eq:2.13}\textnormal;

\item given finite-dimensional $\Q_p$-vector spaces
$V_1=\Q_p^{d_1}$ and $V_2=\Q_p^{d_2}$, an $F$-linear map
$T:V_1\to V_2$ whose matrix with respect to the standard bases is
given exactly over $F$, and subgroups $B\leq V_1$ and $C\leq V_2$
given in the form \eqref{eq:2.13}, computing
$T(B)$, $B\cap\Ker T$, and
$\{x\in B:Tx\in C\}$\textnormal;

\item the corresponding operations in quotients:
given $V'=\Q_p^{d'}$, subgroups
$B_0\leq B_1\leq V$ and $C_0\leq C_1\leq V'$ in the form
\eqref{eq:2.13}, and an $F$-linear map
$T:V\to V'$ given exactly over $F$ with
$T(B_1)\leq C_1$ and $T(B_0)\leq C_0$, one can compute sums and
intersections of subgroups in the quotients, and images, kernels and
inverse images under the induced map.  Equality of cosets represented
by elements of $F^d$ is decidable as well.
\end{enumerate}
Moreover, every subgroup obtained in \textup{(ii)} and \textup{(iii)} can again
be computed in the form~\eqref{eq:2.13}, with all its generators given
by their coordinates with respect to the given
basis and with all coordinates represented exactly in $F$.
\end{lemma}

\begin{proof}
By its definition, we have $\mathcal O=\{a\in F:v_p(a)\geq0\}$. For an element $a\in F$ given in the prescribed exact representation, membership in $\mathcal O$ is therefore decidable. Moreover, if $a,b\in\mathcal O$ and $a\neq0$, then
$$
a\mid b\text{ in }\mathcal O
\iff
b/a\in\mathcal O
\iff
v_p(a)\leq v_p(b).
$$
Hence divisibility in $\mathcal O$ is decidable and, when it holds,
the quotient $b/a\in\mathcal O$ is computable using the field
operations in $F$.  Also
$a\in\mathcal O^\times$ if and only if $v_p(a)=0$, so membership in
$\mathcal O^\times$ is decidable.

We first prove \textup{(i)}.  Put
$$
D_F=\operatorname{span}_F\{d_1,\ldots,d_r\}\leq F^d.
$$
Gaussian elimination over $F$ effectively computes a basis
$u_1,\ldots,u_k$ of $D_F$ and extends it to a basis
$u_1,\ldots,u_d$ of $F^d$.  Put $e=d-k$ and define
$q:F^d\to F^e$ by
$$
q\left(\sum_{i=1}^d a_i u_i\right)
=
(a_{k+1},\ldots,a_d).
$$
Since equality and the field operations in $F$ are computable, the
matrix of $q$ with respect to the standard bases can be computed
effectively, and by construction $\Ker(q)=D_F$.
The same matrix defines a $\Q_p$-linear map
$q:\Q_p^d\to\Q_p^e$.  Since its kernel is obtained from
$\Ker(q)$ by extension of scalars,
$$
\Ker(q:\Q_p^d\to\Q_p^e)
=
\operatorname{span}_{\Q_p}\{d_1,\ldots,d_r\}
=
D.
$$
In particular, $D\cap F^d=D_F$.  For $x\in F^d$ we therefore have
\begin{equation}
x\in B=D+L
\iff
q(x)\in\sum_{j=1}^s\Z_p q(\ell_j).
\label{eq:2.15}
\end{equation}

Set
$$
M_0=\sum_{j=1}^s\mathcal O q(\ell_j)\leq F^e.
$$ Let $A$ be the $e\times s$ matrix whose columns are
$q(\ell_1),\ldots,q(\ell_s)$.  Thus
$M_0=\{Az:z\in\mathcal O^s\}$.
The vectors $q(\ell_j)$ have coordinates in $F$, but these coordinates need not belong to $\mathcal O$.  Since $v_p$ is computable, we can effectively choose an integer $c\geq0$ such that every coordinate of every vector $p^cq(\ell_j)$ has non-negative valuation, and hence lies in $\mathcal O$.  Thus every entry of $p^cA$ lies in $\mathcal O$, and we may apply Smith reduction over $\mathcal O$.

This reduction is effective.  At each step, choose a non-zero matrix entry having minimal $p$-adic valuation and move it to the upper-left corner.  If $a$ is this pivot and $b$ is any other entry, then $v_p(a)\leq v_p(b)$, and therefore $a$ divides $b$ in $\mathcal O$. Thus elementary row and column operations over $\mathcal O$ can be used to clear the remaining entries in the pivot row and pivot column. Repeating the same procedure on the smaller matrix which remains
eventually produces a diagonal matrix.

Smith reduction therefore gives matrices $U\in\operatorname{GL}_e(\mathcal O)$ and $V\in\operatorname{GL}_s(\mathcal O)$ such that
$$
U(p^cA)V
=
\operatorname{diag}(\delta_1,\ldots,\delta_t,0,\ldots,0),
$$
where $\delta_i\in\mathcal O\setminus\{0\}$.  The matrix $V$ performs an invertible change of generators among the columns, and therefore does not change the $\mathcal O$-module generated by them.  On the other hand, multiplication by $U$ applies the same invertible $\mathcal O$-linear transformation to the ambient space $F^e$.

Let $\varepsilon_1,\ldots,\varepsilon_e$ denote the standard basis of $F^e$. For each $i$, put $n_i=v_p(\delta_i)$. Since $\delta_i\mathcal O=p^{n_i}\mathcal O$, the columns of $U(p^cA)V$ generate
$$
U(p^cM_0)
=
p^{n_1}\mathcal O\varepsilon_1
\oplus\cdots\oplus
p^{n_t}\mathcal O\varepsilon_t.
$$
Applying $U^{-1}$ and setting $f_i=U^{-1}\varepsilon_i$, we obtain
$$
p^cM_0
=
p^{n_1}\mathcal O f_1
\oplus\ldots\oplus
p^{n_t}\mathcal O f_t.
$$
Since $U$ is invertible, the vectors $f_1,\ldots,f_t$ are $F$-linearly independent. Multiplying by $p^{-c}$ and putting $e_i=n_i-c$, we obtain
\begin{equation}
M_0
=
p^{e_1}\mathcal O f_1
\oplus\cdots\oplus
p^{e_t}\mathcal O f_t,
\qquad
e_1,\ldots,e_t\in\Z.
\label{eq:2.17}
\end{equation}
The row and column operations, and hence the vectors $f_i$ and the
change-of-generator matrices, are computed effectively.

Extend $f_1,\ldots,f_t$ effectively, by Gaussian elimination over $F$, to an~\hbox{$F$-basis} $f_1,\ldots,f_e$ of $F^e$. Since $\Z_{(p)}\subseteq\mathcal O\subseteq\Z_p$ and $\Z_{(p)}$ is dense in $\Z_p$, the closure of~$\mathcal O$ in $\Z_p$ is $\Z_p$. Taking $p$-adic closures in
$$
M_0
=
p^{e_1}\mathcal O f_1
\oplus\ldots\oplus
p^{e_t}\mathcal O f_t
$$
and using $M_0=\sum_{j=1}^s\mathcal O q(\ell_j)$ gives
$$
\sum_{j=1}^s\Z_pq(\ell_j)
=
p^{e_1}\Z_pf_1
\oplus\ldots\oplus
p^{e_t}\Z_pf_t.
$$

The standard coordinates of $q(x)$ are computable from the known matrix of $q$ and the standard coordinates of $x$. Since the basis $f_1,\ldots,f_e$ is also known effectively, Gaussian elimination over $F$ computes the unique coefficients
$$
q(x)=\sum_{i=1}^e\xi_if_i,
\qquad
\xi_i\in F.
$$
Then \eqref{eq:2.15} is equivalent to $\xi_i=0$ for $i>t$ and $v_p(\xi_i)\geq e_i$ for $1\leq i\leq t$.  These are finitely many effective tests, so membership in $B$ is decidable.

Suppose that the above membership conditions hold.  Then $v_p(\xi_i)\geq e_i$ for $1\leq i\leq t$, so $\eta_i:=p^{-e_i}\xi_i\in\mathcal O$.  Since $\xi_i=0$ for $i>t$, we have
$$
q(x)=\sum_{i=1}^t p^{e_i}\eta_i f_i\in M_0.
$$

We now recover coefficients with respect to the original generators $q(\ell_1),\ldots,q(\ell_s)$. Recall that
$$
U(p^cA)V=\Delta=\operatorname{diag}(p^{n_1},\ldots,p^{n_t},0,\ldots,0).
$$
Since $q(x)\in M_0$, the vector
$$
\gamma:=p^cUq(x)=(\gamma_1,\ldots,\gamma_e)^T
$$
belongs to $U(p^cM_0)$. Hence $\gamma_i\in p^{n_i}\mathcal O$ for $1\leq i\leq t$ and $\gamma_i=0$ for $i>t$. Define $w\in\mathcal O^s$ by
$$
w_i=p^{-n_i}\gamma_i\quad(1\leq i\leq t),
\qquad
w_i=0\quad(i>t),
$$
and put $z=Vw$. Then $z\in\mathcal O^s$ is computable and
$$
p^cUAz
=
U(p^cA)Vw
=
\Delta w
=
\gamma
=
p^cUq(x).
$$
Since $p^cU$ is invertible over $F$, it follows that
$$
Az=q(x),
$$
that is,
$$
q(x)=\sum_{j=1}^s z_jq(\ell_j).
$$

By linearity of $q$,
$q\left(x-\sum_{j=1}^s z_j\ell_j\right)=0$.  Hence
$$
x-\sum_{j=1}^s z_j\ell_j
\in
\Ker q
=
D_F
=
\operatorname{span}_F\{d_1,\ldots,d_r\}.
$$
All terms on the left are known exactly in $F^d$, so Gaussian elimination over $F$ computes coefficients $a_1,\ldots,a_r\in F$ such that $x-\sum_{j=1}^s z_j\ell_j=\sum_{i=1}^r a_id_i$.  Therefore
$$
x
=
\sum_{i=1}^r a_i d_i
+
\sum_{j=1}^s z_j\ell_j,
\qquad
a_i\in F,\quad z_j\in\mathcal O.
$$
This proves \textup{(i)}.

Sums and images are immediate.  If
$$
B=D+L,
\qquad
B'=D'+L',
$$
then
$$
B+B'=(D+D')+(L+L'),
$$
so one obtains a presentation of the required form by adjoining the corresponding generators. Likewise, if $T$ is an $F$-linear map, then
$$
T(B)
=
\operatorname{span}_{\Q_p}\{Td_1,\ldots,Td_r\}
+
\sum_{j=1}^s\Z_pT\ell_j,
$$
and all resulting vectors have exact coordinates in $F$.

We next treat intersections, kernels and inverse images. Each of these problems reduces to a finite system with some variables ranging freely over~$\Q_p$ and the remaining variables constrained to lie in $\Z_p$. For example, if
$$
B=D+L,
\qquad
B'=D'+L',
$$
then $B\cap B'$ is described by solutions of
$$
\sum_i u_id_i+\sum_j z_j\ell_j
=
\sum_i u_i'd_i'+\sum_j z_j'\ell_j',
$$
where $u_i,u_i'\in\Q_p$ and $z_j,z_j'\in\Z_p$. Similarly, conditions defining kernels and inverse images have the same form after introducing the appropriate auxiliary variables.

Thus, after collecting all variables ranging over $\Q_p$ in $u$ and all variables constrained to lie in $\Z_p$ in $z$, we obtain a system
\begin{equation}
Au+Cz=0,
\qquad
u\in\Q_p^r,\quad z\in\Z_p^s,
\label{eq:mixedsystem}
\end{equation}
where, for some effectively known integer $N$,
$$
A\in M_{N\times r}(F),
\qquad
C\in M_{N\times s}(F).
$$

For a fixed $z\in\Z_p^s$, the equation $Au+Cz=0$ is solvable in $u\in\Q_p^r$ if and only if $-Cz$ belongs to the $\Q_p$-span of the columns of $A$. Gaussian elimination over $F$ effectively computes an $F$-linear map $q_A:F^N\to F^h$ whose kernel is the $F$-span of these columns. After extension of scalars to $\Q_p$, the kernel of the same map is their $\Q_p$-span.  Hence
$$
Au+Cz=0\text{ is solvable in }u
\iff
q_A(Cz)=0.
$$
Thus, putting $C'=q_AC$, the original system reduces to
$$
C'z=0,
\qquad
z\in\Z_p^s,
$$
where $C'$ has entries in $F$.

Choose $c\geq0$ such that every entry of $p^cC'$ belongs to $\mathcal O$. Since $p^c\neq0$, the systems $C'z=0$ and $p^cC'z=0$ have the same solutions in $\Z_p^s$. Smith reduction over $\mathcal O$ effectively gives matrices $R\in\operatorname{GL}_h(\mathcal O)$, $S\in\operatorname{GL}_s(\mathcal O)$ such that
$$
R(p^cC')S
=
\operatorname{diag}(\delta_1,\ldots,\delta_t,0,\ldots,0),
$$
where $\delta_1,\ldots,\delta_t$ are non-zero. Since $R$ and $S$ are invertible and each $\delta_i\neq0$, the solutions in $\Z_p^s$ are precisely
$$
\Ker(C':\Z_p^s\to\Q_p^h)
=
\bigoplus_{\nu=1}^{s-t}\Z_pz_\nu,
\qquad
z_\nu:=S\varepsilon_{t+\nu},
$$
where $\varepsilon_1,\ldots,\varepsilon_s$ is the standard basis of $F^s$. Since $S\in\operatorname{GL}_s(\mathcal O)$, each $z_\nu$ has coordinates in $\mathcal O$, and the generators $z_\nu$ are effectively computable.

For each generator $z_\nu$, we have $C'z_\nu=0$, that is,
$q_A(Cz_\nu)=0$.  Since
$z_\nu\in F^s$, the vector $Cz_\nu$ belongs to $F^N$.  Hence
$q_A(Cz_\nu)=0$ implies that $Cz_\nu$ lies in the $F$-span of the
columns of $A$.  Therefore the equation
$Au=-Cz_\nu$ has a solution already in $F^r$, and Gaussian elimination
over $F$ computes one such solution $u_\nu\in F^r$.
The same Gaussian elimination computes vectors
$k_1,\ldots,k_\rho\in F^r$
whose $\Q_p$-span is
$\Ker(A:\Q_p^r\to\Q_p^N)$.

Every solution of \eqref{eq:mixedsystem} is therefore of the form
$$
z=\sum_\nu\lambda_\nu z_\nu,
\qquad
u=\sum_\nu\lambda_\nu u_\nu+\sum_{i=1}^{\rho}\mu_i k_i,
\qquad
\lambda_\nu\in\Z_p,\quad \mu_i\in\Q_p.
$$
In each of the intersection, kernel and inverse-image problems, the ambient element is given by an effectively known $F$-linear expression
$$
\Phi(u,z)=Pu+Qz.
$$
Hence
$$
\Phi(u,z)
=
\sum_{i=1}^{\rho}\mu_i\,Pk_i
+
\sum_\nu\lambda_\nu\,(Pu_\nu+Qz_\nu).
$$
Thus the subgroup obtained from the solution set has the required form
$$
D''+L'',
\qquad
D''=\operatorname{span}_{\Q_p}\{Pk_1,\ldots,Pk_\rho\},
\qquad
L''=\sum_\nu\Z_p(Pu_\nu+Qz_\nu).
$$
Since $P,Q,k_i,u_\nu,z_\nu$ are all effectively known over $F$, this presentation is computable and all its generators have coordinates in $F$. This proves intersections, kernels and inverse images, and hence \textup{(ii)} and \textup{(iii)}.

Finally, let $W=\Q_p^{d'}$, let
$$
B_0\leq B_1\leq V,
\qquad
C_0\leq C_1\leq W
$$
be subgroups given in the form \eqref{eq:2.13}, and let $T:V\to W$ be an $F$-linear map given exactly over $F$, with $T(B_1)\leq C_1$ and $T(B_0)\leq C_0$. Then $T$ induces
$$
\overline T:B_1/B_0\longrightarrow C_1/C_0.
$$
We represent a subgroup of $B_1/B_0$ by its full preimage in $B_1$, and similarly for $C_1/C_0$. With this convention, equality of cosets represented by $x,y\in F^d$ is decidable by \textup{(i)}, since
$$
x+B_0=y+B_0
\iff
x-y\in B_0.
$$
Moreover,
$$
\Ker(\overline T)
=
\{x\in B_1:Tx\in C_0\}/B_0,
\qquad
\operatorname{Im}(\overline T)
=
(T(B_1)+C_0)/C_0.
$$
More generally, if $C_0\leq H\leq C_1$ is given in the form \eqref{eq:2.13}, then
$$
\overline T^{-1}(H/C_0)
=
\{x\in B_1:Tx\in H\}/B_0.
$$
Thus kernels, images and inverse images in the quotients are computed by \textup{(ii)} and \textup{(iii)}. Sums and intersections are computed by taking sums and intersections of the corresponding full preimages and then quotienting by $B_0$. The coefficient data retained in the inverse-image computation also give representatives of preimages. Hence all the operations in \textup{(iv)} are effective.
\end{proof}

Next, we need the following result which allows us to pass to a virtually torsion-free cover (see \cite[Theorem~1.5]{KL20} and \cite[Proposition~1.3]{KL20}).

\begin{theorem}\label{thm:KL}
Let $G$ be a finitely generated virtually soluble minimax group.  There exist a finitely generated virtually torsion-free virtually soluble minimax group $H$ and an epimorphism $\phi:H\twoheadrightarrow G$
such that
\begin{equation}
 \Fitt(H)=\phi^{-1}(\Fitt(G)).
 \label{eq:2.21}
\end{equation}
If $G$ is soluble, $H$ may be chosen soluble with the same derived length, and the two Fitting subgroups have the same nilpotency class.
\end{theorem}

\section{Effective data for the fixed cover}\label{sec:3}

From this point on, the group $G$ and the cover $\phi:H\twoheadrightarrow G$ supplied by Theorem~\ref{thm:KL} are fixed. The purpose of this section is to make effective the group-theoretic data attached to this cover that will be used in the subsequent algorithms. We then make effective the relevant subgroups of the torsion-free nilpotent group $N_0$ and the successive abelian factors associated with its fixed series. Lemma~\ref{lem:wordcoords} provides the conversion from words representing elements of $N_0$ to their coordinate vectors, while Lemma~\ref{lem:effectivechains} gives the general effective construction of the $p$-primary chains needed for the factors $V_i$ and $W_i$. Lemma~\ref{lem:fixed} applies this construction and collects the resulting membership, coordinate and lifting procedures for the fixed subgroups and factors. Finally, Lemma~\ref{lem:ratpadicmembership} gives the local-to-global membership criterion for abelian minimax groups that will later be applied to the images arising in the conjugacy-lifting procedure.

\begin{convention}\label{def:fixeddata}
From now on, let $G$ be a fixed group satisfying the hypotheses of the Main Theorem.  By Proposition~\ref{prop:fatr}, $G$ is a finitely generated virtually soluble minimax group.  Therefore Theorem~\ref{thm:KL} applies, and we fix an epimorphism $\phi:H\twoheadrightarrow G$ as in that theorem. We also fix the recursive presentation of $G$, including its finite generating set and a procedure enumerating its defining relators. The input to the conjugacy algorithm consists of two words $x$ and $y$ in these fixed generators. All the preceding group-theoretic data are regarded as fixed throughout the argument, while $x$ and $y$ are the variable input. Some other fixed data will be introduced later.

Whenever finitely many fixed non-rational $p$-adic coefficients $\alpha_1,\ldots,\alpha_s$ occur in the constructions below, we consider the finitely generated field
$$
\Q(\alpha_1,\ldots,\alpha_s)\leq\Q_p.
$$
Choose a transcendence basis $\beta_1,\ldots,\beta_t$ among the $\alpha_i$.  Then the remaining $\alpha_i$ are algebraic over $\Q(\beta_1,\ldots,\beta_t)$, and therefore
$$
\Q(\alpha_1,\ldots,\alpha_s)/
\Q(\beta_1,\ldots,\beta_t)
$$
is a finite extension.  Choose a basis
$e_1,\ldots,e_d$ of this extension, with each $e_k$ a monomial in the
$\alpha_i$.  For every pair $i,j$, fix the coefficients
$c_{ij}^1,\ldots,c_{ij}^d\in\Q(\beta_1,\ldots,\beta_t)$ such that
$$
e_ie_j=\sum_{k=1}^d c_{ij}^k e_k,
$$
and also fix the coordinates of each $\alpha_i$ with respect to the
basis $e_1,\ldots,e_d$.  Since the basis is finite, this requires only
finitely many coefficients and gives precisely the algebraic data
needed in Lemma~\ref{lem:field}. The remaining hypothesis of Lemma~\ref{lem:field}, namely the ability
to approximate each $\alpha_i$ effectively by rational numbers to
arbitrary $p$-adic precision, is verified in each application below.  Thus, whenever such coefficients are introduced,
all the hypotheses of Lemma~\ref{lem:field} are satisfied.
More generally, whenever finitely many finitely generated exact coefficient fields $F^{(1)},\ldots,F^{(r)}\leq\Qp$ occur simultaneously in one computation, we replace them by the finitely generated subfield $E\leq\Qp$ generated jointly by their coefficient data, equipped with exact coefficient data of the same type. With respect to the fixed exact representations, the inclusions $F^{(\nu)}\hookrightarrow E$ are taken to be explicitly computable. Thus all exact $p$-adic coefficients occurring simultaneously may be handled in one common coefficient field.
\end{convention}

We shall use the Word Problem both in $G$ and in $H$.  Since $G$ is finitely generated, virtually soluble and minimax, and has a recursive presentation, \cite[Theorem~9.3.4]{LR04} applied to the recursively enumerable
subgroup $\{1\}$ gives a word-problem algorithm in $G$. By Theorem~\ref{thm:KL}, the group $H$ is finitely generated and virtually torsion-free minimax. Such groups are linear over $\Q$ (see \cite[Section~1]{KL20}).  Hence, for some $n$, there exists an injective homomorphism
$$
\rho:H\longrightarrow \operatorname{GL}_n(\Q).
$$
Since the group $G$, and hence the chosen cover $H\twoheadrightarrow G$, is fixed throughout, we choose and fix one such representation $\rho$, together with the rational matrices representing a fixed finite generating set of $H$.  No procedure for constructing this representation from the recursive presentation of $G$ is required; these finitely many matrices are treated as fixed auxiliary data depending only on the chosen group $G$. Once these matrices are fixed, the Word Problem in $H$ is decidable. Indeed, given a word in the fixed generators of $H$, its image under $\rho$ is computed by multiplying the corresponding rational matrices, and, since $\rho$ is injective, the word represents the identity in $H$ if and only if the resulting matrix is the identity matrix.

Now put
\begin{equation}
K:=\Ker(\phi)=\{h\in H:\phi(h)=1\}.
\label{eq:kernel-K}
\end{equation}
For a word $w$ in the fixed generators of $H$, the word $\phi(w)$ in the fixed generators of $G$ is effectively obtained from the fixed images under $\phi$ of the generators of $H$. Since the Word Problem in $G$ is decidable,
\begin{equation}
w\in K
\iff
\phi(w)=1\text{ in }G
\label{eq:deckernel-K}
\end{equation}
is decidable.  Hence membership in $K$ is decidable. Finally, the decidability of the Word Problem in $H$ implies that $H$ has a recursive presentation on its fixed finite generating set. Therefore, for every recursively enumerable subgroup $L\leq H$ that is fixed in advance, \cite[Theorem~9.3.4]{LR04} provides an algorithm deciding membership in $L$.  The algorithm is allowed to depend on the fixed subgroup $L$.

\smallskip

We shall also need to pass effectively from words in $H$ which are known to lie in a fixed nilpotent subgroup to their coordinates.

\begin{lemma}\label{lem:wordcoords}
Let $L$ be a torsion-free nilpotent minimax normal subgroup of $H$, and let $u_1,\ldots,u_h$, $\operatorname{cr}$, $P$ and $R$ be the coordinate data supplied by~\textnormal{Lemma~\ref{lem:coordinates}} for $L$.  Thus every element $x\in L$ has a unique expression $x=u_1^{a_1}\ldots u_h^{a_h}$ with $\operatorname{cr}(x)=(a_1,\ldots,a_h)\in\Q^h$.

Fix a finite generating set $s_1,\ldots,s_m$ of $H$, and assume that $H/L$ is finitely presented.  Fix a presentation $H/L=\langle \overline{s}_1,\ldots,\overline{s}_m \mid r_1,\ldots,r_t\rangle$, where $\overline{s}_i=s_iL$.  For $\nu=1,\ldots,t$, let $y_\nu=r_\nu(s_1,\ldots,s_m)\in L$, and fix the coordinate vectors $\operatorname{cr}(y_\nu)$.  Also fix, for every $i=1,\ldots,m$, $j=1,\ldots,h$, and $\varepsilon\in\{\pm1\}$, the coordinate vector $\operatorname{cr}(u_j^{\,s_i^\varepsilon})$. 

Then there is an algorithm which, on input a word $w$ in $s_1^{\pm1},\ldots,s_m^{\pm1}$ and representing an element of $L$, computes its coordinate vector $\operatorname{cr}(w)\in\Q^h$.
\end{lemma}
\begin{proof}
By the chosen presentation of $H/L$, the elements $y_1,\ldots,y_t$ defined in the statement normally generate $L$ in $H$.  Indeed, the kernel of the natural epimorphism from $H$ onto the presented quotient $H/L$ is the normal closure in $H$ of the lifts of the defining relators, and these lifts are precisely $y_1,\ldots,y_t$.

Enumerate all finite products
\begin{equation}
z=\prod_{\nu=1}^q
\bigl(y_{j_\nu}^{\varepsilon_\nu}\bigr)^{h_\nu},
\qquad
\varepsilon_\nu\in\{\pm1\},
\label{eq:2.23}
\end{equation}
where $q\geq0$, $j_\nu\in\{1,\ldots,t\}$, and each $h_\nu$ ranges
over the words in the fixed generators
$s_1^{\pm1},\ldots,s_m^{\pm1}$ of $H$.  These products enumerate the
normal closure of $y_1,\ldots,y_t$, and hence all elements of $L$.
For each such product, use the word-problem algorithm in $H$ to test
whether $z=w$.  Since $w$ is known to belong to $L$, this search
eventually terminates with a product \eqref{eq:2.23} representing $w$.

It remains to compute the coordinate vector of that product.  The
vectors $\operatorname{cr}(y_j)$ are fixed data, and hence the
coordinates of $y_j^{-1}$ are computed from $R$ in
Lemma~\ref{lem:coordinates}.  Moreover, for every
$s_i^{\pm1}$ the coordinate vectors
$\operatorname{cr}(u_j^{\,s_i^{\pm1}})$ are fixed.  Therefore
formula~\eqref{eq:external-aut}, together with the polynomials~$P$ and
$R$, computes the action by conjugation of each $s_i^{\pm1}$ on
coordinate vectors.  If
$h_\nu=s_{i_1}^{\delta_1}\ldots s_{i_r}^{\delta_r}$, with
$\delta_k\in\{\pm1\}$, the coordinate vector of
$\bigl(y_{j_\nu}^{\varepsilon_\nu}\bigr)^{h_\nu}$ is consequently
obtained by successively applying these explicitly known conjugation
maps.

Finally, the polynomial multiplication map $P$ computes the coordinate
vector of the product of the finitely many factors in
\eqref{eq:2.23}.  Since this product represents $w$, the resulting
vector is precisely $\operatorname{cr}(w)$.
\end{proof}

Now, fix a generating set $h_1,\ldots,h_m$ for $H$, and put
$N=\Fitt(H)$. For each
fixed generator of $G$, fix a word in $h_1^{\pm1},\ldots,h_m^{\pm1}$
whose image under $\phi$ is that generator. By~\eqref{eq:2.21}, $K\leq N$. Since $H$ is virtually
torsion-free, choose a torsion-free normal subgroup $H_0$ of $H$ of
finite index. Set
\begin{equation}
 N_0=N\cap H_0,
 \qquad
 K_0=K\cap H_0=K\cap N_0.
 \label{eq:3.1}
\end{equation}
Then $N_0$ is torsion-free nilpotent and normal in $H$.
By \cite[Proposition~1.3]{KL20}, $H/N$ is virtually abelian.
Since $H/N_0$ is finite-by-abelian-by-finite and finitely generated,
we conclude that it is also virtually abelian. But a finitely
generated virtually abelian group is finitely presented, so we fix
once and for all a finite presentation
$$
H/N_0=
\langle \overline h_1,\ldots,\overline h_m
\mid r_1,\ldots,r_t\rangle,
\qquad
\overline h_i=h_iN_0.
$$
For $\nu=1,\ldots,t$, put
$y_\nu=r_\nu(h_1,\ldots,h_m)\in N_0$.
Lifting its finitely many relators shows that $N_0$ is finitely
normally generated in $H$.

Define
\begin{equation}
 N_0=N_1\geq N_2\geq\ldots\geq N_{c+1}=1,
 \qquad
 N_i=\Isol_{N_0}(\gamma_i(N_0)),
 \label{eq:3.3}
\end{equation}
and put
\begin{equation}
 V_i=N_i/N_{i+1},
 \qquad
 J_i=K_0\cap N_i,
 \qquad
 W_i=J_iN_{i+1}/N_{i+1}\leq V_i.
 \label{eq:3.4}
\end{equation}
In particular, every $V_i$ is torsion-free abelian minimax.

Apply Lemma~\ref{lem:coordinates} to $N_0$, making the
choices in its proof so that, for each $i$, the finite set chosen
in $N_i$ has images in $V_i$ generating a free abelian subgroup of
full rank with divisible quotient. If
$u_{0,1},\ldots,u_{0,h_0}$
are the resulting coordinate elements and
$N_0^{\ast}=\langle u_{0,1},\ldots,u_{0,h_0}\rangle$, choose
them, as in the proof of that lemma, so that for each $i$ there is
an index $e_i$ such that
$$
N_i\cap N_0^{\ast}
=
\langle u_{0,e_i},\ldots,u_{0,h_0}\rangle.
$$
For each $\nu=1,\ldots,h_0$, fix a word in
$h_1^{\pm1},\ldots,h_m^{\pm1}$ representing $u_{0,\nu}$ in $H$.

Fix also the coordinate vector of each $y_\nu$ with respect to
$u_{0,1},\ldots,u_{0,h_0}$. Moreover, for every
$i=1,\ldots,m$, $\nu=1,\ldots,h_0$, and
$\varepsilon\in\{\pm1\}$, fix the coordinate vector of
$
u_{0,\nu}^{\,h_i^\varepsilon}
$
with respect to $u_{0,1},\ldots,u_{0,h_0}$.

For each $i$, we use
$u_{0,e_i},\ldots,u_{0,h_0}$, together with the corresponding
restrictions of the coordinate polynomials for $N_0$, as the
coordinate data for $N_i$. We include these choices, together with the finite presentation of $H/N_0$
and the coordinate vectors fixed above,
among the fixed data of Convention~\ref{def:fixeddata}.

Note that all the hypotheses of Lemma~\ref{lem:wordcoords} are therefore
satisfied for~\hbox{$L=N_0$} and $s_i=h_i$. Hence, whenever a word in
$h_1^{\pm1},\ldots,h_m^{\pm1}$ is known to represent an element
of $N_0$, its coordinate vector with respect to
$u_{0,1},\ldots,u_{0,h_0}$ is computable.

After extending the coordinates to $\Qp$ as above, the coordinate
vector of an element of $(N_i)_p$ is therefore identified with its
coordinate vector in~$(N_0)_p$ by adjoining $e_i-1$ initial zero
entries.

For each $i$, set
$$
S_i=(N_i\cap N_0^{\ast})N_{i+1}/N_{i+1}\leq V_i.
$$
Since the finite set used in the construction of $N_0^{\ast}$ at
level $i$ maps into $S_i$, the group $S_i$ has full rank in $V_i$;
moreover, $V_i/S_i$ is a quotient of the divisible quotient fixed
in that construction, and is therefore divisible.  Since
$$
N_{i+1}\cap N_0^{\ast}
=
\langle u_{0,e_{i+1}},\ldots,u_{0,h_0}\rangle,
$$
the elements
$$
s_{i,1}=u_{0,e_i}N_{i+1},\ldots,
s_{i,r_i}=u_{0,e_{i+1}-1}N_{i+1},
\qquad
r_i=e_{i+1}-e_i,
$$
form a basis of $S_i$.  In particular, they form a $\Q$-basis of
$(V_i)_{\Q}=V_i\otimes_{\Z}\Q$, and their word lifts to $N_i$ are
the words already fixed for
$u_{0,e_i},\ldots,u_{0,e_{i+1}-1}$.

Choose also a free abelian subgroup $T_i\leq W_i$ of full rank such
that $W_i/T_i$ is divisible, and fix a basis
$t_{i,1},\ldots,t_{i,d_i}$ of $T_i$, together with word lifts of
these elements to $J_i$.

For every $p\in\spec(V_i)$, fix an $\mathbf F_p$-basis
\(
e_{i,p,1}^{V},\ldots,e_{i,p,r_{i,p}^{V}}^{V}
\)
of $(V_i/S_i)[p]$.  For each basis element, fix a representative
$b_{i,p,j,1}^{V}\in V_i$ and a word
$\widehat b_{i,p,j,1}^{V}$ in
$h_1^{\pm1},\ldots,h_m^{\pm1}$ representing an element of $N_i$
whose image under
\hbox{$\pi_i^V:N_i\twoheadrightarrow V_i$}
is $b_{i,p,j,1}^{V}$. Likewise, for every $p\in\spec(W_i)$, fix an $\mathbf F_p$-basis
of $(W_i/T_i)[p]$, representatives
$b_{i,p,j,1}^{W}\in W_i$ of its basis elements, and words
$\widehat b_{i,p,j,1}^{W}$ in
$h_1^{\pm1},\ldots,h_m^{\pm1}$ representing elements of $J_i$
whose images under
$\pi_i^W:J_i\twoheadrightarrow W_i$
are the corresponding representatives.

All these choices are included among the fixed data of
Convention~\ref{def:fixeddata}. They are finite in number, since the
series \eqref{eq:3.3} has finitely many factors and each of
$\spec(V_i)$ and $\spec(W_i)$ is finite.  With these fixed data, the following computations can be performed.

\begin{lemma}\label{subgroupsniarere}
The subgroups $N_i$ are recursively enumerable in $H$.
\end{lemma}
\begin{proof}
The subgroup $N_0$ is recursively enumerable because it is finitely normally generated in $H$. Now, starting with an enumeration of~$N_0$, enumerate all
simple commutators
$[n_1,\ldots,n_i]$ with $n_1,\ldots,n_i\in N_0$, and then all finite
products of these commutators and their inverses.  Since~$\gamma_i(N_0)$ is generated by the simple commutators of weight
$i$, this enumerates precise\-ly~$\gamma_i(N_0)$. Now enumerate pairs $(w,q)$, where $w$ is a word in
$h_1^{\pm1},\ldots,h_m^{\pm1}$ representing an element of~$N_0$ and
$q\geq1$, together with the above enumeration of~$\gamma_i(N_0)$.  Output $w$ whenever the Word Problem in $H$
verifies
$w^q=z$ for some enumerated $z\in\gamma_i(N_0)$.  If $w$ is output, then the element represented by $w$ has a positive power in $\gamma_i(N_0)$, and hence belongs to $\Isol_{N_0}(\gamma_i(N_0))$.  Conversely, if $n\in\Isol_{N_0}(\gamma_i(N_0))$, then $n^q\in\gamma_i(N_0)$ for some $q\geq1$, so the corresponding word $w$, exponent $q$, and element $z=n^q$ eventually occur in the enumeration and $w$ is output.  Therefore this procedure enumerates precisely $N_i=\Isol_{N_0}(\gamma_i(N_0))$.
\end{proof}

We next make effective, in a form that can be applied uniformly to the factors $V_i$ and $W_i$, the $p$-primary chains used in the description of tensor $p$-completions from Section~\ref{sec:2}.

\begin{lemma}\label{lem:effectivechains}
Let $B$ be a torsion-free abelian minimax group whose elements are effectively represented, and such that addition, inverses and equality in~$B$ are computable. Let $L\leq B$ be a free abelian subgroup with $r_0(L)=r_0(B)$ such that~$B/L$ is divisible, fix a basis $\ell_1,\ldots,\ell_r$ of $L$, and assume that membership in~$L$ is decidable.

Let $\pi:Y\twoheadrightarrow B$ be an epimorphism, where the elements of $Y$ are given by an effective enumeration of words on a certain system of generators, and assume that $\pi(y)$ can be computed from such a word.  For every $p\in\spec(B)$, fix an $\mathbf F_p$-basis $e_{1,1},\ldots,e_{r_p,1}$ of $(B/L)[p]$, representatives $b_{p,j,1}\in B$ satisfying $b_{p,j,1}+L=e_{j,1}$, and words $\widehat b_{p,j,1}\in Y$ satisfying $\pi(\widehat b_{p,j,1})=b_{p,j,1}$.

Then there is a single procedure which, given $p\in\spec(B)$ and $n\geq1$, constructs, for every $j=1,\ldots,r_p$, elements $b_{p,j,1},\ldots,b_{p,j,n}\in B$ and words $\widehat b_{p,j,1},\ldots,\widehat b_{p,j,n}\in Y$ such that $\pi(\widehat b_{p,j,k})=b_{p,j,k}$ for every $k\leq n$,
\[
p(b_{p,j,k+1}+L)=b_{p,j,k}+L
\qquad (1\leq k<n),
\]
each $b_{p,j,k}+L$ has order $p^k$, and
\[
(B/L)[p^k]
=
\bigoplus_{j=1}^{r_p}
\langle b_{p,j,k}+L\rangle
\qquad (1\leq k\leq n).
\]

With respect to the fixed basis $\ell_1,\ldots,\ell_r$, identify
$L\otimes_{\Z}\Zp$ with $\Zp^r$.  Then the limits
\[
v_{p,j}
=
\lim_{n\to\infty}p^n b_{p,j,n}
\in L\otimes_{\Z}\Zp
\]
are effectively approximable: for every $n$, the coordinates of $p^n b_{p,j,n}\in L$ can be computed and differ from those of $v_{p,j}$ by an element of $p^n\Zp^r$.  Moreover,
\[
B_p
=
(L\otimes_{\Z}\Zp)
+
\sum_{j=1}^{r_p}\Qp v_{p,j}.
\]

After fixing an exact presentation, as in {\rm Convention~\ref{def:fixeddata}}, of the subfield of $\Qp$ generated over $\Q$ by the coordinates of the finitely many $v_{p,j}$, {\rm Lemma~\ref{lem:additive}} gives terminating membership and preimage procedures in $B_p$.  If $p\notin\spec(B)$, then $B_p=L\otimes_{\Z}\Zp$.
\end{lemma}

\begin{proof}
Fix $p\in\spec(B)$.  We construct the elements $b_{p,j,n}$ and their lifts inductively on $n$.  The element $b_{p,j,1}$ is already part of the fixed data. Suppose that $b_{p,j,n}$ and a word $\widehat b_{p,j,n}\in Y$ with $\pi(\widehat b_{p,j,n})=b_{p,j,n}$ have been constructed. Enumerate the words representing elements $y\in Y$, compute $b=\pi(y)\in B$, and use the membership test for $L$ to test whether
\[
pb-b_{p,j,n}\in L.
\]
Since $B/L$ is divisible, there exists $b\in B$ satisfying this condition. Since $\pi$ is onto, such an element occurs in the enumeration of $Y$. Hence the search terminates. For the first successful word $y$, put $b_{p,j,n+1}=b$ and $\widehat b_{p,j,n+1}=y$. Then $p(b_{p,j,n+1}+L)=b_{p,j,n}+L$. Since $b_{p,j,n}+L$ has order $p^n$, it follows that $b_{p,j,n+1}+L$ has order $p^{n+1}$. It is also clear that, for every $n\geq1$, the classes $b_{p,1,n}+L,\ldots,b_{p,r_p,n}+L$ are independent, so
\[
(B/L)[p^n]
=
\bigoplus_{j=1}^{r_p}
\langle b_{p,j,n}+L\rangle.
\]

Since $b_{p,j,n}+L$ has order $p^n$, we have $p^n b_{p,j,n}\in L$. The relations constructed above are precisely those used in Lemma~\ref{lem:abelianpadic}; hence
\[
v_{p,j}
=
\lim_{n\to\infty}p^n b_{p,j,n}
\]
exists in $L\otimes_{\Z}\Zp$ and satisfies
\[
v_{p,j}-p^n b_{p,j,n}
\in
p^n(L\otimes_{\Z}\Zp).
\]

It remains to prove the effectivity of these approximations.  Given $p^n b_{p,j,n}\in L$, its coordinates with respect to the fixed basis $\ell_1,\ldots,\ell_r$ can be found effectively: enumerate the tuples $(a_1,\ldots,a_r)\in\Z^r$ and use equality in $B$ until $p^n b_{p,j,n} = a_1\ell_1+\ldots+a_r\ell_r$. Since the element belongs to $L$, this search terminates.  Under the identification $L\otimes_{\Z}\Zp\simeq\Zp^r$, the preceding inclusion therefore shows that these integer coordinate vectors approximate the coordinates of $v_{p,j}$ modulo $p^n\Zp^r$.

The constructed chains have exactly the properties used in Lemma~\ref{lem:abelian-q-completion}, and hence
\[
B_p
=
(L\otimes_{\Z}\Zp)
+
\sum_{j=1}^{r_p}\Qp v_{p,j}.
\]
Relative to the basis $\ell_1,\ldots,\ell_r$, this is a subgroup of $\Qp^r$ of the form considered in Lemma~\ref{lem:additive}: $L\otimes_{\Z}\Zp$ is generated by the standard basis, and the divisible part is spanned by the vectors $v_{p,j}$.  After fixing the exact field data specified in the statement, Lemma~\ref{lem:additive} therefore gives the required terminating membership and preimage procedures. Finally, if $p\notin\spec(B)$, Lemma~\ref{lem:abelian-q-completion} gives directly $B_p=L\otimes_{\Z}\Zp$.
\end{proof}

We now apply the preceding construction to the fixed factors $V_i$ and $W_i$ and collect, in a form convenient for later use, all the rational and $p$-adic computations that will be needed in the lifting argument.

\begin{lemma}\label{lem:fixed}
For every $i$ the following computations can be performed.

\begin{enumerate}[label=\textup{(\roman*)}]

\item Given a word in $h_1^{\pm1},\ldots,h_m^{\pm1}$, one can decide whether the element of $H$ represented by that word belongs to $N_i$ or to $J_i$. If the word represents an element $n\in N_i$, one can compute the coordinates of $nN_{i+1}$ in $(V_i)_{\Q}$ with respect to $s_{i,1},\ldots,s_{i,r_i}$, and one can decide whether $nN_{i+1}$ belongs to $W_i$. Moreover, given an element of $(V_i)_{\Q}$ by its coordinates with respect to $s_{i,1},\ldots,s_{i,r_i}$, one can decide whether it belongs to $(W_i)_{\Q}$ and, if so, compute its coordinates with respect to $t_{i,1},\ldots,t_{i,d_i}$.

$\;$ In addition, one can compute a basis of $(V_i)_{\Q}/(W_i)_{\Q}$ and the matrix, with respect to $s_{i,1},\ldots,s_{i,r_i}$ and this quotient basis, of the natural map $(V_i)_{\Q}\to(V_i)_{\Q}/(W_i)_{\Q}$. In particular, the quotient coordinates of every element of $(V_i)_{\Q}$ are computable.

\item Given a prime $p$, one can construct finite presentations
$$
(V_i)_p=\operatorname{span}_{\Qp}\{d_1^V,\ldots,d_a^V\}+\sum_{k=1}^{b}\Zp\ell_k^V
$$
and
$$
(W_i)_p=\operatorname{span}_{\Qp}\{d_1^W,\ldots,d_{a'}^W\}+\sum_{k=1}^{b'}\Zp\ell_k^W
$$
inside $(V_i)_{\Q}\otimes_{\Q}\Qp$. If $p\notin\spec(V_i)$, one may take $a=0$, $b=r_i$ and $\ell_k^V=s_{i,k}\otimes1$ for $1\leq k\leq r_i$; likewise, if $p\notin\spec(W_i)$, one may take $a'=0$, $b'=d_i$ and $\ell_k^W=t_{i,k}\otimes1$ for $1\leq k\leq d_i$. In particular, in these cases the corresponding presentation vectors have rational coordinates. All the displayed vectors have coordinates in a finitely generated subfield $F_{i,p}\leq\Qp$ whose exact coefficient data are fixed as in {\rm Convention~\ref{def:fixeddata}}; with these data, equality, the field operations and the $p$-adic valuation in $F_{i,p}$ are computable.

$\;$ Given an input vector $x\in F_{i,p}^{r_i}$, with coordinates given in the fixed exact representation of $F_{i,p}$, which represents an element of $(V_i)_p$ with respect to $s_{i,1},\ldots,s_{i,r_i}$, one can compute coefficients $\lambda_1,\ldots,\lambda_a\in F_{i,p}$ and $z_1,\ldots,z_b\in F_{i,p}\cap\Zp$ such that
$$
x=\sum_{j=1}^{a}\lambda_jd_j^V+\sum_{k=1}^{b}z_k\ell_k^V.
$$
One can also decide whether $x\in(W_i)_p$. If the answer is positive, one can compute coefficients $\mu_1,\ldots,\mu_{a'}\in F_{i,p}$ and $z'_1,\ldots,z'_{b'}\in F_{i,p}\cap\Zp$ such that
$$
x=\sum_{j=1}^{a'}\mu_jd_j^W+\sum_{k=1}^{b'}z'_k\ell_k^W.
$$

$\;$ Moreover, given coefficients $\lambda_j\in F_{i,p}$ and $z_k\in F_{i,p}\cap\Zp$ defining
$$
v=\sum_j\lambda_jd_j^V+\sum_k z_k\ell_k^V\in(V_i)_p,
$$
one can compute, in the $\Qp$-extension of the coordinate system fixed for~$N_i$, a lift $\widetilde v\in(N_i)_p$ satisfying $(\pi_i^V)_p(\widetilde v)=v$. By adjoining the initial zero coordinates, this also gives its coordinate vector with respect to $u_{0,1},\ldots,u_{0,h_0}$.

$\;$ Likewise, given coefficients $\mu_j\in F_{i,p}$ and $z'_k\in F_{i,p}\cap\Zp$ defining
$$
w=\sum_j\mu_jd_j^W+\sum_k z'_k\ell_k^W\in(W_i)_p,
$$
one can compute a lift $\widetilde w\in(J_i)_p$ satisfying $(\pi_i^W)_p(\widetilde w)=w$, both in the $\Qp$-extension of the coordinate system supplied by {\rm Lemma~\ref{lem:coordinates}} for~$J_i$ and, after the fixed coordinate conversion, in the coordinates with respect to $u_{0,1},\ldots,u_{0,h_0}$ in $(N_0)_p$

$\;$ All these decomposition, membership, preimage and lifting procedures remain valid after replacing $F_{i,p}$ by any finitely generated subfield $E\leq\Qp$ containing $F_{i,p}$, provided that $E$ is given by an exact finite presentation for which equality, the field operations and the $p$-adic valuation are computable, and that the inclusion $F_{i,p}\hookrightarrow E$ is effective with respect to the chosen exact representations.

\end{enumerate}
\end{lemma}

\begin{proof}
We first prove \textup{(i)}. Since $N_i$ is a fixed recursively enumerable subgroup of $H$ (see Lemma~\ref{subgroupsniarere}), \cite[Theorem~9.3.4]{LR04} decides membership in~$N_i$. Membership in $K$ is decidable by \eqref{eq:deckernel-K}. Moreover, since $H_0\trianglelefteq H$ has finite index, we fix the finite quotient $H/H_0$ and the images in this quotient of the fixed generators of~$H$. Hence membership in $H_0$ is decidable. Since $K_0=K\cap H_0$, membership in $K_0$ is decidable, and therefore so is membership in $J_i=K_0\cap N_i$.

We represent an element of $V_i=N_i/N_{i+1}$ by a word in $h_1^{\pm1},\ldots,h_m^{\pm1}$ which is known to represent an element $n\in N_i$, the represented element being $nN_{i+1}$. Addition and inverses in this representation are computed by multiplication and inversion of words. Equality is decidable, since
$$
nN_{i+1}=n'N_{i+1}\quad\Longleftrightarrow\quad n^{-1}n'\in N_{i+1},
$$
and membership in $N_{i+1}$ is decidable.

Likewise, an element of $W_i=J_iN_{i+1}/N_{i+1}$ may be represented by a word which is known to represent an element $j\in J_i$, the represented element being $jN_{i+1}$. Addition and inverses are again computed by multiplication and inversion of words, and equality is decided in the same way modulo~$N_{i+1}$. Thus addition, inverses and equality are computable in the chosen representations of both $V_i$ and $W_i$.

Let now $n\in N_i$ be represented by a given word. Since $n\in N_0$, Lemma~\ref{lem:wordcoords}, applied to $N_0$, computes its coordinate vector with respect to $u_{0,1},\ldots,u_{0,h_0}$. By the choice of the coordinate data adapted to the series \eqref{eq:3.3}, this vector has the form
$$
(0,\ldots,0,a_{e_i},\ldots,a_{h_0}),\qquad a_\nu\in\Q.
$$
Moreover,
$$
s_{i,1}=u_{0,e_i}N_{i+1},\ldots,s_{i,r_i}=u_{0,e_{i+1}-1}N_{i+1}.
$$
Hence the coordinates of $nN_{i+1}$ in $(V_i)_{\Q}$ with respect to $s_{i,1},\ldots,s_{i,r_i}$ are simply
$$
(a_{e_i},\ldots,a_{e_{i+1}-1}),
$$
and are therefore computable directly from the given word.

Finally,
$$
nN_{i+1}\in W_i\quad\Longleftrightarrow\quad n\in J_iN_{i+1}.
$$
The subgroup $J_iN_{i+1}$ is fixed and recursively enumerable: enumerate words representing elements $j\in J_i$ and $u\in N_{i+1}$ and enumerate the products $ju$. Hence \cite[Theorem~9.3.4]{LR04} decides membership in $J_iN_{i+1}$.

It remains to prove the rational assertion in \textup{(i)}. Since $T_i$ has full rank in~$W_i$, the elements $t_{i,1},\ldots,t_{i,d_i}$ form a $\Q$-basis of $(W_i)_{\Q}$. For each $k=1,\ldots,d_i$, apply the procedure just proved to the fixed word lift of $t_{i,k}$ and compute its coordinates in $(V_i)_{\Q}$ with respect to $s_{i,1},\ldots,s_{i,r_i}$. Let
$$
C_i\in M_{r_i\times d_i}(\Q)
$$
be the matrix whose $k$-th column is this coordinate vector.

Now let $x\in(V_i)_{\Q}$ be given by its coordinate vector $\mathbf{x}\in\Q^{r_i}$ with respect to $s_{i,1},\ldots,s_{i,r_i}$. Since the columns of $C_i$ are the coordinate vectors of a $\Q$-basis of $(W_i)_{\Q}$,
$$
x\in(W_i)_{\Q}\quad\Longleftrightarrow\quad C_i\boldsymbol{\lambda}=\mathbf{x}\text{ for some }\boldsymbol{\lambda}\in\Q^{d_i}.
$$
Gaussian elimination over $\Q$ decides whether such a vector exists. When it exists, it is unique, and its entries are precisely the coordinates of $x$ with respect to $t_{i,1},\ldots,t_{i,d_i}$.

The same Gaussian elimination computes a matrix
$$
R_i^{\mathrm{quot}}\in M_{(r_i-d_i)\times r_i}(\Q)
$$
of rank $r_i-d_i$ such that
$$
\Ker R_i^{\mathrm{quot}}=\operatorname{Im}C_i=(W_i)_{\Q}
$$
under the coordinate identification $(V_i)_{\Q}\simeq\Q^{r_i}$. For each $\nu=1,\ldots,r_i-d_i$, solve
$$
R_i^{\mathrm{quot}}\mathbf q_{i,\nu}=\mathbf e_\nu
$$
over $\Q$. If $q_{i,\nu}\in(V_i)_{\Q}$ denotes the element with coordinate vector $\mathbf q_{i,\nu}$, then
$$
q_{i,1}+(W_i)_{\Q},\ldots,q_{i,r_i-d_i}+(W_i)_{\Q}
$$
form a basis of $(V_i)_{\Q}/(W_i)_{\Q}$, and for every $x\in(V_i)_{\Q}$ with coordinate vector~$\mathbf x$ the quotient coordinate vector of $x+(W_i)_{\Q}$ in this basis is
$
R_i^{\mathrm{quot}}\mathbf x.
$
We fix this quotient basis and quotient matrix from now on. This completes the proof of \textup{(i)}.

We record first the membership procedures for the lattices $S_i$ and $T_i$ that will be used below. If $v\in V_i$ is represented by a word for an element of $N_i$, then \textup{(i)} computes its coordinates $\lambda_1,\ldots,\lambda_{r_i}$ with respect to $s_{i,1},\ldots,s_{i,r_i}$. Since $S_i=\langle s_{i,1},\ldots,s_{i,r_i}\rangle$, we have
$$
v\in S_i\quad\Longleftrightarrow\quad \lambda_1,\ldots,\lambda_{r_i}\in\Z.
$$
Thus membership in $S_i$ is decidable. Likewise, if $w\in W_i$ is represented by a word for an element of $J_i$, then the rational procedure in \textup{(i)} computes its coordinates $\mu_1,\ldots,\mu_{d_i}$ with respect to $t_{i,1},\ldots,t_{i,d_i}$. Since $T_i=\langle t_{i,1},\ldots,t_{i,d_i}\rangle$, we have
$$
w\in T_i\quad\Longleftrightarrow\quad \mu_1,\ldots,\mu_{d_i}\in\Z.
$$
Thus membership in $T_i$ is decidable.

\smallskip

We now prove \textup{(ii)}. We first apply Lemma~\ref{lem:effectivechains} to
$$
B=V_i,\qquad L=S_i,\qquad Y=N_i,
$$
with the epimorphism
$$
\pi_i^V:N_i\twoheadrightarrow V_i,\qquad \pi_i^V(n)=nN_{i+1}.
$$
Indeed, $V_i$ is torsion-free abelian minimax and, by \textup{(i)}, its chosen word representation has computable addition, inverses and equality. The subgroup~$S_i$ is free abelian of full rank, $V_i/S_i$ is divisible, the basis $s_{i,1},\ldots,s_{i,r_i}$ is fixed, and membership in $S_i$ is decidable by the preceding paragraph. The elements of $N_i$ can be effectively enumerated by words in $h_1^{\pm1},\ldots,h_m^{\pm1}$, and~$\pi_i^V(n)$ is computed simply by retaining such a word as a representative of~$nN_{i+1}$. Finally, for every $p\in\spec(V_i)$, the required $\mathbf F_p$-basis of $(V_i/S_i)[p]$, its representatives in $V_i$, and their word lifts to $N_i$ were fixed before the statement of the lemma. Thus all the hypotheses of Lemma~\ref{lem:effectivechains} are satisfied.

We apply the same lemma to
$$
B=W_i,\qquad L=T_i,\qquad Y=J_i,
$$
with
$$
\pi_i^W:J_i\twoheadrightarrow W_i,\qquad \pi_i^W(j)=jN_{i+1}.
$$
The group $W_i$ is torsion-free abelian minimax and, by \textup{(i)}, its chosen word representation has computable addition, inverses and equality. The subgroup~$T_i$ is free abelian of full rank, $W_i/T_i$ is divisible, the basis $t_{i,1},\ldots,t_{i,d_i}$ is fixed, and membership in $T_i$ is decidable by the preceding paragraph. Since membership in $J_i$ is decidable, the elements of $J_i$ can be effectively enumerated by words in $h_1^{\pm1},\ldots,h_m^{\pm1}$, and $\pi_i^W(j)$ is computed by retaining such a word as a representative of $jN_{i+1}$. For every $p\in\spec(W_i)$, the required $\mathbf F_p$-basis of $(W_i/T_i)[p]$, its representatives in $W_i$, and their word lifts to $J_i$ were fixed before the statement of the lemma. Hence all the hypotheses of Lemma~\ref{lem:effectivechains} are satisfied.

For uniform notation, put
$$
r_{i,p}^{V}=\dim_{\mathbf F_p}(V_i/S_i)[p]\quad\text{if }p\in\spec(V_i),\qquad r_{i,p}^{V}=0\quad\text{otherwise},
$$
and define $r_{i,p}^{W}$ analogously using $(W_i/T_i)[p]$.

For $p\in\spec(V_i)$, Lemma~\ref{lem:effectivechains} produces effectively approximable limit vectors
$$
v_{i,p,1}^{V},\ldots,v_{i,p,r_{i,p}^{V}}^{V}\in S_i\otimes_{\Z}\Zp
$$
such that
$$
(V_i)_p=\operatorname{span}_{\Qp}\{v_{i,p,1}^{V},\ldots,v_{i,p,r_{i,p}^{V}}^{V}\}+\sum_{k=1}^{r_i}\Zp(s_{i,k}\otimes1).
$$
Write
$$
v_{i,p,j}^{V}=\sum_{k=1}^{r_i}\alpha_{i,p,j,k}(s_{i,k}\otimes1),\qquad \alpha_{i,p,j,k}\in\Zp.
$$
For every prescribed precision, Lemma~\ref{lem:effectivechains} computes rational approximations to all the coefficients $\alpha_{i,p,j,k}$. If $p\notin\spec(V_i)$, the first summand in the displayed presentation is absent.

Likewise, for $p\in\spec(W_i)$ the same lemma produces effectively approximable limit vectors
$$
v_{i,p,1}^{W},\ldots,v_{i,p,r_{i,p}^{W}}^{W}\in T_i\otimes_{\Z}\Zp
$$
such that
$$
(W_i)_p=\operatorname{span}_{\Qp}\{v_{i,p,1}^{W},\ldots,v_{i,p,r_{i,p}^{W}}^{W}\}+\sum_{k=1}^{d_i}\Zp(t_{i,k}\otimes1).
$$
Write
$$
v_{i,p,j}^{W}=\sum_{k=1}^{d_i}\beta_{i,p,j,k}(t_{i,k}\otimes1),\qquad \beta_{i,p,j,k}\in\Zp.
$$
Again, all the coefficients $\beta_{i,p,j,k}$ are effectively approximable to arbitrary $p$-adic precision. If $p\notin\spec(W_i)$, the first summand is absent.

For $p\in\spec(V_i)\cup\spec(W_i)$, let $F_{i,p}\leq\Qp$ be the subfield generated over $\Q$ by all the coefficients $\alpha_{i,p,j,k}$ and $\beta_{i,p,j,k}$ occurring above; if $p\notin\spec(V_i)\cup\spec(W_i)$, put $F_{i,p}=\Q$. We use the exact finite presentation of $F_{i,p}$ fixed as part of Convention~\ref{def:fixeddata}. The approximation procedures supplied by Lemma~\ref{lem:effectivechains}, together with Lemma~\ref{lem:field}, make equality, the field operations and the $p$-adic valuation effective in $F_{i,p}$.

Thus the first displayed presentation has the form required in \textup{(ii)}, with
$$
d_j^V=v_{i,p,j}^{V},\qquad \ell_k^V=s_{i,k}\otimes1,
$$
where the vectors $d_j^V$ are omitted when $p\notin\spec(V_i)$. To obtain the corresponding presentation of $(W_i)_p$ inside $(V_i)_{\Q}\otimes_{\Q}\Qp$, use the rational matrix $C_i$ constructed in the proof of \textup{(i)} to express each $t_{i,k}$, and hence each $v_{i,p,j}^{W}$, in the basis $s_{i,1},\ldots,s_{i,r_i}$. We then take
$$
d_j^W=v_{i,p,j}^{W},\qquad \ell_k^W=t_{i,k}\otimes1,
$$
again omitting the vectors $d_j^W$ when $p\notin\spec(W_i)$. Since $C_i$ has rational entries, all coordinates of the displayed vectors belong to $F_{i,p}$.

For the decomposition and membership procedures in \textup{(ii)}, let $x\in(V_i)_p$ be given by its coordinate vector in the exact coefficient representation of $F_{i,p}$. Applying Lemma~\ref{lem:additive} to the displayed presentation of $(V_i)_p$ gives coefficients
$$
\lambda_1,\ldots,\lambda_a\in F_{i,p},\qquad z_1,\ldots,z_b\in F_{i,p}\cap\Zp
$$
such that
$$
x=\sum_{j=1}^{a}\lambda_jd_j^V+\sum_{k=1}^{b}z_k\ell_k^V.
$$
Applying the same lemma to the displayed presentation of $(W_i)_p$ decides whether $x\in(W_i)_p$. When the answer is positive, it computes coefficients
$$
\mu_1,\ldots,\mu_{a'}\in F_{i,p},\qquad z'_1,\ldots,z'_{b'}\in F_{i,p}\cap\Zp
$$
such that
$$
x=\sum_{j=1}^{a'}\mu_jd_j^W+\sum_{k=1}^{b'}z'_k\ell_k^W.
$$

It remains to construct the lifts required in \textup{(ii)}. Consider first
$$
(\pi_i^V)_p:(N_i)_p\twoheadrightarrow(V_i)_p.
$$
Recall that the coordinate data fixed for $N_i$ are the final segment $u_{0,e_i},\ldots,u_{0,h_0}$ of the coordinate data for $N_0$. Hence, whenever a word in $h_1^{\pm1},\ldots,h_m^{\pm1}$ is known to represent an element of $N_i$, Lemma~\ref{lem:wordcoords}, applied to $N_0$, computes its coordinate vector with respect to $u_{0,1},\ldots,u_{0,h_0}$; its first $e_i-1$ entries are zero, and deleting them gives its coordinate vector in the fixed system for~$N_i$.

The fixed lifts of $s_{i,1},\ldots,s_{i,r_i}$ to $N_i$ are $u_{0,e_i},\ldots,u_{0,e_{i+1}-1}$, so their coordinate vectors in $N_i$ are the first $r_i$ standard basis vectors. Moreover, in the application of Lemma~\ref{lem:effectivechains} to $\pi_i^V$, the lemma computes, for every $j$ and $n$, the integers $c_{i,p,j,k,n}$ determined by
$$
p^n b_{i,p,j,n}^{V}=\sum_{k=1}^{r_i}c_{i,p,j,k,n}s_{i,k},
$$
and constructs a word $\widehat b_{i,p,j,n}^{V}\in N_i$ lifting $b_{i,p,j,n}^{V}$. By the preceding paragraph, the coordinates of this lift in the fixed coordinate system for $N_i$ are computable from $j$ and $n$.

We may therefore apply Lemma~\ref{lem:factorlift} with $Y=N_i$, $B=V_i$, $L=S_i$, $q=p$ and coefficient field $F_{i,p}$. Indeed, the coefficients $a_{j,k}$ in that lemma are precisely the coefficients $\alpha_{i,p,j,k}$ of
$$
v_{i,p,j}^{V}=\sum_{k=1}^{r_i}\alpha_{i,p,j,k}(s_{i,k}\otimes1),
$$
while the remaining required data are supplied by the preceding paragraph and by Lemma~\ref{lem:effectivechains}. Thus, given coefficients $\lambda_j\in F_{i,p}$ and $z_k\in F_{i,p}\cap\Zp$ defining
$$
v=\sum_j\lambda_jd_j^V+\sum_k z_k\ell_k^V\in(V_i)_p,
$$
Lemma~\ref{lem:factorlift} computes the coordinate vector, in the fixed coordinate system for $(N_i)_p$, of an element $\widetilde v\in(N_i)_p$ satisfying
$$
(\pi_i^V)_p(\widetilde v)=v.
$$
If $p\notin\spec(V_i)$, the divisible sum is absent and the same argument applies using only the lattice generators.

For the corresponding lift from $(W_i)_p$, apply Lemma~\ref{lem:coordinates} to $J_i$ and fix the resulting coordinate elements $y_{i,1},\ldots,y_{i,h_i'}$, coordinate polynomials, and words in $h_1^{\pm1},\ldots,h_m^{\pm1}$ representing the $y_{i,\nu}$. Coordinates of a word known to represent an element of $J_i$ are computable in this system: if $J_i^\ast=\langle y_{i,1},\ldots,y_{i,h_i'}\rangle$ is the subgroup supplied by Lemma~\ref{lem:coordinates}, then $J_i=\Isol_{J_i}(J_i^\ast)$, so for such a word $w$ there exist $q\geq1$ and $a_1,\ldots,a_{h_i'}\in\Z$ with
$$
w^q=y_{i,1}^{a_1}\ldots y_{i,h_i'}^{a_{h_i'}}.
$$
Enumerating $q$ and the integer tuples and testing these equalities with the Word Problem in $H$ terminates; the rational power polynomials for $J_i$ then compute the coordinate vector of $w$.

Consequently the coordinates in $J_i$ of the fixed lifts of $t_{i,1},\ldots,t_{i,d_i}$ and of every word $\widehat b_{i,p,j,n}^{W}$ constructed by Lemma~\ref{lem:effectivechains} are computable. That lemma also computes the integers determined by
$$
p^n b_{i,p,j,n}^{W}=\sum_{k=1}^{d_i}c'_{i,p,j,k,n}t_{i,k}.
$$
Hence Lemma~\ref{lem:factorlift} applies to
$$
(\pi_i^W)_p:(J_i)_p\twoheadrightarrow(W_i)_p
$$
with $Y=J_i$, $B=W_i$, $L=T_i$, $q=p$ and coefficient field $F_{i,p}$; its coefficients $a_{j,k}$ are precisely the coefficients $\beta_{i,p,j,k}$ of
$$
v_{i,p,j}^{W}=\sum_{k=1}^{d_i}\beta_{i,p,j,k}(t_{i,k}\otimes1).
$$
Thus, from coefficients $\mu_j\in F_{i,p}$ and $z'_k\in F_{i,p}\cap\Zp$ defining
$$
w=\sum_j\mu_jd_j^W+\sum_k z'_k\ell_k^W\in(W_i)_p,
$$
the lemma computes, in the coordinate system just fixed for $J_i$, an element $\widetilde w\in(J_i)_p$ satisfying
$(\pi_i^W)_p(\widetilde w)=w.$

We finally convert this output to the ambient coordinates of $N_0$, which will be useful below. For each fixed coordinate element $y_{i,\nu}\in J_i\leq N_0$, Lemma~\ref{lem:wordcoords} computes its rational coordinate vector with respect to $u_{0,1},\ldots,u_{0,h_0}$. If $(\xi_1,\ldots,\xi_{h_i'})$ is the coordinate vector of $\widetilde w$ in $(J_i)_p$, then, using these fixed vectors together with the $\Qp$-extensions of the power and multiplication polynomials for $N_0$, we compute the coordinate vector of
$\widetilde w=y_{i,1}^{\xi_1}\ldots y_{i,h_i'}^{\xi_{h_i'}}$
with respect to $u_{0,1},\ldots,u_{0,h_0}$. Thus both kinds of factor lifts required in~\textup{(ii)} are available in coordinate systems compatible with the fixed ambient coordinates of $N_0$. Finally, let $E\leq\Qp$ be a finitely generated exact coefficient field containing $F_{i,p}$, and suppose that the inclusion $F_{i,p}\hookrightarrow E$ is effective with respect to the chosen exact representations. Via this inclusion, all the presentation vectors and all the fixed chain data used above are exactly represented over $E$. Lemmas~\ref{lem:additive} and~\ref{lem:factorlift} therefore apply over $E$, and the coordinate conversions use only the fixed rational power and multiplication polynomials. Hence all the decomposition, membership, preimage and lifting procedures above remain effective after replacing $F_{i,p}$ by $E$. This proves~\textup{(ii)}.
\end{proof}

We shall finally need the following criterion for
membership in abelian minimax groups.

\begin{lemma}\label{lem:ratpadicmembership}
Let $A$ be an abelian minimax group whose elements have a fixed effective representation, in which equality is computable. Let $B$ be a recursively enumerable subgroup of $A$ in this representation. Suppose that the membership problems
\begin{equation}
 a\otimes1\in B_{\Q}=B\otimes_{\mathbb Z}\mathbb Q,
 \label{eq:3.11}
\end{equation}
and, for every prime $p$,
\begin{equation}
 a\in B_p\leq A_p
 \label{eq:3.12}
\end{equation}
where $B_p$ is identified with its image in $A_p$ by
{\rm Lemma~\ref{lem:tensor}\textup{(i)}}, are decidable with one algorithm taking $p$ and $a$ as input. Then membership in $B$
is decidable.
\end{lemma}
\begin{proof}
On input $a\in A$, first perform \eqref{eq:3.11}.  If it fails, return NO.  Suppose therefore that
$a\otimes1\in B_{\Q}.$
Run the following two computations by
alternately executing one step of each unfinished computation:
\begin{enumerate}[label=\textup{(\alph*)}]
\item Enumerate $B$, compare each enumerated element with $a$ using
the decidable equality in $A$, and return YES when equality occurs.

\item Enumerate the primes as $p_1,p_2,\ldots$.  Starting with $p_1$,
run the terminating membership test \eqref{eq:3.12}.  If it returns
NO, return NO; if it returns YES, proceed to the next prime.
\end{enumerate}
From the hypothesis it follows that this is a single procedure.

If $a\in B$, then branch \textup{(a)} eventually encounters $a$, because $B$ is recursively enumerable, and returns YES.  Moreover $a\in B_p$ for every prime $p$, so branch~(b) never returns NO. It remains to prove that the procedure stops when $a\notin B$.  In this case $
C=(\langle a\rangle+B)/B$ is a non-trivial cyclic group generated by $a+B$.

Suppose first that $C$ is infinite.  Then $C\simeq\Z$.  Since $\Q$ is a flat $\Z$-module, tensoring the injection $C\longrightarrow A/B$ with $\Q$ gives an injection
$$
C\otimes_{\Z}\Q
\longrightarrow
(A/B)\otimes_{\Z}\Q.
$$
Hence the image of $a+B$, which is non-zero in $C\otimes_{\Z}\Q$, is non-zero in $(A/B)\otimes_{\Z}\Q$. On the other hand, exactness of tensoring
$$
0\longrightarrow B\longrightarrow A\longrightarrow A/B
\longrightarrow0
$$
with the flat $\Z$-module $\Q$ gives
$$
0\longrightarrow B_{\Q}\longrightarrow A_{\Q}
\longrightarrow (A/B)\otimes_{\Z}\Q
\longrightarrow0.
$$
Thus
$a\otimes1\notin B_{\Q}$,
a contradiction. This means that, whenever $C$ is infinite, the algorithm has already terminated with NO after the first use of \eqref{eq:3.11}.

Therefore in this case $C$ is finite, of order $m>1$. Choose a prime $p\mid m$. Since $C$ is finitely generated, the definition \eqref{eq:tensor-p-definition} gives
$C_p=\widehat C_p$. Writing $m=p^eu$ with $e\geq1$ and $p\nmid u$, the pro-$p$ completion of the cyclic group $C$ is its non-trivial cyclic $p$-primary quotient of order $p^e$.  In particular, the image of the generator $a+B$ in $C_p$ is non-zero. Now, applying again Lemma~\ref{lem:tensor}\textup{(i)} we get an injection
$$
C_p\longrightarrow(A/B)_p,
$$
and hence the image of $a+B$ is non-zero in $(A/B)_p$.

Finally, an application of Lemma~\ref{lem:tensor}\textup{(ii)}
gives
$A_p/B_p\simeq(A/B)_p$. Hence the image of $a$ in $A_p/B_p$ is non-zero, and therefore
$a\notin B_p$. Thus, when branch \textup{(b)} reaches this prime $p$, its terminating membership test returns NO. Since only finitely many primes precede $p$ in the enumeration and every preceding test terminates, branch \textup{(b)} reaches $p$ after finitely many steps.

We have therefore shown that the procedure returns YES in finite time when $a\in B$, while if $a\notin B$ it returns NO either at the initial test or at one of the $p$-adic tests. Hence membership in $B$ is decidable.
\end{proof}

\section{Kernels and images}\label{sec:4}

The purpose of this section is to make effective the computation of kernels and images of the homomorphisms that arise in the successive lifting steps.  Let $Y\leq N_0$ be a subgroup for which $Y_{\Q}$ is given by a finite rational row description (see the definition below), and let $f:Y\to A$ be an effectively computable homomorphism to an abelian minimax group. The main rational result of the section, Lemma~\ref{lem:rationalkernel}, computes a finite rational row description of $(\Ker f)_{\Q}$ and a basis of the rational image $f(Y)_{\Q}$.  Lemma~\ref{lem:rationalrow} provides the rational row operations needed for this computation, while Lemma~\ref{lem:transport} shows that the resulting membership and rational row data can be transported effectively under conjugation in $H$.

The rational computation alone does not determine membership in the actual image $f(Y)$.  For the applications below we must also decide, for every prime $p$, membership in its tensor $p$-completion. Lemma~\ref{lem:padicproduct} describes the completions of the fixed products $N_iK_0$, and Lemma~\ref{lem:fixedcharts} gives explicit parameterizations of the $p$-completions of the fixed subgroups occurring in the lifting argument.  These parameterizations convert the required membership conditions into finite systems of polynomial equations over $\Qp$, whose solvability is decided uniformly in $p$ by Lemma~\ref{lem:padicdecision}.

\begin{definition}
Let $x\in(N_0)_{\Q}\setminus\{1\}$. With respect to the coordinate system $u_{0,1},\ldots,u_{0,h_0}$ fixed in Section~\ref{sec:3} for $N_0$, write $x=u_{0,1}^{\xi_1(x)}\ldots u_{0,h_0}^{\xi_{h_0}(x)}$, where $\xi_j(x)\in\Q$. The \emph{pivot index} of $x$ is $\nu(x)=\min\{j:\xi_j(x)\neq0\}$.  We say that $x$ is \emph{normalized} if $\xi_{\nu(x)}(x)=1$. A normalized element of $(N_0)_{\Q}$ will be called a \emph{rational row}.

A \emph{finite rational row description} of a rational subgroup $R\leq(N_0)_{\Q}$ is an ordered tuple $b_1,\ldots,b_s\in R$ such that
$$
\nu(b_1)<\nu(b_2)<\ldots<\nu(b_s),
\qquad
\xi_{\nu(b_j)}(b_j)=1
\quad (1\leq j\leq s),
$$
and every $r\in R$ has a unique expression
\begin{equation}
r=b_1^{q_1}\ldots b_s^{q_s},
\qquad q_1,\ldots,q_s\in\Q.
\end{equation}
The coordinate vector of each $b_j$ with respect to
$u_{0,1},\ldots,u_{0,h_0}$ is part of the finite description.
\end{definition}

\begin{remark}\label{rem:coordinate-tails}
We record a property of the coordinate system $u_{0,1},\ldots,u_{0,h_0}$ which will be used below. In the construction of Lemma~\ref{lem:coordinates}, each torsion-free abelian factor of the central series can be refined, by choosing a basis of that factor, to a series with infinite cyclic central factors. Thus the coordinate elements may be chosen so that, for $j=1,\ldots,h_0$, the subgroup $N_0^{\ast}(j)=\langle u_{0,j},\ldots,u_{0,h_0}\rangle$ is normal in $N_0^{\ast}$ and $[N_0^{\ast},N_0^{\ast}(j)]\leq N_0^{\ast}(j+1)$, where $N_0^{\ast}(h_0+1)=1$. Put $Q_j=\Isol_{(N_0)_{\Q}}(N_0^{\ast}(j))$. By the unique coordinate expressions in Lemma~\ref{lem:coordinates},
$$
Q_j=\{x\in(N_0)_{\Q}:
\xi_1(x)=\ldots=\xi_{j-1}(x)=0\},
$$
and $Q_j/Q_{j+1}$ is the one-dimensional $\Q$-vector space generated by $u_{0,j}Q_{j+1}$. Moreover, $[(N_0)_{\Q},Q_j]\leq Q_{j+1}$, since the corresponding inclusion holds before passing to the rational completion. Consequently, let $y,z\in(N_0)_{\Q}$ and suppose that $\xi_1(y)=\ldots=\xi_{j-1}(y)=0$. Then $y\in Q_j$, so $yz$ and $z$ have the same image in $(N_0)_{\Q}/Q_j$ and therefore have the same coordinates of indices smaller than $j$. Modulo $Q_{j+1}$, the subgroup $Q_j/Q_{j+1}$ is central, and its coordinate with respect to $u_{0,j}Q_{j+1}$ is additive. Hence, for every $\lambda\in\Q$,
$$
\begin{array}{c}
\xi_\ell(yz)=\xi_\ell(z)\quad(\ell<j),\qquad
\xi_j(yz)=\xi_j(y)+\xi_j(z),\\[0.2cm]
\xi_\ell(y^\lambda)=0\quad(\ell<j),\qquad
\xi_j(y^\lambda)=\lambda\xi_j(y).
\end{array}
$$
\end{remark}

\begin{remark}\label{rem:rational-reduction}
We shall also use the following reduction procedure. Let $B=(b_1,\ldots,b_s)$ be an ordered tuple of rational rows with strictly increasing pivot indices $p_k=\nu(b_k)$. Given $x\in(N_0)_{\Q}$, put $z_0=x$ and, for $k=1,\ldots,s$, set
\begin{equation}
z_k=b_k^{-\xi_{p_k}(z_{k-1})}z_{k-1}.
\label{eq:rational-reduction}
\end{equation}
Since $b_k$ has zero coordinates before $p_k$ and $\xi_{p_k}(b_k)=1$, Remark~\ref{rem:coordinate-tails} gives $\xi_\ell(z_k)=\xi_\ell(z_{k-1})$ for $\ell<p_k$ and $\xi_{p_k}(z_k)=0$. Put $\rho_B(x)=z_s$, with $\rho_B(x)=x$ if $s=0$. Its coordinate vector is computable from that of $x$ by the fixed rational power and multiplication polynomials for~$N_0$. If $\alpha_k=\xi_{p_k}(z_{k-1})$, the equalities defining the $z_k$ give
$$
x=b_1^{\alpha_1}\ldots b_s^{\alpha_s}\rho_B(x).
$$
Thus the reduction procedure computes both $\rho_B(x)$ and the rational numbers $\alpha_1,\ldots,\alpha_s$.
\end{remark}

\begin{remark}\label{rem:rational-reduction-properties}
If $B$ is a finite rational row description of a rational subgroup $R\leq(N_0)_{\Q}$, then
$$
x\in R\quad\Longleftrightarrow\quad\rho_B(x)=1.
$$
Indeed, if $x=b_1^{q_1}\ldots b_s^{q_s}\in R$, then Remark~\ref{rem:coordinate-tails} shows successively that the numbers $\alpha_k$ in Remark~\ref{rem:rational-reduction} satisfy $\alpha_k=q_k$ for $k=1,\ldots,s$, and hence $\rho_B(x)=1$. Conversely, if $\rho_B(x)=1$, the expression for $x$ in Remark~\ref{rem:rational-reduction} shows that $x\in R$. Moreover, if the coordinates of an element $x(\lambda)$ are rational polynomials in $\lambda$, then the same is true, inductively from \eqref{eq:rational-reduction}, for the coordinates of every $z_k(\lambda)$ and hence of $\rho_B(x(\lambda))$.
\end{remark}

\medskip

The following technical observation will ensure that the rational subgroups produced by the row-reduction procedure remain stable under the normalizer operations occurring below.

\begin{lemma}\label{lem:rational-normalizer}
Let $T\leq(N_0)_{\Q}$ be a rational subgroup admitting a finite rational row description. Then $N_{(N_0)_{\Q}}(T)$ is a rational subgroup of $(N_0)_{\Q}$.
\end{lemma}
\begin{proof}
Let $C=(c_1,\ldots,c_t)$ be a finite rational row description of $T$, and put \hbox{$p_i=\nu(c_i)$.} Let $g\in(N_0)_{\Q}$ and suppose that $g^m\in N_{(N_0)_{\Q}}(T)$ for some $m\geq1$. We prove that $g\in N_{(N_0)_{\Q}}(T)$.

Fix $i\in\{1,\ldots,t\}$. For $\lambda\in\Q$, consider $c_i^{g^\lambda}=g^{-\lambda}c_i g^\lambda$. By the fixed multiplication and rational power polynomials for $N_0$, every coordinate of $c_i^{g^\lambda}$ is a rational polynomial in $\lambda$. By Remark~\ref{rem:rational-reduction-properties}, the same is therefore true of every coordinate of $\rho_C(c_i^{g^\lambda})$.

For every $n\in\mathbb Z$, the element $g^{mn}$ normalizes $T$. Hence $c_i^{g^{mn}}\in T$, and Remark~\ref{rem:rational-reduction-properties} gives $\rho_C(c_i^{g^{mn}})=1$. Thus every coordinate of $\rho_C(c_i^{g^\lambda})$ vanishes at all points $\lambda=mn$, $n\in\mathbb Z$. Since these are infinitely many rational numbers, each coordinate polynomial is identically zero. In particular,
$$
\rho_C(c_i^g)=\rho_C(c_i^{g^{-1}})=1.
$$
Another application of Remark~\ref{rem:rational-reduction-properties} gives $c_i^g,c_i^{g^{-1}}\in T$ for every $i$.

Now every element of $T$ has the form $c_1^{q_1}\ldots c_t^{q_t}$ with $q_1,\ldots,q_t\in\Q$. Conjugation preserves rational powers in the uniquely divisible group $(N_0)_{\Q}$, and therefore $T^g\leq T$ and $T^{g^{-1}}\leq T$. Hence $T^g=T$, so $g\in N_{(N_0)_{\Q}}(T)$. Therefore $N_{(N_0)_{\Q}}(T)$ is rational.
\end{proof}

We can now turn the preceding reduction procedure into effective row operations for rational subgroups of $(N_0)_{\Q}$.

\begin{lemma}\label{lem:rationalrow}
There are terminating procedures with the following properties.
\begin{enumerate}[label=\textup{(\roman*)}]
\item Given finitely many coordinate vectors, with respect to $u_{0,1},\ldots,u_{0,h_0}$, representing a finite set $S\subseteq(N_0)_{\Q}$, one can compute a finite rational row description of $R=\Isol_{(N_0)_{\Q}}(\langle S\rangle)$.

\item Given a finite rational row description $b_1,\ldots,b_s$ of a rational subgroup $R\leq(N_0)_{\Q}$, one can compute a finite rational row description of~$D_R=\Isol_R([R,R])$, a basis of the finite-dimensional $\Q$-vector space $R/D_R$, and the matrix of the quotient map $\pi_R:R\to R/D_R$ with respect to the row coordinates determined by $b_1,\ldots,b_s$ and the computed basis of~$R/D_R$.

\item Given the data computed in \textup{(ii)} and a subspace $W\leq R/D_R$ by its coordinates with respect to the computed basis of $R/D_R$, one can compute a finite rational row description of $\pi_R^{-1}(W)$.
\end{enumerate}
\end{lemma}

\begin{proof}
We basically use the reduction procedure $\rho_B$ of Remark~\ref{rem:rational-reduction}.

\smallskip

For \textup{(i)}, start with $B$ empty and place the finitely many elements of $S$ in a finite list. At any stage, let $B=(b_1,\ldots,b_s)$ be the rational rows which have been inserted so far, where $s$ may be zero. Take an element $x$ from the list and compute $r=\rho_B(x)$. If $r=1$, make no change. If $r\neq1$, put $p=\nu(r)$ and $b=r^{1/\xi_p(r)}$. The fixed rational power polynomials compute the coordinate vector of $b$, and Remark~\ref{rem:coordinate-tails} gives $\xi_p(b)=1$ and $\xi_\ell(b)=0$ for $\ell<p$. If $s>0$, Remark~\ref{rem:rational-reduction} gives $\xi_{\nu(b_k)}(r)=0$ for every $k=1,\ldots,s$. Since $\xi_p(r)\neq0$, the index $p$ is different from $\nu(b_1),\ldots,\nu(b_s)$. If $s=0$, there are no existing pivot indices. Thus in either case $p$ is a new pivot index.

Insert the new rational row $b$ into $B$ so that the pivot indices remain strictly increasing. Put $x$ back into the finite list. For every other rational row $c$ currently in $B$, also put $[b,c]$ and $[c,b]$ in the list. Their coordinate vectors are computable from those of $b$ and $c$ using the fixed multiplication and rational power polynomials for $N_0$.

Each time a new rational row is inserted into $B$, its pivot index is different from the pivot indices of all rational rows already in $B$. Since the possible pivot indices are $1,\ldots,h_0$, at most $h_0$ rational rows can be inserted into~$B$. The initial list of elements to be processed is finite. Each time a new rational row is inserted into $B$, only finitely many elements are added to this list. Since only finitely many rational rows can be inserted into $B$, only finitely many elements are added to the list during the whole procedure. After the last insertion into $B$, processing an element can no longer produce a non-trivial value of $\rho_B$, since this would produce a new rational row with a new pivot index. Hence, from that point on, each step removes one element from the list and adds none. Therefore the list is eventually empty and the procedure terminates.

Let $B=(b_1,\ldots,b_s)$ be the final ordered tuple of rational rows and put $p_k=\nu(b_k)$. Thus $p_1<\ldots<p_s$. Define
$$
R_B=\{b_1^{q_1}\ldots b_s^{q_s}:q_1,\ldots,q_s\in\Q\}.
$$
We first prove uniqueness of these expressions. Suppose $x=b_1^{q_1}\ldots b_s^{q_s}$. Remark~\ref{rem:coordinate-tails} shows that the $p_1$-th coordinate of $x$ is $q_1$, because every $b_k$ with~\hbox{$k>1$} has zero coordinates before $p_k$ and hence zero $p_1$-th coordinate. Remark~\ref{rem:rational-reduction} then shows that the first reduction step removes $b_1^{q_1}$. Repeating the same argument at~\hbox{$p_2,\ldots,p_s$} recovers successively $q_2,\ldots,q_s$. Thus the exponent tuple $(q_1,\ldots,q_s)$ is determined by $x$, and the expression is unique.

We now show that $R_B$ is a rational subgroup. We first note the following consequence of the construction. Let $i<j$. Since $b_j\in Q_{p_j}$, Remark~\ref{rem:coordinate-tails} gives
$$
[b_i,b_j],[b_j,b_i]\in Q_{p_j+1}.
$$
Both commutators were placed in the finite list when the later of $b_i$ and $b_j$ was inserted. Hence, when the procedure terminates,
$$
\rho_B([b_i,b_j])=\rho_B([b_j,b_i])=1.
$$
The decomposition in Remark~\ref{rem:rational-reduction} therefore expresses each of these commutators as an ordered product of rational powers of the final rational rows (in particular, every commutator of two final rational rows belongs to $R_B$). We claim that only $b_{j+1},\ldots,b_s$ can occur with non-zero exponent in such an expression. Suppose, to the contrary, that one of $b_1,\ldots,b_j$ occurs with non-zero exponent, and let $b_\ell$ be the first such rational row. Since all preceding exponents are zero, Remark~\ref{rem:coordinate-tails} shows that the $p_\ell$-th coordinate of the product is precisely the non-zero exponent of $b_\ell$. On the other hand, the commutator belongs to $Q_{p_j+1}$, so all its coordinates of index at most $p_j$ are zero. Since $p_\ell\leq p_j$, this is a contradiction. Thus both $[b_i,b_j]$ and $[b_j,b_i]$ are ordered products of rational powers of $b_{j+1},\ldots,b_s$.

For $k=1,\ldots,s$, put $R_k=\{b_k^{q_k}\ldots b_s^{q_s}:q_k,\ldots,q_s\in\Q\}$, and put \hbox{$R_{s+1}=\{1\}$.} We prove, by descending induction on $k$, that $R_k$ is a rational subgroup of $(N_0)_{\Q}$. Assume that $R_{k+1}$ is a rational subgroup. For every $j>k$, the conclusion above gives $[b_j,b_k]\in R_{j+1}\leq R_{k+1}$. Hence $b_j^{b_k}=b_j[b_j,b_k]\in R_{k+1}$, and therefore $R_{k+1}^{b_k}\leq R_{k+1}$. We prove the reverse inclusion. We show, by descending induction on $j=s,\ldots,k+1$, that $b_j\in R_{k+1}^{b_k}$. For $j=s$, one has $[b_s,b_k]\in R_{s+1}=\{1\}$, and hence $b_s^{b_k}=b_s$. Now suppose that $b_{j+1},\ldots,b_s\in R_{k+1}^{b_k}$. The subgroup $R_{k+1}^{b_k}$ is rational, being the conjugate of the rational subgroup $R_{k+1}$. It therefore contains $R_{j+1}$, since $R_{j+1}$ consists of the ordered products of rational powers of $b_{j+1},\ldots,b_s$. Put $c_j=[b_j,b_k]\in R_{j+1}$. Since $b_j^{b_k}=b_jc_j$, both $b_j^{b_k}$ and $c_j$ belong to $R_{k+1}^{b_k}$, and hence $b_j=b_j^{b_k}c_j^{-1}\in R_{k+1}^{b_k}$. Thus $R_{k+1}\leq R_{k+1}^{b_k}$, and consequently $R_{k+1}^{b_k}=R_{k+1}$. Hence $b_k$ normalizes $R_{k+1}$.

Put $N_k=N_{(N_0)_{\Q}}(R_{k+1})$. By Lemma~\ref{lem:rational-normalizer}, $N_k$ is rational. Since $b_k\in N_k$, it follows that $b_k^q\in N_k$ for every $q\in\Q$. Thus every rational power of $b_k$ normalizes $R_{k+1}$. Since the rational powers of $b_k$ form a subgroup, it follows that
$$
R_k=\{b_k^q t:q\in\Q,\ t\in R_{k+1}\}
$$
is a subgroup of $(N_0)_{\Q}$.

We now prove that $R_k$ is rational. Let $x\in(N_0)_{\Q}$ and suppose that $x^m\in R_k$ for some $m\geq1$. Since $R_k\leq N_k$ and $N_k$ is rational, $x^m\in N_k$ implies $x\in N_k$. The subgroup $R_{k+1}$ is normal in $N_k$, so the quotient $N_k/R_{k+1}$ is uniquely divisible. Write $x^m=b_k^q t$ with $q\in\Q$ and $t\in R_{k+1}$. In~$N_k/R_{k+1}$ one has
$$
(xR_{k+1})^m=(b_kR_{k+1})^q.
$$
Unique divisibility therefore gives
$$
xR_{k+1}=(b_kR_{k+1})^{q/m},
$$
and hence $x\in b_k^{q/m}R_{k+1}\subseteq R_k$. Thus $R_k$ is rational.

Therefore induction yields that $R_B=R_1$ is a rational subgroup of $(N_0)_{\Q}$. Together with the uniqueness of the expressions $b_1^{q_1}\ldots b_s^{q_s}$ proved above, this shows that the final rational rows give a finite rational row description of~$R_B$.

It remains to identify $R_B$ with $I:=\Isol_{(N_0)_{\Q}}(\langle S\rangle)$. We prove inductively that every rational row inserted into $B$, and every element subsequently placed in the finite list, belongs to $I$. The elements initially placed in the finite list are precisely the elements of $S$, so they belong to $I$. Suppose that all rational rows inserted so far belong to $I$ and that an element $x\in I$ is processed. Since $\rho_B(x)$ is obtained from $x$ by multiplication by rational powers of these rational rows, $\rho_B(x)\in I$. If $\rho_B(x)\neq1$, then the new rational row $\rho_B(x)^{1/\xi_{\nu(\rho_B(x))}(\rho_B(x))}$ also belongs to $I$, because $I$ is rational. The commutators subsequently placed in the finite list belong to $I$, since their entries belong to $I$. Hence every final rational row belongs to $I$, and therefore $R_B\leq I$.

Conversely, whenever an element $x\in S$ is removed from the list without being reinserted, it satisfies $\rho_B(x)=1$ for the tuple $B$ present at that stage. Remark~\ref{rem:rational-reduction} therefore expresses $x$ as an ordered product of rational powers of those rational rows, all of which belong to the final tuple. Hence $S\subseteq R_B$. Since $R_B$ is rational, $I\leq R_B$. Consequently $R_B=\Isol_{(N_0)_{\Q}}(\langle S\rangle)$, and the final rational rows give the required finite rational row description. This proves \textup{(i)}.

For \textup{(ii)}, let $r_1,\ldots,r_s$ be the given finite rational row description of~$R$. Start with an empty ordered tuple $B_D$ of rational rows and place all commutators $[r_j,r_k]$ in a finite list. At any stage, take an element $x$ from this finite list and compute $\rho_{B_D}(x)$. If $\rho_{B_D}(x)=1$, make no change. If $v=\rho_{B_D}(x)\neq1$, put $p=\nu(v)$ and $d=v^{1/\xi_p(v)}$. Insert the rational row $d$ into $B_D$ so that the pivot indices remain strictly increasing, and put $x$ back into the finite list. For every $j=1,\ldots,s$, also put $[d,r_j]$ and $[r_j,d]$ in the finite list. For every other rational row $e$ currently in $B_D$, put $[d,e]$ and $[e,d]$ in the finite list. Each time a new rational row is inserted into $B_D$, its pivot index differs from the pivot indices of all rational rows already in $B_D$. Thus at most $h_0$ rational rows can be inserted. Only finitely many elements are added to the finite list at each such insertion. After the last insertion, a non-trivial value of $\rho_{B_D}$ would produce a further rational row, which is impossible. Hence the finite list is eventually empty and the procedure terminates.

Let $d_1,\ldots,d_t$ be the final rational rows, in increasing order of pivot index, and let $q_\ell=\nu(d_\ell)$. Put
$$
D=\{d_1^{a_1}\ldots d_t^{a_t}:a_1,\ldots,a_t\in\Q\}.
$$
The pivot argument used in \textup{(i)} gives uniqueness of these expressions.

If $\ell<k$, then
$$
[d_\ell,d_k],[d_k,d_\ell]\in Q_{q_k+1}
$$
by Remark~\ref{rem:coordinate-tails}. These commutators were processed and therefore have value $1$ under the final map $\rho_{B_D}$. By Remark~\ref{rem:rational-reduction}, each of them is therefore expressed as an ordered product of rational powers of the final rational rows, and the coordinate-tail argument of Remark~\ref{rem:coordinate-tails} shows that only $d_{k+1},\ldots,d_t$ can occur with non-zero exponent. Applying the descending induction proved in \textup{(i)}, now to $d_1,\ldots,d_t$, shows that $D$ is a rational subgroup and that $d_1,\ldots,d_t$ is a finite rational row description of $D$.

We next prove that $D\trianglelefteq R$. Fix $r_j$. When the rational row $d_\ell$ was inserted into $B_D$, the commutators $[d_\ell,r_j]$ and $[r_j,d_\ell]$ were placed in the finite list. Since the procedure terminates with this list empty, both commutators have value $1$ under the final map $\rho_{B_D}$. Remark~\ref{rem:rational-reduction} therefore expresses each of them as an ordered product of rational powers of $d_1,\ldots,d_t$. Since $d_\ell\in Q_{q_\ell}$, Remark~\ref{rem:coordinate-tails} gives
$$
[d_\ell,r_j],[r_j,d_\ell]\in Q_{q_\ell+1}.
$$
Therefore, in the expressions above, the exponents of $d_1,\ldots,d_\ell$ are zero. Indeed, if one of these rational rows occurred with non-zero exponent, let $d_k$ be the first such rational row. By Remark~\ref{rem:coordinate-tails}, the coordinate at $q_k$ of the corresponding ordered product would then be non-zero, whereas every element of $Q_{q_\ell+1}$ has zero coordinates at all indices at most $q_\ell$. Since $q_k\leq q_\ell$, this is impossible. Thus $[d_\ell,r_j]$ and $[r_j,d_\ell]$ are ordered products of rational powers of $d_{\ell+1},\ldots,d_t$. Consequently, for each~$\ell$ one has
$$
d_\ell^{r_j}=d_\ell c_\ell
$$
for some
$$
c_\ell\in\{d_{\ell+1}^{a_{\ell+1}}\ldots d_t^{a_t}:a_{\ell+1},\ldots,a_t\in\Q\}.
$$
Going backwards from $\ell=t$ to $\ell=1$, exactly as above, shows that conjugation by $r_j$ maps $D$ onto $D$. Thus $r_j$ normalizes $D$. Since the normalizer of $D$ is rational, every rational power $r_j^q$ also normalizes $D$. Every element of $R$ has an expression $r_1^{a_1}\ldots r_s^{a_s}$ with rational exponents, so $D\trianglelefteq R$.

Every initial commutator $[r_j,r_k]$ belongs to $D$. Hence the images of $r_1,\ldots,r_s$ commute in $R/D$. Their rational powers generate $R/D$, so $R/D$ is abelian. Therefore $[R,R]\leq D$. Since $D$ is rational in $R$, it follows that $D_R=\Isol_R([R,R])\leq D$.

For the reverse inclusion, first note that $D_R$ is normal in $R$. Indeed, if $x^m\in[R,R]$ and $g\in R$, then $(x^g)^m=(x^m)^g\in[R,R]$, and hence $x^g\in D_R$. Every element in the initial finite list belongs to $D_R$. Suppose that all rational rows inserted so far belong to $D_R$ and that an element $x\in D_R$ is processed. Since the computation of $\rho_{B_D}(x)$ uses only rational powers of rational rows in $D_R$, one has $\rho_{B_D}(x)\in D_R$. If this element is non-trivial, its normalized root belongs to $D_R$, because $D_R$ is rational. Every commutator subsequently placed in the finite list belongs to $D_R$: commutators with the $r_j$ belong to $D_R$ by normality, and commutators between rational rows already in $D_R$ belong to $D_R$ because it is a subgroup. Thus every final rational row $d_\ell$ belongs to $D_R$, and hence $D\leq D_R$. Therefore $D=D_R$.

It remains to compute a basis of $R/D_R$ and the matrix of $\pi_R$. Start with the finite rational row description $d_1,\ldots,d_t$ of $D_R$, ordered by increasing pivot index, and denote this ordered tuple by $B$. Put $r_1,\ldots,r_s$ in a finite list. Process this finite list by the procedure of \textup{(i)}: if $x$ has non-trivial value $v=\rho_B(x)$, insert the rational row $v^{1/\xi_{\nu(v)}(v)}$ into $B$, put $x$ back into the finite list, and add the commutators of the new rational row with every other rational row currently in $B$.

All rational rows inserted in this procedure belong to $R$. Indeed, this is true initially for the $d_\ell$, and whenever a new rational row is produced, it is obtained from an element of $R$ by multiplication and rational powers of elements of $R$; the rationality of $R$ therefore keeps it inside $R$. Moreover, every commutator placed in the finite list belongs to $[R,R]\leq D_R$. The same pivot argument and descending induction as in \textup{(i)} therefore show, at termination, that the final ordered tuple of rational rows is a finite rational row description of a rational subgroup of $R$.

Let $c_1,\ldots,c_d$ be the rational rows inserted in addition to $d_1,\ldots,d_t$, numbered in increasing order of their pivot indices. Whenever an original rational row $r_j$ is removed from the list without being reinserted, it has value $1$ under the reduction map associated with the tuple $B$ present at that stage, and hence belongs to the rational subgroup described by that tuple. This subgroup is contained in the rational subgroup described by the final tuple. Thus every $r_j$ belongs to the rational subgroup described by the final rational rows. Since $r_1,\ldots,r_s$ is a finite rational row description of $R$, this subgroup contains $R$. The reverse inclusion was proved above. Hence the final ordered tuple obtained by combining the $d_\ell$ and the $c_\nu$ in increasing order of pivot index is a finite rational row description of $R$.

The cosets $c_1D_R,\ldots,c_dD_R$ form a basis of the $\Q$-vector space $R/D_R$. They span because every element of $R$ has an expression in the combined finite rational row description, and every $d_\ell$ maps to zero in the quotient. For linear independence, suppose that $c_1^{a_1}\ldots c_d^{a_d}\in D_R$. Express this element using $d_1,\ldots,d_t$. We then have two expressions of the same element with respect to the combined finite rational row description of $R$: in one, every exponent belonging to a rational row $d_\ell$ is zero, and in the other every exponent belonging to a rational row $c_\nu$ is zero. Uniqueness gives $a_1=\ldots=a_d=0$.

For each $j=1,\ldots,s$, compute the unique expression of $r_j$ with respect to the final finite rational row description of $R$. Let $e_1,\ldots,e_{t+d}$ denote the tuple obtained by combining $d_1,\ldots,d_t$ and $c_1,\ldots,c_d$ in increasing order of pivot index. Thus there are unique rational numbers $\beta_{j,1},\ldots,\beta_{j,t+d}$ such that

$$
r_j=e_1^{\beta_{j,1}}\ldots e_{t+d}^{\beta_{j,t+d}}.
$$

For $\nu=1,\ldots,d$, let $\lambda_{j,\nu}$ be the exponent $\beta_{j,k}$ corresponding to the rational row $e_k=c_\nu$. Since $d_1,\ldots,d_t$ belong to $D_R$, their images under $\pi_R:R\to R/D_R$ are trivial. Therefore

$$
\pi_R(r_j)=\sum_{\nu=1}^d\lambda_{j,\nu}(c_\nu D_R).
$$

Hence $(\lambda_{j,1},\ldots,\lambda_{j,d})^{\mathsf T}$ is the coordinate vector of $\pi_R(r_j)$ with respect to the computed basis $c_1D_R,\ldots,c_dD_R$ of $R/D_R$.

The coordinates of an element of $R$ with respect to the given finite rational row description $r_1,\ldots,r_s$ are its exponents in the unique expression $r_1^{q_1}\ldots r_s^{q_s}$. Thus the image of the $j$-th coordinate element $r_j$ is represented by the column vector $(\lambda_{j,1},\ldots,\lambda_{j,d})^{\mathsf T}$. Taking these vectors as columns therefore gives the matrix of $\pi_R$ with respect to the coordinates determined by $r_1,\ldots,r_s$ and the basis $c_1D_R,\ldots,c_dD_R$ of $R/D_R$.

\smallskip

For \textup{(iii)}, compute from the given coordinates a basis $\bar w_1,\ldots,\bar w_t$ of $W$. Write $\bar w_j=\sum_{\nu=1}^d\mu_{j,\nu}(c_\nu D_R)$, with $\mu_{j,\nu}\in\Q$, and compute $w_j=c_1^{\mu_{j,1}}\ldots c_d^{\mu_{j,d}}$. Then $\pi_R(w_j)=\bar w_j$. Put $L=\langle D_R,w_1,\ldots,w_t\rangle$. We claim that
\begin{equation}
\pi_R^{-1}(W)=\Isol_R(L).
\label{eq:inverse-image-rational}
\end{equation}
If $x\in\pi_R^{-1}(W)$, write $\pi_R(x)=\sum_j a_j\bar w_j$ with $a_j\in\Q$, and choose $m\geq1$ such that $ma_j\in\mathbb Z$ for every $j$. Then $\pi_R(x^m)$ is an integral linear combination of $\bar w_1,\ldots,\bar w_t$. Hence there are integers $n_1,\ldots,n_t$ such that $\pi_R(x^m)=\pi_R(w_1^{n_1}\ldots w_t^{n_t})$. Therefore $x^m(w_1^{n_1}\ldots w_t^{n_t})^{-1}\in D_R$, so $x^m\in L$ and $x\in\Isol_R(L)$.

Conversely, $\pi_R^{-1}(W)$ is rational in $R$. Indeed, if $x^m\in\pi_R^{-1}(W)$, then $m\pi_R(x)=\pi_R(x^m)\in W$. Since $W$ is a $\Q$-subspace, $\pi_R(x)\in W$. Thus $x\in\pi_R^{-1}(W)$. Since $\pi_R^{-1}(W)$ contains $D_R,w_1,\ldots,w_t$, it contains $\Isol_R(L)$. This proves \eqref{eq:inverse-image-rational}.

Finally, put $L_0=\langle d_1,\ldots,d_t,w_1,\ldots,w_t\rangle$. Since $d_1,\ldots,d_t$ is a finite rational row description of $D_R$, one has $\Isol_R(L_0)=\Isol_R(L)$. Moreover, $R$ is rational in $(N_0)_{\Q}$. Hence, if $x\in(N_0)_{\Q}$ and $x^m\in L_0\leq R$ for some $m\geq1$, then $x\in R$, and therefore $\Isol_R(L_0)=\Isol_{(N_0)_{\Q}}(L_0)$. Applying \textup{(i)} to the coordinate vectors of $d_1,\ldots,d_t,w_1,\ldots,w_t$ computes a finite rational row description of $\pi_R^{-1}(W)$.
\end{proof}

We now apply these row operations to the kernel and image problem announced at the beginning of the section.

\begin{lemma}\label{lem:rationalkernel}
Let $Y$ be a recursively enumerable torsion-free nilpotent minimax subgroup of $N_0$, and suppose that a finite rational row description of $Y_{\Q}$ is given. Let $f:Y\longrightarrow A$ be a homomorphism to an abelian minimax group $A$. Assume that, given the coordinate vector of an element $y\in Y$ with respect to $u_{0,1},\ldots,u_{0,h_0}$, one can compute $f(y)$ and the coordinates of $f(y)\otimes1$ with respect to a fixed basis of $A_{\Q}$.

Then one can compute a finite rational row description of $(\Ker f)_{\Q}\leq Y_{\Q}$ and a basis of the $\Q$-vector space $f(Y)_{\Q}\leq A_{\Q}$. Moreover, given an element $a\in f(Y)_{\Q}$ by its coordinates with respect to the fixed basis of $A_{\Q}$, one can compute the coordinate vector, with respect to $u_{0,1},\ldots,u_{0,h_0}$, of an element $\widetilde a\in Y_{\Q}$ whose image in $A_{\Q}$ is $a$.
\end{lemma}
\begin{proof}
The homomorphism
$y\mapsto f(y)\otimes1$ from $Y$ to $A_{\Q}$ extends uniquely to a homomorphism $f_{\Q}:Y_{\Q}\longrightarrow A_{\Q}$. We first identify its kernel. Under the fixed embedding of the rational completion of $\Ker f$ in $Y_{\Q}$, one has
$$
(\Ker f)_{\Q}
=
\Isol_{Y_{\Q}}(\Ker f).
$$
Since $A_{\Q}$ is torsion-free, $\Ker f_{\Q}$ is isolated in $Y_{\Q}$ and contains $\Ker f$.  Hence $(\Ker f)_{\Q}\leq\Ker f_{\Q}$. Conversely, let $x\in\Ker f_{\Q}$. Since $Y_{\Q}=\Isol_{Y_{\Q}}(Y)$, there is $m\geq1$ such that $x^m\in Y$. Then $$ f(x^m)\otimes1 = f_{\Q}(x^m) = m f_{\Q}(x) = 0. $$
Thus $f(x^m)$ is a torsion element of $A$.  Hence there is $n\geq1$ such that $f(x^{mn})=0$, so $x^{mn}\in\Ker f$.  Therefore $x\in(\Ker f)_{\Q}$.  We have proved
\begin{equation}
\Ker f_{\Q}=(\Ker f)_{\Q}.
\label{eq:4.6}
\end{equation}

We now claim that $\im f_{\Q}=f(Y)_{\Q}$. The left-to-right inclusion is immediate, as $\im f_{\Q}$ is a $\Q$-subspace of $A_{\Q}$ containing $f(Y)\otimes1$. Conversely, if $x\in Y_{\Q}$, choose $m\geq1$ such that $x^m\in Y$.  Then
$$
f_{\Q}(x)=\frac1m\bigl(f(x^m)\otimes1\bigr),
$$
so $f_{\Q}(x)$ belongs to the $\Q$-span of $f(Y)\otimes1$.

Put now $D_Y=\Isol_{Y_{\Q}}([Y_{\Q},Y_{\Q}])$. Since $A_{\Q}$ is abelian, $[Y_{\Q},Y_{\Q}]\leq\Ker f_{\Q}$.  The subgroup $\Ker f_{\Q}$ is isolated, and therefore $D_Y\leq\Ker f_{\Q}$. Consequently $f_{\Q}$ induces a unique $\Q$-linear map
$$
\overline f_{\Q}:Y_{\Q}/D_Y\longrightarrow A_{\Q}
$$
such that
$f_{\Q}=\overline f_{\Q}\circ\pi_Y$, where
$\pi_Y:Y_{\Q}\to Y_{\Q}/D_Y$ is canonical epimorphism.  By~\eqref{eq:4.6},
\begin{equation}
(\Ker f)_{\Q}
=
\pi_Y^{-1}(\Ker\overline f_{\Q})\qquad\text{ and }
\qquad
f(Y)_{\Q}
=
\operatorname{Im}(\overline f_{\Q}).
\label{eq:4.7}
\end{equation}

Apply Lemma~\ref{lem:rationalrow}\textup{(ii)} to the given finite rational row description of $Y_{\Q}$.  This computes a finite rational row description of $D_Y$, a basis of the finite-dimensional $\Q$-vector space $V=Y_{\Q}/D_Y,$ and the matrix of $\pi_Y$ with respect to the coordinates determined by the given finite rational row description of~$Y_{\Q}$ and the computed basis of $V$. Clearly, the subset $\pi_Y(Y)$ spans $V$. 

Let $d=\dim_{\Q}V$.  Enumerate the elements of $Y$.  Maintain a list $y_1,\ldots,y_k$, initially empty.  Whenever a new element $y\in Y$ is enumerated, compute the coordinates of $\pi_Y(y)$ in the computed basis of $V$.  If $\pi_Y(y)$ does not belong to the $\Q$-span of $\pi_Y(y_1),\ldots,\pi_Y(y_k)$, append $y$ to the list; otherwise discard it.  Stop when $k=d$. The procedure terminates because $\pi_Y(Y)$ spans $V$. At termination, $\pi_Y(y_1),\ldots,\pi_Y(y_d)$ is a basis of $V$.

For each $j$, the coordinates of $f(y_j)\otimes1$ in the fixed basis of $A_{\Q}$ are computable by hypothesis. Since
$$
\overline f_{\Q}(\pi_Y(y_j))=f(y_j)\otimes1,
\qquad
1\leq j\leq d,
$$
these coordinate vectors determine the matrix of $\overline f_{\Q}$ with respect to the basis $\pi_Y(y_1),\ldots,\pi_Y(y_d)$ of $V$ and the fixed basis of $A_{\Q}$.

Rational Gaussian elimination applied to this matrix computes a basis of $\Ker\overline f_{\Q}$, expressed in the computed basis of $V$, and a basis of $\operatorname{Im}\overline f_{\Q}$, expressed in the fixed basis of $A_{\Q}$.  By \eqref{eq:4.7}, the latter is a basis of $f(Y)_{\Q}$.  Applying Lemma~\ref{lem:rationalrow}\textup{(iii)} to $\Ker\overline f_{\Q}\leq V$ computes a finite rational row description of
$$
\pi_Y^{-1}(\Ker\overline f_{\Q})
=
(\Ker f)_{\Q}.
$$

Finally, let $a\in f(Y)_{\Q}$ be given by its coordinates in the fixed basis of $A_{\Q}$.  Using the matrix of $\overline f_{\Q}$, rational Gaussian elimination computes $q_1,\ldots,q_d\in\Q$ such that
$$
a
=
\sum_{j=1}^d q_j\bigl(f(y_j)\otimes1\bigr).
$$
Set
$
\widetilde a
=
y_1^{q_1}\ldots y_d^{q_d}\in Y_{\Q}.
$
Then
$f_{\Q}(\widetilde a)=a$.  The coordinate vectors of the $y_j$
with respect to $u_{0,1},\ldots,u_{0,h_0}$ are computable, and the
fixed rational power and multiplication polynomials for $N_0$
therefore compute the coordinate vector of $\widetilde a$.
\end{proof}

In the lifting induction the current element will repeatedly be replaced by a conjugate. We therefore need the preceding membership and rational-row data to be transported effectively under conjugation in $H$.

\begin{lemma}\label{lem:transport}
Let $Y$ be a subgroup of $N_0$ for which the following data are given\textnormal:
\begin{enumerate}[label=\textup{(\roman*)}]
\item a membership algorithm for $Y$\textnormal;
\item a finite rational row description of $Y_{\Q}$.
\end{enumerate}
Given a word representing an element $h\in H$, one can compute a membership algorithm for $Y^h$ and a finite rational row description of $(Y^h)_{\Q}$.

\end{lemma}
\begin{proof}
Write the given word representing $h$ as
$w\!=\!h_{i_1}^{\varepsilon_1}\ldots h_{i_t}^{\varepsilon_t}$, where~\hbox{$\varepsilon_k\in\{\pm1\}$.} If $v$ is a word in
$h_1^{\pm1},\ldots,h_m^{\pm1}$ representing an element $n\in H$,
then $wvw^{-1}$ represents $n^{h^{-1}}$. Since
$n\in Y^h$ if and only if $n^{h^{-1}}\in Y$, the given membership
algorithm for $Y$ decides membership in $Y^h$.

Let $r_1,\ldots,r_s$ be the given finite rational row description of $Y_{\Q}$. For every $i=1,\ldots,m$ and $\varepsilon\in\{\pm1\}$, the coordinate vectors of $u_{0,1}^{\,h_i^\varepsilon},\ldots,u_{0,h_0}^{\,h_i^\varepsilon}$ with respect to $u_{0,1},\ldots,u_{0,h_0}$ are part of the fixed data. Hence \eqref{eq:external-aut} computes, in this coordinate system, the automorphism of $(N_0)_{\Q}$ induced by conjugation by $h_i^\varepsilon$. Successively applying these maps to the letters of $w$ computes the coordinate vector of $r_\nu^h$ for every $\nu=1,\ldots,s$. Since conjugation by $h$ extends to an automorphism of $(N_0)_{\Q}$, it commutes with taking the rational closure of subgroups of $N_0$. Then
$$
\Isol_{(N_0)_{\Q}}\bigl(\langle r_1^h,\ldots,r_s^h\rangle\bigr)
=
\left(
\Isol_{(N_0)_{\Q}}\bigl(\langle r_1,\ldots,r_s\rangle\bigr)
\right)^h
=
(Y_{\Q})^h
=
(Y^h)_{\Q}.
$$
Lemma~\ref{lem:rationalrow}\textup{(i)}, applied to the computed
coordinate vectors of $r_1^h,\ldots,r_s^h$, therefore computes a
finite rational row description of $(Y^h)_{\Q}$.

\end{proof}

For each prime $p$ and each $i=1,\ldots,c$, we shall use the fixed additive presentations of $(V_i)_p$ and $(W_i)_p$ supplied by Lemma~\ref{lem:fixed}\textup{(ii)}, together with the corresponding lifting procedures to $(N_i)_p$ and $(J_i)_p$. These procedures compute the lifts in the $\Qp$-extension of the fixed coordinate system of $N_0$, as specified in that lemma.


We shall also need an explicit description of the products $N_iK_0$
after tensor $p$-completion.  Set $A_i=N_iK_0/N_{i+1}K_0$.
The natural map
$$n\in N_i\mapsto nN_{i+1}K_0\in A_i$$ is surjective and has kernel $N_i\cap N_{i+1}K_0=N_{i+1}(N_i\cap K_0)=N_{i+1}J_i$, where $J_i=N_i\cap K_0$ as in \eqref{eq:3.4}.  Hence
$$
A_i\simeq N_i/N_{i+1}J_i\simeq V_i/W_i.
$$
In particular, every $A_i$ is an abelian minimax group.

Tensor exactness, applied to $W_i\trianglelefteq V_i$, gives the canonical identification
$$
(A_i)_p\simeq(V_i/W_i)_p\simeq(V_i)_p/(W_i)_p.
$$
We use this identification to represent elements of $(A_i)_p$ by cosets of $(W_i)_p$ in~$(V_i)_p$.  Lemma~\ref{lem:fixed}\textup{(ii)} supplies finite additive presentations of $(V_i)_p$ and~$(W_i)_p$ in the common ambient space $(V_i)_{\Q}\otimes_{\Q}\Qp$, over the exact coefficient field~$F_{i,p}$ fixed there.  Lemma~\ref{lem:additive}\textup{(iv)} therefore makes equality of represented cosets decidable. In particular, if $b=nN_{i+1}K_0\in A_i$ with $n\in N_i$, then its image in~$(A_i)_p$ is represented by $(nN_{i+1})+(W_i)_p$, and Lemma~\ref{lem:fixed}\textup{(i)} computes the coordinates of $nN_{i+1}$ with respect to $s_{i,1},\ldots,s_{i,r_i}$.

Rationally, flatness of $\Q$ over $\Z$ gives the canonical identification
$$
(A_i)_{\Q}\simeq(V_i/W_i)_{\Q}\simeq
(V_i)_{\Q}/(W_i)_{\Q}.
$$
We use this identification to represent elements of $(A_i)_{\Q}$ in the quotient $(V_i)_{\Q}/(W_i)_{\Q}$.  Lemma~\ref{lem:fixed}\textup{(i)} supplies a basis of this quotient and the matrix of the natural quotient map $(V_i)_{\Q}\to(V_i)_{\Q}/(W_i)_{\Q}$, so the quotient coordinates of represented elements are computable.

\begin{lemma}\label{lem:padicproduct}
For every $i$ and every prime $p$,
\begin{equation}
(N_iK_0)_p=(N_i)_p(K_0)_p
\label{eq:4.21}
\end{equation}
inside $(N_0)_p$, and this subgroup is closed in $(N_0)_p$. Moreover, if $g=nk=n'k'$ with $n,n'\in(N_i)_p$ and $k,k'\in(K_0)_p$, then $n$ and $n'$ have the same image in~$(A_i)_p$, and hence, under the fixed identification
$$
(A_i)_p\simeq(V_i)_p/(W_i)_p,
$$
determine the same represented coset.
\end{lemma}

\begin{proof}
Put $P_i=N_iK_0$. The natural map $K_0\to P_i/N_i$ is surjective, with kernel $J_i=N_i\cap K_0$. Hence $K_0/J_i\simeq P_i/N_i$. By Lemma~\ref{lem:tensor}\textup{(ii)}, the induced map
$$
(K_0)_p\longrightarrow (P_i/N_i)_p\simeq(P_i)_p/(N_i)_p
$$
is surjective. Therefore every coset of $(N_i)_p$ in $(P_i)_p$ has a representative in $(K_0)_p$. Thus, for every $g\in(P_i)_p$, there exists $k\in(K_0)_p$ such that $g(N_i)_p=k(N_i)_p$, whence $gk^{-1}\in(N_i)_p$, and therefore $(P_i)_p\leq(N_i)_p(K_0)_p$. The reverse inclusion follows from Lemma~\ref{lem:tensor}\textup{(i)}, applied to $N_i,K_0\leq P_i$. Hence
$$
(N_iK_0)_p=(N_i)_p(K_0)_p.
$$

The groups $P_i$ and $N_0$ are nilpotent and minimax. The inclusion $P_i\hookrightarrow N_0$ extends to a continuous homomorphism $(P_i)_p\to(N_0)_p$, which is injective by Lemma~\ref{lem:tensor}\textup{(i)}. By \cite[Lemma~5.4(iv)]{KL20} together with \cite[Lemma~4.2(i),(vii)]{KL20}, its image is closed. Thus $(N_iK_0)_p=(N_i)_p(K_0)_p$ is closed in $(N_0)_p$.

Finally, let $q_i:P_i\to A_i=P_i/N_{i+1}K_0$ be the quotient map. Its extension $(q_i)_p:(P_i)_p\to(A_i)_p$ is trivial on $(K_0)_p$, since $q_i$ is trivial on $K_0$. Hence, if $g=nk=n'k'$ with $n,n'\in(N_i)_p$ and $k,k'\in(K_0)_p$, then
$$
(q_i)_p(n)=(q_i)_p(g)=(q_i)_p(n').
$$
Thus $n$ and $n'$ have the same image in $(A_i)_p$, and therefore determine the same represented coset under the fixed identification $(A_i)_p\simeq(V_i)_p/(W_i)_p$.
\end{proof}

We next record the only additional $p$-adic input needed in the lifting argument.  

\begin{lemma}\label{lem:fixedcharts}
Set
$$
\mathcal Y=\{N_0,N_1,\ldots,N_{c+1},K_0\}.
$$
For every $Y\in\mathcal Y$ there is a finite set of primes $\Sigma_Y$ with the following property. Given a prime $p$, one can construct non-negative integers $r=r(Y,p)$ and $s=s(Y,p)$ and elements
$$
r_{1,p},\ldots,r_{r,p},s_{1,p},\ldots,s_{s,p}\in Y_p,
$$
whose coordinate vectors with respect to the fixed $\Qp$-extension of the coordinate system $u_{0,1},\ldots,u_{0,h_0}$ of $N_0$ can be approximated effectively by rational vectors to arbitrary $p$-adic precision, with a computable modulus, such that
\begin{equation}
\Theta_{Y,p}:\Qp^r\times\Zp^s\longrightarrow Y_p,\qquad
(\lambda,z)\longmapsto r_{1,p}^{\lambda_1}\ldots r_{r,p}^{\lambda_r}s_{1,p}^{z_1}\ldots s_{s,p}^{z_s}
\label{eq:fixedchart}
\end{equation}
is surjective. For every $\nu$ and every $\lambda\in\Qp$, the element $r_{\nu,p}^{\lambda}$ belongs to $Y_p$.

If $p\notin\Sigma_Y$, then $r=0$ and the coordinate vectors of the fixed elements $s_{1,p},\ldots,s_{s,p}$ are rational. If $p\in\Sigma_Y$, the finitely many non-rational coordinates of the fixed elements $r_{1,p},\ldots,r_{r,p},s_{1,p},\ldots,s_{s,p}$ are included among the fixed coefficient data of {\rm Convention~\ref{def:fixeddata}}; together with the effective rational approximation procedures above, they therefore belong to a finitely generated exact coefficient field in which equality, the field operations and the $p$-adic valuation are computable. For the finitely many groups $Y\in\mathcal Y$ these fields may, for each fixed $p$, be replaced by one common finitely generated exact field containing them through the effective embeddings fixed in Convention~\ref{def:fixeddata}. The coordinate vector of $\Theta_{Y,p}(\lambda,z)$ is a polynomial function of the variables $\lambda_\nu,z_\mu$, with coefficients in this exact field.
\end{lemma}

\begin{proof}
For $Y=N_i$ and $Y=K_0$ we use the fixed central series
$$
N_i\geq N_{i+1}\geq\ldots\geq N_{c+1}=\{1\},\qquad
K_0=J_1\geq J_2\geq\ldots\geq J_{c+1}=\{1\},
$$
respectively. The successive factors of the first series are $V_j=N_j/N_{j+1}$. For the second series, the natural map $x\in J_j\mapsto xN_{j+1}\in V_j$, has kernel $J_j\cap N_{j+1}=J_{j+1}$ and image $W_j$, and hence induces a canonical isomorphism $J_j/J_{j+1}\simeq W_j$. We use this as an effective identification. Lemma~\ref{lem:fixed}\textup{(ii)} supplies finite $p$-adic additive presentations of $(V_j)_p$ and $(W_j)_p$, together with effective procedures for lifting elements of $(V_j)_p$ to $(N_j)_p$ and elements of $(W_j)_p$ to $(J_j)_p$, in the fixed ambient coordinates of $(N_0)_p$.

We prove simultaneously that, for every $1\leq a\leq b\leq c+1$, each of the fixed quotients $N_a/N_b$ and $J_a/J_b$ has, after tensor $p$-completion, a surjective parameterization of the form \eqref{eq:fixedchart}, with the same conclusions concerning $\Qp$-powers, exact coordinates and coefficient fields. We argue by induction on $b-a$. If $a=b$, the quotient is trivial and we take $r=s=0$.

Suppose that $b=a+1$. The quotient is $V_a$ in the first series and is canonically identified with $W_a$ in the second. Lemma~\ref{lem:fixed}\textup{(ii)} gives its tensor $p$-completion in the form
$$
Q_p=\operatorname{span}_{\Qp}\{d_1,\ldots,d_r\}+\sum_{k=1}^s\Zp\ell_k.
$$
Taking $r_{\nu,p}=d_\nu$ and $s_{k,p}=\ell_k$ gives the required surjective parameterization. Since $Q_p$ is abelian, multiplication by scalars in $\Qp$ and $\Zp$ is precisely the corresponding power operation.

Assume now that $b-a\geq2$ and that the required parameterizations have been constructed for all the fixed quotients with fewer than $b-a$ non-trivial factors. Let $Q$ be either $N_a/N_b$ or $J_a/J_b$, and let $Z$ be the last non-trivial term of its induced central series. Thus $Z=N_{b-1}/N_b$ in the first case and $Z=J_{b-1}/J_b\simeq W_{b-1}$ in the second. In either case $Z$ is central in $Q$. Put $\overline Q=Q/Z$. Then $\overline Q$ is, respectively, $N_a/N_{b-1}$ or $J_a/J_{b-1}$, so it has one fewer non-trivial factor. Tensor exactness gives $Q_p/Z_p\simeq_{\pi_p}(\overline Q)_p$. The induction hypothesis gives a surjective parameterization of $(\overline Q)_p$. Using successively the lifting
procedures of Lemma~\ref{lem:fixed}\textup{(ii)} along the fixed central series, we compute in $Q_p$ a lift, with exact ambient coordinates, of each fixed element occurring in this parameterization.

Let $\overline s\in(\overline Q)_p$ be a fixed element which occurs in the parameterization supplied by the induction hypothesis in the form $\overline s^{\,z}$, with $z\in\Zp$, and choose a computed lift $s\in Q_p$. For every $z\in\Zp$, functoriality of tensor $p$-completion gives $\pi_p(s^z)=\overline s^{\,z}$. Thus, once a lift $s$ of $\overline s$ has been chosen, the elements $s^z$, $z\in\Zp$, provide lifts of all the elements $\overline s^{\,z}$ occurring in the parameterization.

Now let $\overline r\in(\overline Q)_p$ be one of the fixed elements occurring in the parameterization supplied by the induction hypothesis with an exponent ranging over $\Qp$. Thus $\overline r^{\,\lambda}\in(\overline Q)_p$ for every $\lambda\in\Qp$. We construct a lift $r\in Q_p$ of $\overline r$ such that $r^\lambda\in Q_p$ and $\pi_p(r^\lambda)=\overline r^{\,\lambda}$ for every $\lambda\in\Qp$.

Write the fixed additive presentation of the central factor as
$$
Z_p=D+L,\qquad
D=\operatorname{span}_{\Qp}\{d_1,\ldots,d_u\},\qquad
L=\sum_{k=1}^v\Zp\ell_k.
$$
Choose a lift $\widetilde y_0\in Q_p$ of $\overline r$, together with its exact coordinate vector. Suppose that, for some $n\geq0$, we have constructed a lift $\widetilde y_n\in Q_p$ of $\overline r^{\,p^{-n}}$, together with its exact coordinate vector. By the induction hypothesis, the exact coordinate vector of $\overline r^{\,p^{-(n+1)}}$ is computable by evaluating the parameterization of $(\overline Q)_p$ at $p^{-(n+1)}$ in the parameter corresponding to $\overline r$ and at zero in all the other parameters. Applying successively the lifting procedures of Lemma~\ref{lem:fixed}\textup{(ii)} to the factors between $Q$ and $Z$, we compute a lift $y_{n+1}\in Q_p$ of $\overline r^{\,p^{-(n+1)}}$, together with its exact coordinate vector.

The elements $y_{n+1}^{p}$ and $\widetilde y_n$ have the same image under $\pi_p$, since
$$
\pi_p(y_{n+1}^{p})=(\overline r^{\,p^{-(n+1)}})^p=\overline r^{\,p^{-n}}=\pi_p(\widetilde y_n).
$$
Hence $c_n:=y_{n+1}^{p}\widetilde y_n^{-1}$ belongs to $\operatorname{Ker}(\pi_p)=Z_p$. Since the exact coordinate vectors of $y_{n+1}$ and $\widetilde y_n$ are known, the fixed power, inversion and multiplication polynomials compute the exact coordinate vector of $c_n$. We now regard the abelian group $Z_p$ additively. Applying Lemma~\ref{lem:additive}\textup{(i)} to the exact coordinate vector of $c_n$, we compute coefficients $a_{n,1},\ldots,a_{n,u}$ in the fixed exact coefficient field and $z_{n,1},\ldots,z_{n,v}$ in its intersection with $\Zp$ such that
$$
c_n=\sum_{j=1}^u a_{n,j}d_j+\sum_{k=1}^v z_{n,k}\ell_k.
$$
Put $\delta_n:=\sum_{j=1}^u a_{n,j}d_j\in D$ and $\eta_n:=\sum_{k=1}^v z_{n,k}\ell_k\in L$. Thus $c_n=\delta_n+\eta_n$ in the additive notation of $Z_p$.

Since $D$ is a $\Qp$-vector space, we have $-p^{-1}\delta_n\in D$. The coefficients $a_{n,j}$ belong to the fixed exact coefficient field, which contains $\Q$. Hence the coefficients $-p^{-1}a_{n,j}$ are computable exactly, and therefore so is the element
$$
e_n:=-p^{-1}\delta_n\in D.
$$
We now set $\widetilde y_{n+1}=y_{n+1}e_n$, viewing $e_n$ as an element of $Z_p\leq Q_p$. Since $Z_p=\operatorname{Ker}(\pi_p)$, we have $\pi_p(e_n)=1$, and hence $\pi_p(\widetilde y_{n+1})=\pi_p(y_{n+1})=\overline r^{\,p^{-(n+1)}}$. Thus $\widetilde y_{n+1}$ is again a lift of $\overline r^{\,p^{-(n+1)}}$.

Since $e_n\in Z_p$ and $Z_p$ is central in $Q_p$, we have
$$
\widetilde y_{n+1}^{p}\widetilde y_n^{-1}=y_{n+1}^{p}e_n^{p}\widetilde y_n^{-1}=c_ne_n^{p}.
$$
In the additive notation of $Z_p$, the element $e_n^p$ corresponds to $pe_n=-\delta_n$. Since $c_n=\delta_n+\eta_n$, it follows that $c_ne_n^{p}=\eta_n\in L$. Therefore
$$
\widetilde y_{n+1}^{p}\widetilde y_n^{-1}=\eta_n\in L.
$$

For $n\geq0$, set $r_n=\widetilde y_n^{p^n}$. The preceding equality and the centrality of $\eta_n$ give $r_{n+1}r_n^{-1}=\eta_n^{p^n}$. Write the elements of $Z_p$ in the fixed coordinates of its ambient finite-dimensional $\Qp$-space. Since $L=\sum_{k=1}^v\Zp\ell_k$ and the coordinate vectors of the fixed generators $\ell_1,\ldots,\ell_v$ are known exactly, one can compute an integer $C$ such that every coordinate of every element of $L$ has $p$-adic valuation at least $C$. Hence every coordinate of $\eta_n^{p^n}$ has valuation at least $C+n$, and so for every integer $M$ one can compute an integer $N(M)$ such that, whenever $m>n\geq N(M)$, every coordinate of $r_mr_n^{-1}$ has $p$-adic valuation at least $M$.

Since $Q_p$ is closed in the corresponding rational $\Qp$-completion, $(r_n)$ converges to an element $r\in Q_p$. Moreover, the effective Cauchy estimate above computes the coordinate vector of $r$ to arbitrary prescribed $p$-adic precision. For every $n$, we have $\pi_p(\widetilde y_n)=\overline r^{\,p^{-n}}$. Hence, since $r_n=\widetilde y_n^{p^n}$,
$$
\pi_p(r_n)=\pi_p(\widetilde y_n)^{p^n}=\left(\overline r^{\,p^{-n}}\right)^{p^n}=\overline r.
$$
By continuity of $\pi_p$, $\pi_p(r)=\overline r$.

We next prove that $r^\lambda\in Q_p$ for every $\lambda\in\Qp$. By \cite[Lemma~5.4(iv) and Lemma~4.6]{KL20}, the finite residual $R(Q_p)$ is closed and radicable, and every finite-index subgroup of $Q_p$ is open. We first show that $r\in R(Q_p)$. Let $U\leq Q_p$ have finite index. Replacing $U$ by its normal core, we may assume that $U\trianglelefteq Q_p$. Since $U$ is open, the quotient map $Q_p\to Q_p/U$ is continuous. The finite group $Q_p/U$ is a $p$-group. Indeed, choose finitely many elements of $Q_p$ whose images generate $Q_p/U$. By the direct-limit definition of tensor $p$-completion, these elements all belong to the pro-$p$ completion of one finitely generated subgroup of $Q$. Their images in $Q_p/U$ therefore generate a finite quotient of a pro-$p$ group, and hence a finite $p$-group. If $p^e$ is the exponent of $Q_p/U$, then $r_n=\widetilde y_n^{p^n}\in U$ for every $n\geq e$. Since $U$ is closed and $r_n\to r$, we obtain $r\in U$. As this holds for every finite-index subgroup $U$ of $Q_p$, we conclude that $r\in R(Q_p)$.

Recall that $Q_p$ is identified with a subgroup of the rational $\Qp$-completion of $Q$. In this ambient group all rational powers are uniquely defined. Since $R(Q_p)$ is radicable, for every $q\in\Q$ the rational power $r^q$ belongs to $R(Q_p)$. Now let $\lambda\in\Qp$ and choose a sequence $(q_m)$ in $\Q$ with $q_m\to\lambda$. By the fixed power polynomials in the rational $\Qp$-completion, $r^{q_m}\to r^\lambda$. Since $R(Q_p)$ is closed and $r^{q_m}\in R(Q_p)$ for every $m$, it follows that $r^\lambda\in R(Q_p)\leq Q_p$. Moreover, for every $m$, $\pi_p(r^{q_m})=\overline r^{\,q_m}$, and continuity of $\pi_p$ and of the power operation gives
$$
\pi_p(r^\lambda)=\lim_{m\to\infty}\pi_p(r^{q_m})=\lim_{m\to\infty}\overline r^{\,q_m}=\overline r^{\,\lambda}.
$$
Thus $r^\lambda\in Q_p$ and $\pi_p(r^\lambda)=\overline r^{\,\lambda}$ for every $\lambda\in\Qp$.

To prove surjectivity, let $y\in Q_p$. Represent $\pi_p(y)$ by the parameterization of $(\overline Q)_p$, and replace each fixed element in that representation by the corresponding chosen lift in $Q_p$. Let $y_0\in Q_p$ be the resulting product. By construction, $\pi_p(y_0)=\pi_p(y)$, so $y_0^{-1}y\in\ker(\pi_p)=Z_p$. Represent $y_0^{-1}y$ by the fixed additive parameterization of $Z_p$. Combining the two sets of parameters gives a representation of $y$. Hence the resulting map $\Qp^r\times\Zp^s\to Q_p$ is surjective. This completes the simultaneous induction.

We now specify the finite sets $\Sigma_Y$. For each $Y\in\mathcal Y$, let $\Sigma_Y$ be the union of the spectra of the successive abelian factors in the fixed central series used above for $Y$. Explicitly,
$$
\Sigma_{N_0}=\Sigma_{N_1}=\bigcup_{j=1}^{c}\spec(V_j),\qquad
\Sigma_{N_i}=\bigcup_{j=i}^{c}\spec(V_j)\quad(2\leq i\leq c),
$$
$$
\Sigma_{N_{c+1}}=\varnothing,\qquad
\Sigma_{K_0}=\bigcup_{j=1}^{c}\spec(W_j).
$$
These sets are finite and are determined by the fixed data. Suppose first that $p\notin\Sigma_Y$. Then, for every factor in the fixed central series of $Y$, Lemma~\ref{lem:fixed}\textup{(ii)} gives its $p$-completion using only the fixed $\Zp$-generators, whose coordinate vectors are rational. Consequently no $\Qp$-parameter is introduced at any stage of the induction, so $r=0$. Moreover, all fixed elements occurring in the resulting parameterization may be chosen with rational coordinate vectors with respect to $u_{0,1},\ldots,u_{0,h_0}$.

Suppose now that $p\in\Sigma_Y$. The fixed elements in the additive presentations of the successive factors have coordinates in the exact coefficient fields supplied by Lemma~\ref{lem:fixed}\textup{(ii)}, and the remaining fixed lifts obtained from them by the lifting procedures have exact coordinates in the corresponding coefficient fields. For every additional element $r$ constructed above, the effective Cauchy modulus for $(r_n)$ gives, coordinate by coordinate, rational approximations to the coordinates of $r$ to arbitrary prescribed $p$-adic precision. Since $Y$ and $p$ are fixed and the central series has only finitely many factors, only finitely many such additional coordinates occur. We include these coordinates among the fixed coefficient data prescribed by Convention~\ref{def:fixeddata}; the construction of $(r_n)$ supplies the required effective rational approximation procedures. Lemma~\ref{lem:field} then supplies exact equality, field operations and $p$-adic valuation in the resulting finitely generated exact coefficient field. For each fixed $p$, the finitely many coefficient fields arising from the groups $Y\in\mathcal Y$ may be embedded effectively into one common finitely generated exact field, as required by Convention~\ref{def:fixeddata}.

Finally, all fixed elements in the resulting parameterization have exact coordinate vectors in the fixed coefficient field. Since the fixed multiplication and power polynomials have rational coefficients, the coordinate vector of
$$
r_{1,p}^{\lambda_1}\ldots r_{r,p}^{\lambda_r}s_{1,p}^{z_1}\ldots s_{s,p}^{z_s}
$$
is a polynomial function of $\lambda_1,\ldots,\lambda_r,z_1,\ldots,z_s$, with coefficients in that exact field.
\end{proof}

The parameterizations of Lemma~\ref{lem:fixedcharts} will reduce the remaining $p$-adic membership conditions to the existence of solutions of polynomial systems. We record the uniform decision procedure for these systems.

\begin{lemma}\label{lem:padicdecision}
Let $\Sigma$ be a fixed finite set of primes. For each $p\in\Sigma$, fix a finitely generated exact coefficient field $E_p\leq\Qp$ of the type specified in Convention~\ref{def:fixeddata}, and put $E_p=\Q$ for $p\notin\Sigma$.

There is a single terminating procedure, depending only on these fixed data, which, on input a prime $p$, polynomials
$$
f_1,\ldots,f_m\in E_p[X_1,\ldots,X_n]
$$
given by their coefficients in the fixed exact representation of $E_p$, and a subset $I\subseteq\{1,\ldots,n\}$, decides whether there exists $(x_1,\ldots,x_n)\in\Qp^n$ such that
$$
f_1(x_1,\ldots,x_n)=\ldots=f_m(x_1,\ldots,x_n)=0
$$
and $x_i\in\Zp$ for every $i\in I$.
\end{lemma}
\begin{proof}
Suppose first that $p\notin\Sigma$. Then $E_p=\Q$, and hence all coefficients of $f_1,\ldots,f_m$ are rational. By \cite[Proposition~3.5]{Belair88}, B\'elair proves that the theory $T'$ introduced there has the same models as $\operatorname{Th}\{\Qp:p\text{ prime}\}$, and \cite[Theorem~3.8]{Belair88} proves quantifier elimination for $T'$. The discussion following that theorem, \cite[pp.~15--16]{Belair88}, shows that this elimination can be carried out effectively. Moreover, the quantifier-free conditions produced by the procedure can be evaluated effectively from the input prime $p$ and the rational coefficients. Hence there is a single terminating procedure, with $p$ as part of the input, for deciding whether
$$
\exists x_1,\ldots,x_n\in\Qp\;
\left(
\bigwedge_{j=1}^{m} f_j(x_1,\ldots,x_n)=0
\ \wedge\
\bigwedge_{i\in I} v_p(x_i)\geq0
\right)
$$
holds.

Suppose now that $p\in\Sigma$. Since $\Sigma$ is fixed and finite, we may treat each such prime separately. We use Cohen's effective elimination procedure for a fixed $p$-adic field. The basic effective input is supplied by \cite[Theorem~$A_n$, pp.~140--144]{Coh69}, which constructs effectively, from the coefficients of a polynomial, the finite data required in the elimination process. Then \cite[Theorem~3.1]{Coh69} eliminates one quantified variable. Iterating this procedure, Cohen concludes explicitly on \cite[p.~145]{Coh69} that all bound variables can be eliminated and that this yields a decision procedure.

It remains to verify that this procedure can be carried out with coefficients given in the fixed exact representation of $E_p$. By Convention~\ref{def:fixeddata} and Lemma~\ref{lem:field}, equality and the field operations in $E_p$ are computable, as is $v_p(a)$ for every $a\in E_p$, and every element of $E_p$ can be approximated effectively by rationals to arbitrary prescribed $p$-adic precision. In particular, if $a\in E_p$ is integral and $k\geq1$, we can compute $q\in\Q$ such that $v_p(a-q)\geq k$, and hence compute the residue class of $a$ modulo $p^k$. For arbitrary $0\neq a\in E_p$, we first compute $v_p(a)$ and reduce to the integral element $p^{-v_p(a)}a$. Thus all coefficient computations required in Cohen's procedure are effective in the representation fixed for $E_p$.

Since $p$ is fixed, the finite residue rings occurring in the procedure are effective finite rings. Hence Cohen's decision procedure can be carried out with coefficients in $E_p$.

Finally, when $p\notin\Sigma$ we use the procedure obtained in the first part of the proof, while for each of the finitely many primes $p\in\Sigma$ we use the corresponding fixed Cohen procedure. This gives a single terminating algorithm, with $p$ as part of the input, deciding whether the required tuple exists.
\end{proof}

\section{Conjugacy lifting}\label{sec:5}

We now turn the calculations of Sections~\ref{sec:3} and \ref{sec:4} into an algorithm for conjugacy in the quotient $E=H/K_0$.  The subgroup $M=N_0/K_0$ is nilpotent and normal in $E$, and $E/M$ is finitely generated and virtually abelian.  We first decide conjugacy in this quotient, where the problem reduces to finitely many integral linear equations.

The isolated lower central series $N_0=N_1\geq\ldots\geq N_{c+1}=1$ induces the normal series 
$$
M=M_1\geq M_2\geq\ldots\geq M_{c+1}=\{1\},
\qquad
M_i=N_iK_0/K_0,
$$
in $E$. Starting from a conjugacy modulo $M=M_1$, we lift it successively through this series.  More precisely, let $x,y\in E$ be the two elements whose conjugacy is to be decided.  Suppose that, at stage $i$, an element $c\in E$ has been found such that $x^cM_i=yM_i$.  The next step is to determine whether this conjugacy modulo $M_i$ can be lifted modulo $M_{i+1}$.  Equivalently, we ask whether there exists $t\in E$ such that $x^{ct}M_{i+1}=yM_{i+1}$.  Such an element $t$ necessarily centralizes $x^cM_i$ in $E/M_i$.  Replacing $x$ by $x^c$, the lifting problem is therefore reduced to finding an element $t$ in the inverse image of $C_{E/M_i}(xM_i)$ such that $x^tM_{i+1}=yM_{i+1}$.  This condition is expressed as a derivation equation in the abelian factor $M_i/M_{i+1}$.

The first subsection gives the effective procedure for this derivation problem, and the second applies it successively to the series above, thereby deciding conjugacy in $E=H/K_0$.  In Section~\ref{sec:6}, the resulting algorithm is transferred from $E$ to $G$ through the finite kernel $K/K_0$.

\subsection{Equations in virtually abelian groups}

First, we isolate the only external matrix algorithm we need.

\begin{lemma}\label{lem:matrixorbit}
Let $T_1,\ldots,T_r\in\operatorname{GL}_d(\Q)$ commute and let $v,w\in\Q^d$. There is an algorithm which decides whether there exist integers
$n_1,\ldots,n_r$ such that
\begin{equation}
 T_1^{n_1}\ldots T_r^{n_r}v=w.
 \label{eq:5.1}
\end{equation}
If such integers exist, the algorithm returns a vector $\mathbf n_0\in\Z^r$ and a basis of a subgroup $L\leq\Z^r$ such that the set of all solutions $\mathbf n=(n_1,\ldots,n_r)\in\Z^r$ is precisely
\begin{equation}
 \mathbf n_0+L.
 \label{eq:5.2}
\end{equation}
\end{lemma}

\begin{proof}
If $r=0$, equation \eqref{eq:5.1} is simply $v=w$; its solution set is either empty or the singleton subgroup $\Z^0$.  Assume $r\geq1$.  If $v=0$, then every product in \eqref{eq:5.1} sends $v$ to $0$.  Hence there is no solution when $w\neq0$, while for $w=0$ the complete solution set is $0+\Z^r$.  We may therefore assume $v\neq0$.

Construct the smallest $T_i^{\pm1}$-invariant rational subspace $C$ containing $v$: start with $C_0=\Q v$ and repeatedly replace $C_k$ by
$$
 C_{k+1}=C_k+\sum_{i=1}^r(T_iC_k+T_i^{-1}C_k).
$$
The dimension is at most $d$, so the process stabilizes after at most $d-1$ strict increases.  If $w\notin C$, equation \eqref{eq:5.1} has no solution.

Assume $w\in C$.  Choose a basis
$$
 v_1=v,
 \qquad
 v_j=R_jv\quad(2\leq j\leq e),
$$
where each $R_j$ is a product of powers of the commuting $T_i$.  If $T=T_1^{n_1}\ldots T_r^{n_r}$ satisfies $Tv=w$, then
$$
 Tv_j=TR_jv=R_jTv=R_jw.
$$
Thus the restriction $T|_C$ is uniquely determined by the condition $Tv=w$.  Rational row reduction decides whether the assignments
$$
 v_j\longmapsto R_jw
$$
define an invertible linear map $B\in\operatorname{GL}(C)$ commuting with every $T_i|_C$.  If not, \eqref{eq:5.1} has no solution.  If they do, then we proceed. Now, \eqref{eq:5.1} is equivalent to
\begin{equation}
 (T_1|_C)^{n_1}\ldots(T_r|_C)^{n_r}=B.
 \label{eq:5.7}
\end{equation}
Finally, Babai--Beals--Cai--Ivanyos--Luks \cite[Theorem~1.2 and Section~6.4]{BBCIL96}, applied over $\Q$, decide whether \eqref{eq:5.7} has an integral solution and, in the positive case, compute one such solution $\mathbf n_0\in\Z^r$ together with a basis of the following subgroup of $\Z^r$:
$$
L=\left\{\mathbf m\in\Z^r:
(T_1|_C)^{m_1}\ldots(T_r|_C)^{m_r}=I_C\right\}.
$$
The complete set of solutions of \eqref{eq:5.7} is then precisely $\mathbf n_0+L$, which proves the lemma.
\end{proof}

Let a group $P$ act on an abelian group $A$ on the right, and let
$
 \delta:P\longrightarrow A
$
be a derivation, so  
\begin{equation}
 \delta(uv)=\delta(u)^v+\delta(v).
 \label{eq:5.8}
\end{equation}
Define the associated affine action by  
\begin{equation}
 a*u=a^u+\delta(u).
 \label{eq:5.9}
\end{equation}
Equation \eqref{eq:5.8} gives
\begin{equation}
 (a*u)*v=a*(uv).
 \label{eq:5.10}
\end{equation}

We first solve the resulting affine orbit and stabilizer problems when the acting group is free abelian.

\begin{lemma}\label{lem:affineabelian}
Let $B=\langle z_1,\ldots,z_\ell\rangle$ be a finitely generated free abelian group with specified basis. Let $A$ be an abelian minimax group equipped with an effective representation in which addition, inverse and equality are computable. Suppose that $B$ acts on $A$ on the right and that $\delta:B\to A$ is a derivation for this action. Assume that:
\begin{enumerate}[label=\textup{(\roman*)}]
\item for every $i=1,\ldots,\ell$, the actions of $z_i$ and $z_i^{-1}$ on $A$, and the elements $\delta(z_i)$ and $\delta(z_i^{-1})$, are computable;
\item a rational basis of $A_{\Q}$ is fixed, and, with respect to this basis, the action of every $z_i$ on $A_{\Q}$ and the vector $\delta(z_i)\otimes1$ are given by rational matrices and rational coordinate vectors.
\end{enumerate}
Then there are terminating algorithms with the following properties:
\begin{enumerate}[label=\textup{(\alph*)}]
\item given $s,a\in A$, together with the coordinate vectors of $s\otimes1$ and $a\otimes1$ in the fixed rational basis of $A_{\Q}$, one decides whether there exists $z\in B$ such that $s*z=a$ and, in the positive case, returns such a $z$ as a word in $z_1^{\pm1},\ldots,z_\ell^{\pm1}$;
\item given $s\in A$, together with the coordinate vector of $s\otimes1$ in the fixed rational basis of $A_{\Q}$, one computes, as words in $z_1^{\pm1},\ldots,z_\ell^{\pm1}$, a finite generating set for $\operatorname{Stab}_B(s)=\{z\in B:s*z=s\}$.
\end{enumerate}
\end{lemma}

\begin{proof}
Put $h=\dim_{\Q}A_{\Q}$. With respect to the fixed rational basis of $A_{\Q}$, let $M_i\in\operatorname{GL}_h(\Q)$ be the matrix of the action of $z_i$ and let $d_i\in\Q^h$ be the coordinate vector of $\delta(z_i)\otimes1$. The affine transformation $v\mapsto M_iv+d_i$ is represented on $A_{\Q}\oplus\Q$ by
\begin{equation}
 T_i=
 \begin{pmatrix}
 M_i&d_i\\
 0&1
 \end{pmatrix}
 \in\operatorname{GL}_{h+1}(\Q).
 \label{eq:5.16}
\end{equation}
The matrices $T_i$ commute, because the $z_i$ commute and the maps $v\mapsto v*z$ form a right action of $B$.

For $\mathbf n=(n_1,\ldots,n_\ell)\in\Z^\ell$, put $z(\mathbf n)=z_1^{n_1}\ldots z_\ell^{n_\ell}$. We first prove \textup{(a)}. The equation to be decided is
\begin{equation}
 s*z(\mathbf n)=a
 \qquad\text{for some }\mathbf n\in\Z^\ell.
 \label{eq:5.15}
\end{equation}
Applying Lemma~\ref{lem:matrixorbit} to the commuting matrices $T_1,\ldots,T_\ell$ and the augmented coordinate vectors $(\bar s,1)$ and $(\bar a,1)$ decides whether there exists $\mathbf n\in\Z^\ell$ such that
\[
 \bar s*z(\mathbf n)=\bar a,
\]
where $\bar x=x\otimes1$. If there is no such exponent vector, then \eqref{eq:5.15} has no solution. Otherwise the lemma computes $\mathbf n_0\in\Z^\ell$ and a basis $\boldsymbol\ell_1,\ldots,\boldsymbol\ell_t$ of a subgroup $L\leq\Z^\ell$ such that the complete set of exponent vectors solving the equation in $A_{\Q}$ is $\mathbf n_0+L$.

Put $b=s*z(\mathbf n_0)$. Then $\bar b=\bar a$, and hence
\begin{equation}
 b-a\in\tor A.
 \label{eq:5.18}
\end{equation}
If $\boldsymbol\ell\in L$, then both $\mathbf n_0$ and $\mathbf n_0+\boldsymbol\ell$ solve the equation in $A_{\Q}$. Since $B$ is abelian,
\[
 z(\mathbf n_0+\boldsymbol\ell)=z(\mathbf n_0)z(\boldsymbol\ell),
\]
and therefore
\[
 \bar a*z(\boldsymbol\ell)
 =
 (\bar s*z(\mathbf n_0))*z(\boldsymbol\ell)
 =
 \bar s*z(\mathbf n_0+\boldsymbol\ell)
 =
 \bar a.
\]
Consequently, for $1\leq k\leq t$,
\begin{equation}
 a*z(\pm\boldsymbol\ell_k)-a\in\tor A.
 \label{eq:5.19}
\end{equation}

The torsion elements in \eqref{eq:5.18}--\eqref{eq:5.19} are explicitly computable. Since addition and equality in $A$ are computable, their orders are found by testing successive positive multiples until zero is reached. Let $m$ be a common multiple of these finitely many orders. If $g=z(\pm\boldsymbol\ell_k)$ and $\xi\in A[m]$, then
\[
 (a+\xi)*g-a=(a*g-a)+\xi^g.
\]
The first term belongs to $A[m]$ by the choice of $m$, while the second does because the action of $g$ is an automorphism of $A$ and therefore preserves $A[m]$. Thus every $z(\pm\boldsymbol\ell_k)$ preserves the coset $a+A[m]$. Since $A$ is minimax, $A[m]$ is finite, and \eqref{eq:5.18} gives $b\in a+A[m]$.

Starting with $\mathcal O_0=\{b\}$, define
\[
 \mathcal O_{n+1}
 =
 \mathcal O_n\cup
 \{u*z(\pm\boldsymbol\ell_k):u\in\mathcal O_n,\ 1\leq k\leq t\}.
\]
Each $\mathcal O_n$ is effectively computable, and all these sets lie in the finite set $a+A[m]$, so the process terminates. Its stable value is exactly
\begin{equation}
 \{b*z(\boldsymbol\ell):\boldsymbol\ell\in L\}.
 \label{eq:5.20}
\end{equation}
Whenever a new point is reached, retain a word in the elements $z(\pm\boldsymbol\ell_k)$ which produces it from $b$. We can therefore decide whether $a$ belongs to this orbit. If it does, the recorded word gives some $\boldsymbol\ell\in L$ such that $b*z(\boldsymbol\ell)=a$, and hence
\[
 s*\bigl(z(\mathbf n_0)z(\boldsymbol\ell)\bigr)=a.
\]
Thus $z(\mathbf n_0)z(\boldsymbol\ell)$ is the required element of $B$. This proves \textup{(a)}.

We now prove \textup{(b)}. Apply Lemma~\ref{lem:matrixorbit} to the same commuting matrices $T_1,\ldots,T_\ell$, now with both start and target equal to $(\bar s,1)$. Since $\mathbf0$ is a solution, the complete set of exponent solutions is a subgroup $L\leq\Z^\ell$, and the lemma computes a basis $\boldsymbol\ell_1,\ldots,\boldsymbol\ell_t$ of $L$. Every element of $\operatorname{Stab}_B(s)$ belongs to $z(L)$, so $\operatorname{Stab}_B(s)$ is exactly the stabilizer of $s$ inside $z(L)$.

For each basis element, $s*z(\pm\boldsymbol\ell_k)-s$ is torsion. Compute the orders of these finitely many torsion elements and choose a common multiple $m$, taking $m=1$ if $t=0$. If $g=z(\pm\boldsymbol\ell_k)$ and $\xi\in A[m]$, then
\[
 (s+\xi)*g-s=(s*g-s)+\xi^g\in A[m].
\]
Since $A$ is minimax, $A[m]$ is finite. Starting from $\mathcal O_0=\{s\}$, close successively under the transformations $z(\pm\boldsymbol\ell_k)$. The process therefore terminates at the exact orbit $s*z(L)$. Whenever a point $u$ is first reached, retain a word $w_u$ in the alphabet $z(\pm\boldsymbol\ell_1),\ldots,z(\pm\boldsymbol\ell_t)$ such that $s*w_u=u$, taking $w_s=1$.

For every $u$ in this finite orbit and every $\eta\in\{z(\pm\boldsymbol\ell_1),\ldots,z(\pm\boldsymbol\ell_t)\}$ such that $u*\eta=v$, define
\begin{equation}
 g_{u,\eta}=w_u\eta w_v^{-1}.
 \label{eq:5.23}
\end{equation}
Then $s*g_{u,\eta}=s$, so every $g_{u,\eta}$ belongs to $\operatorname{Stab}_B(s)$. Conversely, let an element of $\operatorname{Stab}_B(s)$ be represented by a word $\eta_1\ldots\eta_q$ in the generators $z(\pm\boldsymbol\ell_k)$, and put $u_k=s*(\eta_1\ldots\eta_k)$, with $u_0=s$. Since the element fixes $s$, we have $u_q=s$. Inserting the terms $w_{u_k}^{-1}w_{u_k}$ between successive letters gives
\[
 \eta_1\ldots\eta_q
 =
 (w_{u_0}\eta_1w_{u_1}^{-1})
 (w_{u_1}\eta_2w_{u_2}^{-1})
 \ldots
 (w_{u_{q-1}}\eta_qw_{u_q}^{-1}),
\]
and every factor on the right is one of the elements in \eqref{eq:5.23}. Hence these finitely many elements generate $\operatorname{Stab}_B(s)$, proving \textup{(b)}.
\end{proof}

We now pass from the free abelian case to finitely generated subgroups of a fixed virtually abelian group. The finite quotient will reduce both the derivation equation and the stabilizer computation to finitely many applications of Lemma~\ref{lem:affineabelian}.

\begin{theorem}\label{thm:derivation}
Let $V$ be a fixed finitely generated virtually abelian group. Fix also a finite generating set $X=\{x_1,\ldots,x_m\}$ of $V$. Fix the following finite effective data for $V$\textnormal:
\begin{enumerate}[label=\textup{(\roman*)}]
\item a free abelian normal subgroup $Q_0$ of $V$ of finite index, together with a basis $q_1,\ldots,q_d$ of $Q_0$\textnormal;

\item the finite quotient $F=V/Q_0=\{f_1,\ldots,f_s\}$, with $f_1=1$, together with its multiplication table, a transversal $\tau:F\to V$ with $\tau(1)=1$, and, for every $f,g\in F$, the coordinate vector with respect to $q_1,\ldots,q_d$ of the unique element $\kappa(f,g)\in Q_0$ such that
\begin{equation}
 \tau(f)\tau(g)=\kappa(f,g)\tau(fg)\textnormal;
 \label{eq:5.11}
\end{equation}

\item for every $f\in F$, the integral matrix, with respect to $q_1,\ldots,q_d$, of the automorphism of $Q_0$ induced by conjugation by $\tau(f)$\textnormal;

\item for every $1\leq\nu\leq m$ and $\varepsilon=\pm1$, the unique pair
\[
 \bigl(\mathbf a_{\nu,\varepsilon},f_{\nu,\varepsilon}\bigr)\in\Z^d\times F,
 \qquad
 \mathbf a_{\nu,\varepsilon}=(a_{\nu,\varepsilon,1},\ldots,a_{\nu,\varepsilon,d}),
\]
such that
\[
 x_\nu^\varepsilon=q_1^{a_{\nu,\varepsilon,1}}\ldots q_d^{a_{\nu,\varepsilon,d}}\tau(f_{\nu,\varepsilon}).
\]
\end{enumerate}

\noindent Let $P\leq V$ be the subgroup generated by elements $p_1,\ldots,p_r$, each given as a word in $X^{\pm1}$. Let $A$ be an abelian minimax group equipped with a fixed effective representation in which addition, inverse and equality are computable. Suppose that $P$ acts on~$A$ and that $\delta:P\to A$ is a derivation for this action. Assume moreover that\textnormal:
\begin{enumerate}[label=\textup{(\roman*)},resume]
\item for every $j=1,\ldots,r$, the actions of $p_j$ and $p_j^{-1}$ on $A$, and the elements $\delta(p_j)$ and $\delta(p_j^{-1})$, are computable; hence the action and the value of $\delta$ are computable on every word in $p_1^{\pm1},\ldots,p_r^{\pm1}$ by \eqref{eq:5.8}\textnormal;

\item a rational basis of $A_{\Q}$ is fixed, and, with respect to this basis, the action of every $p_j$ on $A_{\Q}$ and the vector $\delta(p_j)\otimes1$ are given by rational matrices and rational coordinate vectors.
\end{enumerate}

\noindent Then there are terminating algorithms with the following properties\textnormal:
\begin{enumerate}[label=\textup{(\alph*)}]
\item given $a\in A$, together with the coordinate vector of $a\otimes1$ with respect to the rational basis of $A_{\Q}$ fixed in \textup{(vi)}, one decides whether there exists $u\in P$ such that $\delta(u)=a$; if such an element exists, the algorithm returns a word in $p_1^{\pm1},\ldots,p_r^{\pm1}$ representing one such $u$\textnormal;

\item given $a\in A$, together with the coordinate vector of $a\otimes1$ with respect to the rational basis of $A_{\Q}$ fixed in \textup{(vi)}, one computes, as words in $p_1^{\pm1},\ldots,p_r^{\pm1}$, a finite generating set for $\operatorname{Stab}_P(a)=\{u\in P:a*u=a\}$.
\end{enumerate}
\end{theorem}

\begin{proof}
For $\mathbf a=(a_1,\ldots,a_d)\in\Z^d$, write $q(\mathbf a)=q_1^{a_1}\ldots q_d^{a_d}$. Since $Q_0\trianglelefteq V$ and $\tau:F\to V$ is a transversal for $F=V/Q_0$, every element $v\in V$ has a unique expression $v=q(\mathbf a)\tau(f)$ with $\mathbf a\in\Z^d$ and $f\in F$. We now explain how this normal form is effectively computed from the fixed data. By \textup{(iv)}, for every letter $x_\nu^\varepsilon$ in $X^{\pm1}$ we know the unique pair $(\mathbf a_{\nu,\varepsilon},f_{\nu,\varepsilon})$ satisfying $x_\nu^\varepsilon=q(\mathbf a_{\nu,\varepsilon})\tau(f_{\nu,\varepsilon})$. For $f\in F$, let $\rho_f\in\operatorname{GL}_d(\Z)$ be the matrix, with respect to $q_1,\ldots,q_d$, of the automorphism $u\mapsto u^{\tau(f)}$ supplied by \textup{(iii)}, and let $\mathbf k(f,g)\in\Z^d$ be the coordinate vector of $\kappa(f,g)$. Hence $\rho_f^{-1}$ is the matrix of $u\mapsto\tau(f)u\tau(f)^{-1}$. If $q(\mathbf a)\tau(f)$ and $q(\mathbf b)\tau(g)$ are in normal form, then
\[
\begin{aligned}
 q(\mathbf a)\tau(f)q(\mathbf b)\tau(g)
 &=
 q(\mathbf a)\bigl(\tau(f)q(\mathbf b)\tau(f)^{-1}\bigr)\tau(f)\tau(g)\\
 &=
 q(\mathbf a)q(\rho_f^{-1}\mathbf b)q(\mathbf k(f,g))\tau(fg)\\
 &=
 q\bigl(\mathbf a+\rho_f^{-1}\mathbf b+\mathbf k(f,g)\bigr)\tau(fg).
\end{aligned}
\]
Thus, from the two pairs $(\mathbf a,f)$ and $(\mathbf b,g)$, the normal form of their product is obtained by integer matrix multiplication and addition, together with multiplication in the finite group $F$ and the fixed vectors $\mathbf k(f,g)$. Starting with the normal form $q(\mathbf 0)\tau(1)=1$ and reading a word in $X^{\pm1}$ from left to right, repeated use of this multiplication rule therefore computes its unique normal form.

Apply the preceding normal-form procedure to each of the given words $p_1,\ldots,p_r$. Thus, for every $j=1,\ldots,r$, we compute $\mathbf a_j\in\Z^d$ and $f_j\in F$ such that $p_j=q(\mathbf a_j)\tau(f_j)$. Since $q(\mathbf a_j)\in Q_0$, the image of $p_j$ in $F=V/Q_0$ is precisely $f_j$. We denote this image by $\overline p_j=f_j$. Let $\overline P$ be the image of $P$ in~$F$. Starting with $\Omega_0=\{1\}$, define
\[
 \Omega_{n+1}
 =
 \Omega_n\cup
 \{q\overline p_j^{\varepsilon}:q\in\Omega_n,\ 1\leq j\leq r,\ \varepsilon=\pm1\}.
\]
Since $F$ is finite and its multiplication and equality are effective, this sequence can be computed and eventually stabilizes. Its stable value is precisely the subgroup $\overline P=\langle\overline p_1,\ldots,\overline p_r\rangle$. Whenever an element $q$ is reached for the first time, retain the corresponding word $w_q$ in $p_1^{\pm1},\ldots,p_r^{\pm1}$. Thus $w_q$ maps to $q$ in $F$ and we may take $w_1=1$.

Put $P_0=P\cap Q_0$. The restriction to $P$ of the quotient map $V\to F$ has image $\overline P$ and kernel $P_0$. Hence the elements $w_q$, $q\in\overline P$, form a complete set of representatives for the cosets of $P_0$ in $P$. Schreier rewriting gives the finite generating set
\begin{equation}
 \mathcal S=
 \left\{
 w_qp_j^{\varepsilon}
 w_{q\overline p_j^{\varepsilon}}^{-1}:
 q\in\overline P,\ 1\leq j\leq r,\ \varepsilon=\pm1
 \right\}
 \label{eq:5.13}
\end{equation}
for $P_0$. Every element of $\mathcal S$ lies in $Q_0$, so its coordinate vector with respect to $q_1,\ldots,q_d$ is obtained from the normal-form procedure above. Since $Q_0\simeq\Z^d$, the subgroup generated by these finitely many coordinate vectors is a finitely generated free abelian subgroup of $\Z^d$. A basis for this subgroup can be computed effectively. Thus we compute elements $z_1,\ldots,z_\ell\in P_0$ such that $P_0=\langle z_1,\ldots,z_\ell\rangle$, and each $z_i$ is obtained explicitly as a product of powers of the generators in~\eqref{eq:5.13}. Substituting gives each $z_i$ as an explicit word in $p_1^{\pm1},\ldots,p_r^{\pm1}$. Relabel the finitely many words $w_q$ as $c_1,\ldots,c_e$.

Since every $z_i$ and every $z_i^{-1}$ is available as an explicit word in $p_1^{\pm1},\ldots,p_r^{\pm1}$, hypothesis \textup{(v)} and \eqref{eq:5.8} compute their actions on $A$ and the elements $\delta(z_i^{\pm1})$. Likewise, by composing and inverting the rational affine transformations supplied by \textup{(vi)}, we compute the corresponding rational action matrices and the coordinate vectors of $\delta(z_i)\otimes1$. Thus the restriction of the action and of $\delta$ to $P_0$, with basis $z_1,\ldots,z_\ell$, satisfies all the hypotheses of Lemma~\ref{lem:affineabelian}.

\smallskip

We first prove \textup{(a)}. Fix a coset $c_jP_0$ and put $s_j=0*c_j=\delta(c_j)$. Every element of $c_jP_0$ has a unique form $c_jz$ with $z\in P_0$, and by \eqref{eq:5.10}
\[
 \delta(c_jz)=0*(c_jz)=(0*c_j)*z=s_j*z.
\]
The element $s_j$ and the coordinate vector of $s_j\otimes1$ are computable because $c_j$ is an explicit word in the generators $p_1^{\pm1},\ldots,p_r^{\pm1}$. Therefore Lemma~\ref{lem:affineabelian}\textup{(a)}, applied to the restricted action of $P_0$, decides whether there exists $z\in P_0$ such that $s_j*z=a$ and, in the positive case, returns such a $z$ as a word in $z_1^{\pm1},\ldots,z_\ell^{\pm1}$. In that case $u=c_jz$ is an explicit word in $p_1^{\pm1},\ldots,p_r^{\pm1}$ and satisfies $\delta(u)=a$. Repeating this for the finitely many cosets $c_jP_0$ proves \textup{(a)}.

\smallskip

We now prove \textup{(b)}. Write $s=a$. Fix a coset $c_jP_0$. An element of this coset has the form $c_jz$ with $z\in P_0$, and it belongs to $\operatorname{Stab}_P(s)$ precisely when
\begin{equation}
 (s*c_j)*z=s.
 \label{eq:5.22}
\end{equation}
The element $s*c_j$ and the coordinate vector of its image in $A_{\Q}$ are computable from the given data. Hence Lemma~\ref{lem:affineabelian}\textup{(a)} decides whether \eqref{eq:5.22} has a solution $z\in P_0$ and, whenever it does, returns one such $z$. Thus we determine exactly which cosets $c_jP_0$ meet $\operatorname{Stab}_P(s)$ and obtain an explicit element
\[
 r_j=c_jz\in c_jP_0\cap\operatorname{Stab}_P(s)
\]
for every non-empty intersection.

Lemma~\ref{lem:affineabelian}\textup{(b)}, applied to $s$, computes a finite generating set for
\[
 S_0=\operatorname{Stab}_{P_0}(s)=\operatorname{Stab}_P(s)\cap P_0
\]
as words in $z_1^{\pm1},\ldots,z_\ell^{\pm1}$, and hence as words in $p_1^{\pm1},\ldots,p_r^{\pm1}$. Finally, whenever $c_jP_0\cap\operatorname{Stab}_P(s)$ is non-empty, one has
\[
 c_jP_0\cap\operatorname{Stab}_P(s)=r_jS_0.
\]
Indeed, if $g$ belongs to this intersection, then $r_j^{-1}g\in P_0$, and since both $r_j$ and $g$ fix $s$, also $r_j^{-1}g$ fixes $s$; hence $r_j^{-1}g\in S_0$. The reverse inclusion is immediate. Therefore the generators of $S_0$, together with one element $r_j$ from every non-empty coset intersection, form a finite generating set for $\operatorname{Stab}_P(s)$. All of them are retained as explicit words in $p_1^{\pm1},\ldots,p_r^{\pm1}$.
\end{proof}

\subsection{Lifting through the nilpotent subgroup}

Put
$$E=H/K_0,\qquad
 M=N_0/K_0,\qquad
 M_i=N_iK_0/K_0.
$$
Membership in $K_0$ is decidable in $H$, so equality of words in the quotient $E=H/K_0$ is decidable. Since $K_0\leq N_0$, the natural quotient map induces a canonical isomorphism $E/M\simeq H/N_0=P$. We use this isomorphism to identify $E/M$ with $P$ throughout; in particular, the images of the fixed generators are represented by the same words in the two quotients.

Each $N_i$ is characteristic in $N_0$. Thus, since $N_0\trianglelefteq H$ and $K_0\trianglelefteq H$, every $N_iK_0/K_0$ is normal in $E$. Consequently,
\begin{equation}
 M=M_1\geq M_2\geq\ldots\geq M_{c+1}=1,
 \qquad
 M_i/M_{i+1}\simeq A_i
 \label{eq:6.2}
\end{equation}
is a normal series in $E$. The displayed isomorphism is the canonical one induced by
$$
(nK_0)M_{i+1}\longmapsto nN_{i+1}K_0
\qquad (n\in N_iK_0),
$$
and we use it to identify $M_i/M_{i+1}$ with $A_i$ throughout. By \eqref{eq:2.1},
\begin{equation}
 [M,M_i]\leq M_{i+1},
 \label{eq:6.3}
\end{equation}
so, under this identification, $A_i$ is central in $M/M_{i+1}$.

For $x\in E$ let
\begin{equation}
 D_i(x)=\{d\in E:[x,d]\in M_i\},
 \label{eq:6.4}
\end{equation}
be the inverse image of $C_{E/M_i}(xM_i)$, and put $L_i(x)=D_i(x)\cap M$. Let
\begin{equation}
 \Lambda_i(x)=\{n\in N_0:nK_0\in L_i(x)\}
 \label{eq:6.6}
\end{equation}
be its inverse image in $N_0$ under $N_0\twoheadrightarrow M$. Under the fixed identification $E/M=P$, the image $D_i(x)M/M$ is a subgroup of the finitely generated virtually abelian group $P$ and is therefore finitely generated.

For the current element $x$, let $\widehat x\in H$ be the word lift retained by the algorithm. Since $N_0\trianglelefteq H$, conjugation by $\widehat x$ preserves $N_0$ and therefore extends, for every prime $p$, to an automorphism of $(N_0)_p$. For $u\in(N_0)_p$ put
$$
d_x(u)=(u^{\widehat x})^{-1}u.
$$

The induction at stage $i$ maintains the following data and properties:
\begin{enumerate}[label=\textup{(I\arabic*)}]
\item finite generators of $D_i(x)M/M$ and chosen lifts to $D_i(x)$;
\item a membership algorithm for $\Lambda_i(x)$;
\item a finite rational row description of $(\Lambda_i(x))_{\Q}$;
\item for every prime $p$, the identity
\begin{equation}
(\Lambda_i(x))_p
=
\{u\in(N_0)_p:d_x(u)\in(N_iK_0)_p\}.
\label{eq:6.padic-invariant}
\end{equation}
\end{enumerate}

To initialize the induction modulo $M$ and to obtain the corresponding centralizer data, we shall use the following elementary algorithm in the fixed virtually abelian quotient $P=E/M$.

\begin{lemma}\label{lem:virtabelconj}
In the fixed finitely generated virtually abelian group $P=H/N_0$ there are algorithms with the following properties. Given $x,y\in P$ as words in the fixed generators of $P$, one decides whether $x^g=y$ for some $g\in P$ and, in the positive case, returns a word in the fixed generators representing such a $g$. Given $x\in P$ as a word in the fixed generators, one returns, as words in these generators, a finite generating set for $C_P(x)$.
\end{lemma}

\begin{proof}
Since $P$ is fixed, fix once and for all a normal free abelian subgroup $P_0\trianglelefteq P$ of finite index, a basis $q_1,\ldots,q_d$ of $P_0$, and a finite transversal $c_1,\ldots,c_m$ for $P_0$ in $P$, with $c_1=1$.  Fix also the finite quotient $P/P_0$ and the images in it of the fixed generators of $P$.  Thus membership in a coset of $P_0$ is decidable.  We also fix a procedure which, given a word known to represent an element of $P_0$, computes its coordinate vector with respect to $q_1,\ldots,q_d$.  Such a procedure exists since $P$ is fixed: one may enumerate $(n_1,\ldots,n_d)\in\Z^d$ and use the Word Problem in $P$ until $w=q_1^{n_1}\cdots q_d^{n_d}$. The uniqueness of the expression in the basis $q_1,\ldots,q_d$ guarantees termination.

Let $x,y\in P$.  For each $j=1,\ldots,m$, every element of the coset $c_jP_0$ has the form $c_jz$ with $z\in P_0$.  Hence this coset contains an element conjugating~$x$ to $y$ precisely when $x^{c_jz}=y$. If $x^{c_j}P_0\neq yP_0$, which is decidable in the finite quotient $P/P_0$, then $x^{c_jz}=y$ has no solution since $z\in P_0$.  Otherwise $a=(x^{c_j})^{-1}y\in P_0$, and its coordinate vector with respect to $q_1,\ldots,q_d$ is computable by the fixed procedure above. For $z\in P_0$, define
$$
 \theta_j(z)=(x^{c_j})^{-1}(x^{c_j})^z=[x^{c_j},z].
$$
Since $P_0$ is normal, $\theta_j(z)\in P_0$, and since $P_0$ is abelian,
$$
 \theta_j(z_1z_2)=\theta_j(z_1)\theta_j(z_2)
 \qquad(z_1,z_2\in P_0).
$$
Thus $\theta_j:P_0\to P_0$ is a homomorphism.  Its matrix with respect to the basis $q_1,\ldots,q_d$ is computable: for each $k=1,\ldots,d$, compute the word representing $[x^{c_j},q_k]\in P_0$ and then compute its coordinate vector in the fixed basis.  The equation $x^{c_jz}=y$ is therefore equivalent to
\begin{equation}
 \theta_j(z)=a.
 \label{eq:6.11}
\end{equation}
In coordinates, \eqref{eq:6.11} is an integral linear system. Elementary integer linear algebra decides whether it has a solution and, when it does, computes the coordinate vector of one such $z\in P_0$.  Since there are only finitely many cosets $c_jP_0$, this decides whether $x$ and $y$ are conjugate in $P$ and returns a conjugating element when one exists.

For the centralizer, repeat the same calculation with $y=x$. For each $j$, the coset $c_jP_0$ meets $C_P(x)$ precisely when $x^{c_j}P_0=xP_0$ and $\theta_j(z)=(x^{c_j})^{-1}x$ has a solution in $P_0$.  The procedure above decides this and, whenever the intersection is non-empty, returns one element $r_j=c_jz\in C_P(x)\cap c_jP_0$. For the identity coset, the intersection is $C_P(x)\cap P_0=\Ker\theta_1$. The matrix of~$\theta_1$ is already computable, so elementary integer linear algebra computes a basis of this kernel.  Finally, whenever $C_P(x)\cap c_jP_0$ is non-empty and $r_j$ is the representative computed above, one has
\[
 C_P(x)\cap c_jP_0=r_j\bigl(C_P(x)\cap P_0\bigr).
\]
Hence a basis of $C_P(x)\cap P_0$, together with one such $r_j$ from
each non-empty coset intersection, forms a finite generating set for
$C_P(x)$.
\end{proof}

We can now combine the virtually abelian base case with the effective derivation machinery above and lift conjugacy successively through the series \eqref{eq:6.2}.

\begin{theorem}\label{thm:lifting}
There is an algorithm which, given $x,y\in E=H/K_0$ as words in the fixed generators of $E$, decides whether there exists $g\in E$ such that $x^g=y$. In the positive case it returns a word representing one such~$g$. At every level $i$ the algorithm also maintains \textup{(I1)}--\textup{(I4)} for the centralizer of the current element modulo $M_i$.
\end{theorem}

\begin{proof}
We shall apply Theorem~\ref{thm:derivation} with fixed ambient group
$$
V=P=H/N_0\simeq E/M.
$$
Choose once and for all the finite effective data \textup{(i)}--\textup{(iv)} required by Theorem~\ref{thm:derivation} for this fixed group $P$. These data depend only on $P$ and are therefore fixed independently of the inputs and of the stage of the induction.

At the $i$-th stage of the induction, the variable subgroup to which Theorem~\ref{thm:derivation} is applied will be
$$
D_i(x)/L_i(x)\simeq D_i(x)M/M\leq E/M\simeq P,
$$
with generators and chosen lifts supplied by \textup{(I1)}.

\smallskip

First apply Lemma~\ref{lem:virtabelconj} in $E/M\simeq P$ to $xM$ and $yM$. If these elements are not conjugate, the algorithm returns NO. Otherwise Lemma~\ref{lem:virtabelconj} returns a word $w$ in the fixed generators of $E/M$ representing an element $c_0M$ such that $(xM)^{c_0M}=yM$. Replacing in $w$ each generator of $E/M$ by its fixed lift to $E$ gives a word representing an element $c_0\in E$ whose image in $E/M$ is $c_0M$. Let $x_0$ denote the original input $x$. Replace $x$ by the conjugate $x_0^{c_0}$ and put $\gamma=c_0$, so $xM=yM$. The corresponding word representing the current $x$ in $E$ determines, by replacing each generator of $E$ by its fixed lift to $H$, a word lift $\widehat x\in H$.

The centralizer part of Lemma~\ref{lem:virtabelconj}, applied to $xM\in E/M$, gives generators of $D_1(x)/M=C_{E/M}(xM)$ as words in the fixed generators of $E/M$. Replacing in these words each generator of $E/M$ by its fixed lift to $E$ gives elements of $D_1(x)$ whose images generate $D_1(x)/M$. This gives \textup{(I1)} at level~$1$.

Since $M_1=M$, every element of $M$ belongs to $D_1(x)$, and hence
$$
L_1(x)=D_1(x)\cap M=M.
$$
Therefore its inverse image in $N_0$ is $\Lambda_1(x)=N_0$. Membership in $N_0$ is decidable by Lemma~\ref{lem:fixed}\textup{(i)}, and the fixed coordinate elements $u_{0,1},\ldots,u_{0,h_0}$ form a finite rational row description of $(N_0)_{\Q}$ by Lemma~\ref{lem:coordinates}. Hence \textup{(I2)}--\textup{(I3)} hold at level $1$.

Finally, \textup{(I4)} also holds at level $1$. Indeed, $N_1=N_0$ and $K_0\leq N_0$, so $N_1K_0=N_0$. Since $N_0\trianglelefteq H$, conjugation by the fixed lift $\widehat x$ preserves $(N_0)_p$. Thus, for every $u\in(N_0)_p$,
$$
d_x(u)=(u^{\widehat x})^{-1}u\in(N_0)_p=(N_1K_0)_p.
$$
Consequently
$$
(\Lambda_1(x))_p=(N_0)_p=\{u\in(N_0)_p:d_x(u)\in(N_1K_0)_p\},
$$
which is \eqref{eq:6.padic-invariant} for $i=1$. We therefore begin the induction with
\begin{equation}
xM_1=yM_1
\quad\text{and data \textup{(I1)}--\textup{(I4)} for }D_1(x).
\label{eq:6.13}
\end{equation}

Assume that, for some $i$, the current element $x$ satisfies $xM_i=yM_i$ and that \textup{(I1)}--\textup{(I4)} have been constructed for $D_i(x)$. We pass from $M_i$ to $M_{i+1}$ by solving one explicit lifting equation.

We use throughout the natural identification $M_i/M_{i+1}\simeq A_i=V_i/W_i$ induced by $nN_{i+1}K_0\mapsto(nN_{i+1})+W_i$ for $n\in N_i$. Put
\begin{equation}
a=x^{-1}yM_{i+1}\in A_i.
\label{eq:6.14}
\end{equation}
Lift the word $x^{-1}y$ to the fixed generators of $H$, and let $h\in H$ be the element represented by this word. Since $xM_i=yM_i$, we have $h\in N_iK_0$. The fixed subgroups $N_i$ and $K_0$ are recursively enumerable. Dovetail their enumerations and use the Word Problem in $H$ until a pair $n\in N_i$, $k\in K_0$ satisfying $h=nk$ is found. Such a pair exists, so the search terminates. Under the identification above, $a=nN_{i+1}K_0\in A_i$. This element is independent of the successful decomposition: if $nk=n'k'$, then $n^{-1}n'=kk'^{-1}\in N_i\cap K_0=J_i$, and hence $n$ and $n'$ have the same image in $V_i/W_i$. Lemma~\ref{lem:fixed}\textup{(i)} computes the coordinate vector of $nN_{i+1}$ in $(V_i)_{\Q}$ with respect to $s_{i,1},\ldots,s_{i,r_i}$. Applying the fixed quotient matrix $R_i^{\mathrm{quot}}$ then gives the coordinate vector of $a\otimes1$ in the fixed basis of $(A_i)_{\Q}\simeq(V_i)_{\Q}/(W_i)_{\Q}$.

Put $D=D_i(x)$ and $L=L_i(x)$. Since $d\in D$ implies $x^{-1}x^d\in M_i$, the formula
$$
\delta_i:D\longrightarrow A_i,\qquad \delta_i(d)=x^{-1}x^dM_{i+1}
$$
is well-defined. The group $D$ acts on $A_i$ by conjugation. For $d,e\in D$, direct calculation gives
\begin{equation}
\delta_i(de)=\delta_i(d)^e+\delta_i(e),
\label{eq:6.16}
\end{equation}
so $\delta_i$ is a derivation. Since $L\leq M$ and $[M,M_i]\leq M_{i+1}$ by \eqref{eq:6.3}, the action of $L$ on $A_i$ is trivial. Consequently
$$
\delta_i(\ell_1\ell_2)=\delta_i(\ell_1)+\delta_i(\ell_2)\qquad(\ell_1,\ell_2\in L).
$$

For $n\in\Lambda_i(x)$, the element $d_x(n)$ introduced above is represented in $H$ by
\begin{equation}
d_x(n)=(n^{\widehat x})^{-1}n=\widehat x^{-1}n^{-1}\widehat x n.
\label{eq:6.18}
\end{equation}
Its image in $E$ is $x^{-1}x^{\,nK_0}$. Since $nK_0\in L_i(x)\leq D_i(x)$, the definition of~$D_i(x)$ gives $x^{-1}x^{\,nK_0}\in M_i$. On the other hand, the image of $d_x(n)$ in $E=H/K_0$ is $d_x(n)K_0=x^{-1}x^{\,nK_0}$. Hence $d_x(n)K_0\in M_i=N_iK_0/K_0$, and therefore $d_x(n)\in N_iK_0$. Dovetail the fixed enumerations of $N_i$ and $K_0$ and use the Word Problem in $H$ until $d_x(n)=qk$ for some $q\in N_i$ and $k\in K_0$. The search terminates, and we define
\begin{equation}
\widetilde\delta_i(n)=qN_{i+1}K_0\in A_i.
\label{eq:6.20}
\end{equation}
This value is independent of the successful decomposition. Indeed, if $qk=q'k'$, then $q^{-1}q'=kk'^{-1}\in N_i\cap K_0=J_i$, whose image in $V_i=N_i/N_{i+1}$ lies in $W_i=J_iN_{i+1}/N_{i+1}$.

If $\ell=nK_0\in L$, then \eqref{eq:6.18}--\eqref{eq:6.20} show, under the fixed identification of $M_i/M_{i+1}$ with $A_i$, that $\widetilde\delta_i(n)=\delta_i(\ell)$. Hence
$$
\widetilde\delta_i(n_1n_2)=\widetilde\delta_i(n_1)+\widetilde\delta_i(n_2),
$$
and
\begin{equation}
I_i:=\delta_i(L)
=\widetilde\delta_i(\Lambda_i(x))
\leq A_i.
\label{eq:6.22}
\end{equation}
By \textup{(I2)}, the elements of $\Lambda_i(x)$ can be effectively enumerated, and for each such element $n$, the procedure \eqref{eq:6.18}--\eqref{eq:6.20} computes $\widetilde\delta_i(n)$, so $I_i$ is recursively enumerable.

The subgroup $I_i$ is $D$-invariant. In fact, if $\ell\in L=D\cap M$ and $d\in D$, then $\ell^d\in L$, and using \eqref{eq:6.16}, we obtain
$$
\begin{array}{c}
\delta_i(\ell^d)
=\delta_i(d^{-1}\ell d)
=\delta_i(d^{-1})^{\ell d}+\delta_i(\ell)^d+\delta_i(d)
\\[0.2cm]
=\delta_i(d^{-1})^d+\delta_i(\ell)^d+\delta_i(d)
=\delta_i(\ell)^d.
\end{array}
$$

We first compute the rational data needed to decide membership in $I_i$. By \textup{(I2)}, the subgroup $\Lambda_i(x)$ is recursively enumerable, while \textup{(I3)} gives a finite rational row description of $(\Lambda_i(x))_{\Q}$. Moreover, given the coordinate vector of an element $n\in\Lambda_i(x)$ with respect to $u_{0,1},\ldots,u_{0,h_0}$, one can recover a word representing $n$ as follows. Enumerate the elements of $N_0$ as words in the fixed generators of $H$. For each such word, Lemma~\ref{lem:wordcoords} computes its coordinate vector with respect to $u_{0,1},\ldots,u_{0,h_0}$. Comparing these vectors with the given one, the search terminates when the word representing $n$ is reached, by uniqueness of the coordinate expression in Lemma~\ref{lem:coordinates}. The procedure \eqref{eq:6.18}--\eqref{eq:6.20} then computes $\widetilde\delta_i(n)$ together with a word representing an element $q\in N_i$ such that $\widetilde\delta_i(n)=qN_{i+1}K_0$. Lemma~\ref{lem:fixed}\textup{(i)} computes the coordinate vector of $qN_{i+1}$ in $(V_i)_{\Q}$ with respect to $s_{i,1},\ldots,s_{i,r_i}$, and applying the fixed quotient matrix gives the coordinate vector of $\widetilde\delta_i(n)\otimes1$ in the fixed basis of $(A_i)_{\Q}\simeq(V_i)_{\Q}/(W_i)_{\Q}$. Thus all the hypotheses of Lemma~\ref{lem:rationalkernel} are satisfied for $\widetilde\delta_i:\Lambda_i(x)\to A_i$, and that lemma computes
\begin{equation}
(\Ker\widetilde\delta_i)_{\Q}
\quad\text{and}\quad
(I_i)_{\Q}.
\label{eq:6.24}
\end{equation}

Fix a prime $p$. Tensor exactness applied to the epimorphism \hbox{$\widetilde\delta_i:\Lambda_i(x)\twoheadrightarrow I_i$,} together with Lemma~\ref{lem:tensor}\textup{(i)} for $I_i\leq A_i$, gives
\begin{equation}
\widetilde\delta_{i,p}:
(\Lambda_i(x))_p\longrightarrow(A_i)_p
\label{eq:6.25}
\end{equation}
with image exactly $(I_i)_p$. For $u\in(\Lambda_i(x))_p$, define, using the same word lift $\widehat x$, $d_x(u)=(u^{\widehat x})^{-1}u$. By the definition of tensor $p$-completion and its direct-limit topology, $\Lambda_i(x)$ is dense in $(\Lambda_i(x))_p$. Since $d_x(n)\in N_iK_0$ for every $n\in\Lambda_i(x)$, continuity of conjugation by $\widehat x$, inversion and multiplication, together with the closedness of $(N_iK_0)_p$ in $(N_0)_p$, therefore gives
$$
d_x(u)\in(N_iK_0)_p\qquad\bigl(u\in(\Lambda_i(x))_p\bigr).
$$
By Lemma~\ref{lem:padicproduct}, write $d_x(u)=qk$ with $q\in(N_i)_p$ and $k\in(K_0)_p$. On $\Lambda_i(x)$, the map $\widetilde\delta_i$ is obtained by composing $d_x$ with the natural map
$$
N_iK_0\longrightarrow A_i=N_iK_0/N_{i+1}K_0.
$$
Hence its continuous extension $\widetilde\delta_{i,p}$ is obtained in the same way on $(\Lambda_i(x))_p$. Since the factor $k\in(K_0)_p$ maps trivially to $(A_i)_p$, and under the fixed identification
$$
(A_i)_p\simeq(V_i)_p/(W_i)_p,
$$
the image of $q$ is represented by $\bigl(q(N_{i+1})_p\bigr)+(W_i)_p$, so
\begin{equation}
\widetilde\delta_{i,p}(u)
=\bigl(q(N_{i+1})_p\bigr)+(W_i)_p.
\label{eq:6.28}
\end{equation}

We now give a terminating membership test for $(I_i)_p$. Let $b\in A_i$ be given in its fixed representation, say $b=n_bN_{i+1}K_0$ with $n_b\in N_i$. We claim that
\begin{equation}
\begin{split}
b\in(I_i)_p
\quad\Longleftrightarrow\quad
&\text{there exist }u\in(N_0)_p,\ q\in(N_{i+1})_p,\
k\in(K_0)_p\\
&\text{such that }d_x(u)=n_bqk.
\end{split}
\label{eq:6.padic-image-test}
\end{equation}

Suppose first that $b\in(I_i)_p$. By \eqref{eq:6.25}, choose $u\in(\Lambda_i(x))_p$ such that $\widetilde\delta_{i,p}(u)=b$. Since $\widetilde\delta_{i,p}(u)$ is the image of $d_x(u)$ in $(A_i)_p$, while $b$ is the image of $n_b$, we have
$$
n_b^{-1}d_x(u)\in(N_{i+1}K_0)_p.
$$
Lemma~\ref{lem:padicproduct}, applied at level $i+1$, gives $n_b^{-1}d_x(u)=qk$ for some $q\in(N_{i+1})_p$ and $k\in(K_0)_p$, and hence $d_x(u)=n_bqk$.

Conversely, suppose that $d_x(u)=n_bqk$ with $u\in(N_0)_p$, $q\in(N_{i+1})_p$ and $k\in(K_0)_p$. Since $n_b\in N_i$ and $N_{i+1}\leq N_i$, this equality implies $d_x(u)\in(N_iK_0)_p$. The induction hypothesis \eqref{eq:6.padic-invariant} therefore gives $u\in(\Lambda_i(x))_p$. Moreover, $n_b^{-1}d_x(u)=qk\in(N_{i+1}K_0)_p$, so $d_x(u)$ and $n_b$ have the same image in $(A_i)_p$. Hence $\widetilde\delta_{i,p}(u)=b$, and therefore $b\in(I_i)_p$. The claim is proved.

Apply Lemma~\ref{lem:fixedcharts} to the three fixed subgroups $N_0$, $N_{i+1}$ and $K_0$. For the given prime $p$, write the corresponding surjective parameterizations as
$$
\Theta_{N_0,p},\qquad \Theta_{N_{i+1},p},\qquad \Theta_{K_0,p}.
$$
Thus the variables $u$, $q$ and $k$ occurring in \eqref{eq:6.padic-image-test} may be replaced by the corresponding $\Qp$- and $\Zp$-parameters in these three maps. Since the parameterizations are surjective, this replacement does not change the existence condition in \eqref{eq:6.padic-image-test}.

In the fixed $\Qp$-extension of the coordinate system of $N_0$, the coordinate vectors of the three parameterizations are polynomial functions of their parameters by Lemma~\ref{lem:fixedcharts}. Moreover, multiplication and inversion are given by the fixed coordinate polynomials. Conjugation by the current word $\widehat x$ is also given effectively by a polynomial map, obtained by composing the fixed conjugation maps associated with the letters of $\widehat x$. Finally, the coordinate vector of $n_b\in N_i\leq N_0$ is computable from its fixed word representative. Consequently, after substituting the three parameterizations, the equality
$$
d_x(u)=n_bqk
$$
is equivalent to a finite system of polynomial equations over $\Qp$, obtained by equating the corresponding coordinate vectors.

Put
$$
\Sigma=\bigcup_{Y\in\mathcal Y}\Sigma_Y.
$$
Since $\mathcal Y$ is finite and each $\Sigma_Y$ is finite by Lemma~\ref{lem:fixedcharts}, the set $\Sigma$ is fixed and finite. For $p\notin\Sigma$, Lemma~\ref{lem:fixedcharts} gives $r(Y,p)=0$ for each of the three groups $Y=N_0,N_{i+1},K_0$, and the fixed elements occurring in their parameterizations have rational coordinate vectors. The coordinate polynomials for multiplication, inversion and conjugation have rational coefficients, and the coordinate vector of $n_b$ is also rational and effectively computable. Hence the polynomial system obtained from \eqref{eq:6.padic-image-test} has coefficients in $\Q$.

Now suppose that $p\in\Sigma$. For each $Y\in\mathcal Y$, Lemma~\ref{lem:fixedcharts} places the coordinates of the finitely many elements occurring in $\Theta_{Y,p}$ in a finitely generated exact subfield of $\Qp$ of the type fixed in Convention~\ref{def:fixeddata}. Since $\mathcal Y$ is finite, the same lemma allows us, for this fixed prime $p$, to replace these fields by one common finitely generated exact coefficient field $E_p\leq\Qp$. All remaining coefficients entering the polynomial system, those coming from the fixed multiplication, inversion and conjugation polynomials and from the coordinate vector of $n_b$, are rational, and therefore also belong to $E_p$.

Thus, if we put $E_p=\Q$ for $p\notin\Sigma$, then for every prime $p$ the system corresponding to \eqref{eq:6.padic-image-test} consists of finitely many polynomial equations with coefficients given in the fixed exact representation of $E_p$. The parameters arising from the $\Qp$-factors in the three parameterizations are unrestricted variables in $\Qp$, while those arising from the $\Zp$-factors are required to belong to $\Zp$. These are exactly the hypotheses of Lemma~\ref{lem:padicdecision}. Consequently, with the current lift $\widehat x$ fixed at this stage of the induction, Lemma~\ref{lem:padicdecision} gives a single terminating procedure which, on input $p$ and the represented element $b\in A_i$, decides whether \eqref{eq:6.padic-image-test} has a solution, and hence whether $b\in(I_i)_p$.

We now verify the hypotheses needed to decide membership in $I_i$. We represent an element of
$$
A_i=N_iK_0/N_{i+1}K_0\simeq V_i/W_i
$$
by a word representing some $n\in N_i$, with the corresponding element of $A_i$ written $nN_{i+1}K_0$. Products and inverses are computed by multiplying and inverting such representatives. If $a=nN_{i+1}K_0$ and $a'=n'N_{i+1}K_0$, then
$$
a=a'\quad\Longleftrightarrow\quad n^{-1}n'\in N_{i+1}K_0\quad\Longleftrightarrow\quad(n^{-1}n')N_{i+1}\in W_i.
$$
Lemma~\ref{lem:fixed}\textup{(i)} computes the coordinates of $(n^{-1}n')N_{i+1}$ in $V_i$ and decides membership in $W_i$. Thus $A_i$ has a fixed effective representation in which equality is decidable. Moreover, $A_i$ is abelian minimax.

We apply Lemma~\ref{lem:ratpadicmembership} with $A=A_i$ and $B=I_i$. The subgroup $I_i$ is recursively enumerable. For the rational membership test, Lemma~\ref{lem:rationalkernel} has computed in \eqref{eq:6.24} a basis of the $\Q$-vector space $(I_i)_{\Q}\leq(A_i)_{\Q}$. Given $b\in A_i$, represented by $n_b\in N_i$, Lemma~\ref{lem:fixed}\textup{(i)}, together with the fixed quotient matrix $(V_i)_{\Q}\to(V_i)_{\Q}/(W_i)_{\Q}\simeq(A_i)_{\Q}$, computes the coordinate vector of $b\otimes1$ in $(A_i)_{\Q}$. Membership of $b\otimes1$ in $(I_i)_{\Q}$ is therefore decidable by rational linear algebra. For every prime $p$, membership of $b$ in $(I_i)_p$ is decided, uniformly in $p$ and $b$, by \eqref{eq:6.padic-image-test} and the procedure constructed above. All the hypotheses of Lemma~\ref{lem:ratpadicmembership} are therefore satisfied, and membership in~$I_i$ is decidable.

\smallskip

It follows that $A_i/I_i$ also has an effective representation: an element is represented by a representative in $A_i$, the group operations are induced by those of $A_i$, and $a+I_i=a'+I_i$ if and only if $a'-a\in I_i$, which is decidable by the membership algorithm just obtained. Since $A_i$ is abelian minimax, so is~$A_i/I_i$. Moreover, exactness of tensoring with $\Q$ gives
$$
(A_i/I_i)_{\Q}\simeq(A_i)_{\Q}/(I_i)_{\Q}.
$$
Also, the subgroup $I_i$ is invariant under the action of $D$ on $A_i$, so this action descends to $A_i/I_i$. Furthermore, $L$ acts trivially on $A_i$, and hence the induced action on $A_i/I_i$ factors through $D/L$.

Now, if $dL=d'L$, then $d'=d\ell$ for some $\ell\in L$, and the derivation identity \eqref{eq:6.16}, together with the trivial action of $L$ on $A_i$, gives
$$
\delta_i(d')=\delta_i(d\ell)=\delta_i(d)+\delta_i(\ell).
$$
Since $\delta_i(\ell)\in I_i$, the coset $\delta_i(d)+I_i$ depends only on $dL$. Therefore
\begin{equation}
\bar\delta_i:D/L\longrightarrow A_i/I_i,
\qquad
\bar\delta_i(dL)=\delta_i(d)+I_i
\label{eq:6.29}
\end{equation}
is well-defined. The identity \eqref{eq:6.16}, together with the induced action of $D/L$ on $A_i/I_i$, shows that $\bar\delta_i$ is a derivation.

The quotient map $D\to E/M\simeq P$ has kernel $D\cap M=L$. Hence the map $dL\mapsto dM$ identifies $D/L$ with the finitely generated subgroup $DM/M\leq P$. The generators and the words representing them in the fixed generators of~$P$ are precisely those stored in \textup{(I1)}. We now verify the remaining effective data required to apply Theorem~\ref{thm:derivation} with the fixed ambient group $V=P$ and the subgroup $D/L\leq P$.

For every stored generator $dL$ of $D/L$, a word lift $d\in D\leq E$ is known. Conjugation by $d$ and by $d^{-1}$ on $A_i$ is therefore computable from words. Moreover, since $d,d^{-1}\in D$, the elements $x^{-1}x^d$ and $x^{-1}x^{d^{-1}}$ belong to $M_i$; applying the same effective decomposition in $N_iK_0$ used above to define $\widetilde\delta_i$ computes $\delta_i(d)$ and $\delta_i(d^{-1})$. Passing modulo $I_i$ gives the action and the values of $\bar\delta_i$ required in the effective representation of $A_i/I_i$.

Rationally, Lemma~\ref{lem:fixed}\textup{(i)} supplies a fixed basis of $(A_i)_{\Q}\simeq(V_i)_{\Q}/(W_i)_{\Q}$ and the corresponding quotient map. From the fixed coordinate polynomials on $(N_0)_{\Q}$ and the word representing $d$, we compute the matrix of conjugation by $d$ on $(A_i)_{\Q}$ and the coordinate vector of $\delta_i(d)\otimes1$. Equation \eqref{eq:6.24} gives a basis of the $D$-invariant subspace $(I_i)_{\Q}$. Rational linear algebra therefore computes a basis of $(A_i)_{\Q}/(I_i)_{\Q}$, the corresponding quotient matrix, the induced action matrices, and the coordinate vectors of $\bar\delta_i(dL)\otimes1$.

The same quotient matrix, applied to the coordinate vector of $a$ computed after \eqref{eq:6.14}, gives the coordinate vector of $(a+I_i)\otimes1$ in the chosen basis of~$(A_i/I_i)_{\Q}$. Thus all the hypotheses, including the required coordinate data for the input element $a+I_i$, of Theorem~\ref{thm:derivation}\textup{(a)} are satisfied.

We now decide whether there exists $g\in E$ such that $x^gM_{i+1}=yM_{i+1}$. It is enough to consider elements of $D=D_i(x)$. Indeed, if $g\in E$ satisfies $x^gM_{i+1}=yM_{i+1}$, then, since $M_{i+1}\leq M_i$ and the induction hypothesis gives $xM_i=yM_i$, we obtain
$$
x^gM_i=yM_i=xM_i.
$$
Equivalently, $x^{-1}x^g\in M_i$, so by the definition of $D_i(x)$ we have $g\in D$.

For $d\in D$, recall that
$$
\delta_i(d)=x^{-1}x^dM_{i+1}\qquad\text{and}\qquad a=x^{-1}yM_{i+1}.
$$
Hence
\begin{equation}
x^dM_{i+1}=yM_{i+1}
\quad\Longleftrightarrow\quad
\delta_i(d)=a.
\label{eq:6.30}
\end{equation}
We claim that
\begin{equation}
\exists\,d\in D\text{ such that }\delta_i(d)=a
\quad\Longleftrightarrow\quad
a+I_i\in\bar\delta_i(D/L).
\label{eq:6.31}
\end{equation}
For the forward implication, if $\delta_i(d)=a$, then by the definition of $\bar\delta_i$,
$$
\bar\delta_i(dL)=\delta_i(d)+I_i=a+I_i.
$$
Conversely, suppose that $\bar\delta_i(dL)=a+I_i$ for some $dL\in D/L$. Then
$$
a-\delta_i(d)\in I_i.
$$
By \eqref{eq:6.22}, $I_i=\delta_i(L)$, so there exists $\ell\in L$ such that
$$
\delta_i(\ell)=a-\delta_i(d).
$$
Since $L$ acts trivially on $A_i$, the derivation identity \eqref{eq:6.16} gives
$$
\delta_i(d\ell)=\delta_i(d)^\ell+\delta_i(\ell)=\delta_i(d)+\delta_i(\ell)=a.
$$
Thus $d\ell\in D$ satisfies \eqref{eq:6.31}.

By the preceding verification, all the hypotheses of Theorem~\ref{thm:derivation}\textup{(a)} are satisfied with fixed ambient group $V=P$, subgroup $D/L\leq P$, abelian minimax group $A_i/I_i$, derivation $\bar\delta_i$, and target element $a+I_i$, together with the computed coordinate vector of $(a+I_i)\otimes1$. The theorem therefore decides whether \eqref{eq:6.31} holds. If the answer is negative, then by \eqref{eq:6.31} and \eqref{eq:6.30} there is no $d\in D$ satisfying $x^dM_{i+1}=yM_{i+1}$; by the first paragraph, there is consequently no $g\in E$ satisfying this equality.

Suppose now that the answer is positive. Theorem~\ref{thm:derivation} returns a word in the stored generators of $D/L$ representing a coset $dL$ such that
$$
\bar\delta_i(dL)=a+I_i.
$$
Evaluating this word on the chosen lifts stored in \textup{(I1)} produces a word representing an element $d\in D$ with the prescribed coset $dL$. By the definition of $\bar\delta_i$,
\begin{equation}
a-\delta_i(d)\in I_i.
\label{eq:6.32}
\end{equation}

We now construct an element of $L$ whose image under $\delta_i$ is $a-\delta_i(d)$. Use the fixed effective enumeration of $N_0$ by words in the generators $h_1^{\pm1},\ldots,h_m^{\pm1}$ of $H$, and retain exactly those enumerated elements that belong to $\Lambda_i(x)$, using the membership algorithm \textup{(I2)}. For each retained element $n$, compute $\widetilde\delta_i(n)$ by the procedure \eqref{eq:6.18}--\eqref{eq:6.20}, and compare it in $A_i$ with the computable element $a-\delta_i(d)$. Equality in $A_i$ is decidable by the procedure established above. Stop when
\begin{equation}
\widetilde\delta_i(n)=a-\delta_i(d).
\label{eq:6.33}
\end{equation}
This search terminates. For the resulting $n\in\Lambda_i(x)$, put $\ell=nK_0\in L$. By the definition of $\widetilde\delta_i$,
$$
\delta_i(\ell)=\widetilde\delta_i(n).
$$
Since $\ell\in L$ and $L$ acts trivially on $A_i$, \eqref{eq:6.16} and \eqref{eq:6.33} give
$$
\delta_i(d\ell)=\delta_i(d)^\ell+\delta_i(\ell)=\delta_i(d)+\widetilde\delta_i(n)=a.
$$
Therefore, with $t=d\ell\in D$, equation~\eqref{eq:6.30} gives
$$
x^tM_{i+1}=yM_{i+1}.
$$
A word representing $t$ in $E$ is obtained by multiplying the word representing $d$ by the image in $E=H/K_0$ of the word representing $n$.

\smallskip

It remains to construct the data \textup{(I1)}--\textup{(I4)} at level $i+1$. We first determine the image of $D_{i+1}(x)$ in $D/L$. Recall that $D=D_i(x)$ and $L=D\cap M$. Since $M_{i+1}\leq M_i$, one has $D_{i+1}(x)\leq D$. For $d\in D$, the coset $dL$ contains an element of $D_{i+1}(x)$ if and only if
\begin{equation}
\bar\delta_i(dL)=0.
\label{eq:6.35}
\end{equation}
Indeed, every representative of $dL$ has the form $d\ell$ with $\ell\in L$, and $d\ell\in D_{i+1}(x)$ is equivalent, by the definition of $\delta_i$, to $\delta_i(d\ell)=0$. Since $L$ acts trivially on $A_i$, \eqref{eq:6.16} gives $\delta_i(d\ell)=\delta_i(d)+\delta_i(\ell)$. Thus such an $\ell$ exists if and only if $-\delta_i(d)\in\delta_i(L)=I_i$, equivalently $\delta_i(d)\in I_i$, which is precisely $\bar\delta_i(dL)=0$.

We now apply Theorem~\ref{thm:derivation}\textup{(b)} in order to compute the subgroup of $D/L$ consisting of the cosets which contain an element of $D_{i+1}(x)$. Recall that the action of $D/L$ on $A_i/I_i$ is the action induced by conjugation: for $b\in A_i$ and $d\in D$,
$$
(b+I_i)^{dL}=b^d+I_i.
$$
Moreover, by \eqref{eq:6.29}, $\bar\delta_i(dL)=\delta_i(d)+I_i$ is a derivation for this action. The action and the derivation $\bar\delta_i$ define a right affine action of $D/L$ on $A_i/I_i$ by
$$
(b+I_i)*(dL):=(b+I_i)^{dL}+\bar\delta_i(dL)=b^d+\delta_i(d)+I_i.
$$
Write $0$ for the identity element $I_i$ of the additive group $A_i/I_i$. Since every automorphism of $A_i/I_i$ fixes $0$, the preceding formula gives, for every $dL\in D/L$,
$$
0*(dL)=0^{dL}+\bar\delta_i(dL)=\bar\delta_i(dL).
$$
Consequently,
$$
\operatorname{Stab}_{D/L}(0)=\{dL\in D/L:0*(dL)=0\}=\{dL\in D/L:\bar\delta_i(dL)=0\}=\Ker\bar\delta_i.
$$

The effective hypotheses of Theorem~\ref{thm:derivation}\textup{(b)} have already been verified with fixed ambient group $V=P$, subgroup $D/L\leq P$, abelian minimax group $A_i/I_i$, and derivation $\bar\delta_i$. For the present application the input element of $A_i/I_i$ is $0$, and its coordinate vector in the fixed rational basis of $(A_i/I_i)_{\Q}$ is the zero vector. Therefore Theorem~\ref{thm:derivation}\textup{(b)} computes, as words in the stored generators of $D/L$, a finite generating set
$$
c_1L,\ldots,c_rL
$$
for $\operatorname{Stab}_{D/L}(0)=\Ker\bar\delta_i$.

For each $j$, evaluate the returned word on the chosen lifts stored in \textup{(I1)}, obtaining an explicit word representing $c_j\in D$. Since $c_jL\in\Ker\bar\delta_i$, one has $-\delta_i(c_j)\in I_i$. Enumerate $N_0$ by words in the fixed generators of $H$, retain exactly the elements belonging to $\Lambda_i(x)$ by \textup{(I2)}, and compute $\widetilde\delta_i$ on each retained element using \eqref{eq:6.18}--\eqref{eq:6.20}. Equality in $A_i$ is decidable by Lemma~\ref{lem:fixed}\textup{(i)}, through the corresponding membership test in $W_i$. Hence we may stop when an element $n_j\in\Lambda_i(x)$ satisfies
\begin{equation}
\widetilde\delta_i(n_j)=-\delta_i(c_j).
\label{eq:6.36}
\end{equation}
The search terminates because $I_i=\widetilde\delta_i(\Lambda_i(x))$ by \eqref{eq:6.22}. Put $\ell_j=n_jK_0\in L$ and $e_j=c_j\ell_j$. Since $\delta_i(\ell_j)=\widetilde\delta_i(n_j)$ and $L$ acts trivially on $A_i$, \eqref{eq:6.16} and \eqref{eq:6.36} give $\delta_i(e_j)=0$. Thus $e_j\in D_{i+1}(x)$, and the words for $c_j$ and $n_j$ give an explicit word representing $e_j$ in $E$.

The elements $e_1,\ldots,e_r$ give the data required in \textup{(I1)}. Indeed, the map $dL\mapsto dM$ identifies $D/L$ with $DM/M$. Under this identification, the image of $D_{i+1}(x)$ is exactly $\Ker\bar\delta_i$: if $e\in D_{i+1}(x)$, then $\delta_i(e)=0$, so $eL\in\Ker\bar\delta_i$; conversely, \eqref{eq:6.35} shows that every coset in $\Ker\bar\delta_i$ contains an element of~$D_{i+1}(x)$. Therefore the images $e_1M,\ldots,e_rM$ generate $D_{i+1}(x)M/M$.

We now prove that $\Lambda_{i+1}(x)=\Ker\widetilde\delta_i$. Since $D_{i+1}(x)\leq D_i(x)$, one has $L_{i+1}(x)\leq L_i(x)$ and therefore $\Lambda_{i+1}(x)\leq\Lambda_i(x)$. Let $n\in\Lambda_i(x)$ and put $\ell=nK_0\in L_i(x)=L$. By the definition of $\widetilde\delta_i$, $\widetilde\delta_i(n)=\delta_i(\ell)$. Since $\ell\in M$, we have
$$
n\in\Lambda_{i+1}(x)\quad\Longleftrightarrow\quad\ell\in L_{i+1}(x)\quad\Longleftrightarrow\quad\ell\in D_{i+1}(x)\quad\Longleftrightarrow\quad\delta_i(\ell)=0.
$$
Hence
\begin{equation}
\Lambda_{i+1}(x)=\Ker\widetilde\delta_i.
\label{eq:6.37}
\end{equation}
The finite rational row description required in \textup{(I3)} at level $i+1$ is therefore the finite rational row description of $(\Ker\widetilde\delta_i)_{\Q}$ already computed in \eqref{eq:6.24}.

We now construct the membership algorithm required in \textup{(I2)}. The subgroup $N_{i+1}K_0$ is fixed independently of the input $x,y$. It is recursively enumerable, since both $N_{i+1}$ and $K_0$ are recursively enumerable and products $qk$, with $q\in N_{i+1}$ and $k\in K_0$, enumerate $N_{i+1}K_0$. Since $H$ has a recursive presentation, \cite[Theorem~9.3.4]{LR04} therefore decides membership in this fixed subgroup. Given a word $n$, first decide whether it represents an element of~$N_0$. If not, then $n\notin\Lambda_{i+1}(x)$. If $n\in N_0$, compute $d_x(n)$ from the maintained word lift $\widehat x$. Its image in $E=H/K_0$ is $x^{-1}x^{\,nK_0}$, and hence
$$
d_x(n)\in N_{i+1}K_0\quad\Longleftrightarrow\quad x^{-1}x^{\,nK_0}\in M_{i+1}\quad\Longleftrightarrow\quad nK_0\in D_{i+1}(x).
$$
Since $nK_0\in M$, the last condition is equivalent to $nK_0\in L_{i+1}(x)$. Thus $n\in\Lambda_{i+1}(x)$ if and only if $d_x(n)\in N_{i+1}K_0$, giving the required membership algorithm.

It remains to verify \textup{(I4)} at level $i+1$. By \eqref{eq:6.37}, the homomorphism $\widetilde\delta_i:\Lambda_i(x)\to I_i$ is surjective and has kernel $\Lambda_{i+1}(x)$. Hence
$$
1\longrightarrow\Lambda_{i+1}(x)\longrightarrow\Lambda_i(x)\overset{\widetilde\delta_i}{\longrightarrow}I_i\longrightarrow1
$$
is exact. Since $\Lambda_i(x)$ is a nilpotent minimax group, Lemma~\ref{lem:tensor}\textup{(ii)} gives $(\Lambda_i(x))_p/(\Lambda_{i+1}(x))_p\simeq(I_i)_p$. Moreover, Lemma~\ref{lem:tensor}\textup{(i)} identifies $(I_i)_p$ with its image in $(A_i)_p$. By uniqueness of the continuous extension in Lemma~\ref{lem:tensor}\textup{(iii)}, the resulting map $(\Lambda_i(x))_p\to(A_i)_p$ is precisely $\widetilde\delta_{i,p}$. Therefore
$$
(\Lambda_{i+1}(x))_p=\Ker\widetilde\delta_{i,p}.
$$

Let $u\in(\Lambda_i(x))_p$, so $d_x(u)\in(N_iK_0)_p$. Let $\pi_i:N_iK_0\to A_i=N_iK_0/N_{i+1}K_0$ be the quotient map. On $\Lambda_i(x)$ one has $\widetilde\delta_i=\pi_i\circ d_x$, and both sides extend continuously to $(\Lambda_i(x))_p$. Hence $\widetilde\delta_{i,p}(u)=(\pi_i)_p(d_x(u))$. Since $N_{i+1}K_0\trianglelefteq N_iK_0$, Lemma~\ref{lem:tensor}\textup{(ii)} applied to the quotient map $\pi_i$ gives $\Ker(\pi_i)_p=(N_{i+1}K_0)_p$. Thus
$$
\widetilde\delta_{i,p}(u)=0\quad\Longleftrightarrow\quad d_x(u)\in(N_{i+1}K_0)_p.
$$

It remains only to remove the restriction $u\in(\Lambda_i(x))_p$. By Lemma~\ref{lem:tensor}\textup{(i)}, the inclusion $N_{i+1}K_0\leq N_iK_0$ induces $(N_{i+1}K_0)_p\leq(N_iK_0)_p$. Hence, if $u\in(N_0)_p$ satisfies $d_x(u)\in(N_{i+1}K_0)_p$, then $d_x(u)\in(N_iK_0)_p$, and the induction hypothesis \eqref{eq:6.padic-invariant} implies $u\in(\Lambda_i(x))_p$. Combining this observation with the two equalities above gives
\begin{equation}
(\Lambda_{i+1}(x))_p
=
\{u\in(N_0)_p:
d_x(u)\in(N_{i+1}K_0)_p\}.
\label{eq:6.padic-next}
\end{equation}
This is precisely \textup{(I4)} at level $i+1$. Together with the data constructed above, \textup{(I1)}--\textup{(I4)} now hold at level $i+1$ for the current element $x$.

\smallskip

We now replace the current element by the conjugate $x'=x^t$ constructed above, for which $x'M_{i+1}=yM_{i+1}$. Let $\widehat t\in H$ be the word obtained from the computed word for $t\in E$ by replacing the generators of $E$ by their fixed lifts to $H$, and put $\widehat x'=\widehat x^{\,\widehat t}$. This is a word lift of $x'$. Since $M_{i+1}\trianglelefteq E$, conjugation by $t$ gives $D_{i+1}(x')=D_{i+1}(x)^t$. Moreover, $N_0$ and $K_0$ are normal in $H$, and hence the definition of $\Lambda_{i+1}$ gives $\Lambda_{i+1}(x')=\Lambda_{i+1}(x)^{\widehat t}$.

The elements $e_1^t,\ldots,e_r^t$ lie in $D_{i+1}(x')$. Since $M\trianglelefteq E$, conjugation by~$t$ induces an automorphism of $E/M$, and the images of these elements generate $D_{i+1}(x')M/M$. Their words are obtained effectively by conjugating the previously computed words for $e_1,\ldots,e_r$ by the computed word for $t$. This gives \textup{(I1)} for the new current element $x'$.

Apply Lemma~\ref{lem:transport} with $Y=\Lambda_{i+1}(x)$ and $h=\widehat t$. Its two hypotheses are exactly the data already constructed at level $i+1$: the membership algorithm from \textup{(I2)} and the finite rational row description from \textup{(I3)}. The lemma therefore computes a membership algorithm for $\Lambda_{i+1}(x)^{\widehat t}=\Lambda_{i+1}(x')$ and a finite rational row description of its rational completion. Thus \textup{(I2)} and \textup{(I3)} hold for $x'$.

Finally, conjugation by $\widehat t$ is an automorphism of $N_0$ and, by Lemma~\ref{lem:tensor}\textup{(iii)}, extends to an automorphism of $(N_0)_p$. Since $N_{i+1}$ is characteristic in $N_0$ and $K_0\trianglelefteq H$, the subgroup $N_{i+1}K_0$ is normal in $H$; hence this automorphism preserves $(N_{i+1}K_0)_p$. It also identifies $(\Lambda_{i+1}(x))_p$ with $(\Lambda_{i+1}(x)^{\widehat t})_p=(\Lambda_{i+1}(x'))_p$. From $\widehat x'=\widehat x^{\,\widehat t}$ one obtains, for every $u\in(N_0)_p$,
$$
d_{x'}(u^{\widehat t})=d_x(u)^{\widehat t}.
$$
Therefore \eqref{eq:6.padic-next} is carried by conjugation by $\widehat t$ to
$$
(\Lambda_{i+1}(x'))_p=\{v\in(N_0)_p:d_{x'}(v)\in(N_{i+1}K_0)_p\},
$$
which is \textup{(I4)} for $x'$.

Replace the current element $x$ by $x'=x^t$, and replace its word lift $\widehat x$ by $\widehat x'$. At the same time replace $\gamma$ by $\gamma t$. Indeed, before this replacement $x=x_{\mathrm{original}}^\gamma$, and therefore
$$
x'=x^t=(x_{\mathrm{original}}^\gamma)^t=x_{\mathrm{original}}^{\gamma t}.
$$
Thus, after the replacements, we again have $x=x_{\mathrm{original}}^\gamma$. Moreover, $xM_{i+1}=yM_{i+1}$ by the choice of $t$, and \textup{(I1)}--\textup{(I4)} hold at level $i+1$ for the new current element $x$. Hence the induction hypothesis is established at level $i+1$.

\smallskip

The series \eqref{eq:6.2} runs from $M_1=M$ to $M_{c+1}=\{1\}$, so the procedure performs at most $c$ lifting steps. If the derivation equation fails at some level, there is no element of $E$ sending the current $x$ to $y$ modulo the next term of the series. Since the current element has the form $x=x_{\mathrm{original}}^\gamma$, any conjugating element from $x_{\mathrm{original}}$ to $y$ would give, after multiplication on the left by $\gamma^{-1}$, a conjugating element from the current $x$ to $y$, a contradiction. Hence the original inputs are not conjugate in $E$. If every step succeeds, then at the final level $M_{c+1}=\{1\}$, so the induction invariant gives $x=y$. Since $x=x_{\mathrm{original}}^\gamma$, the stored word $\gamma$ satisfies $x_{\mathrm{original}}^\gamma=y$ and is the required conjugating element.
\end{proof}

\section{Proof of the Main Theorem}\label{sec:6}

Theorem~\ref{thm:lifting} solves conjugacy in $E=H/K_0$. Since $K_0\leq K=\Ker(\phi)$, the epimorphism $\phi:H\to G$ induces an epimorphism
$$
\bar\phi:E=H/K_0\longrightarrow G,
\qquad
\bar\phi(hK_0)=\phi(h),
$$
whose kernel is
$$
F:=K/K_0.
$$
Moreover, since $K_0=K\cap H_0$, the map
\begin{equation}
 kK_0\in K/K_0\longmapsto kH_0\in H/H_0
 \label{eq:7.1}
\end{equation}
is injective, with image $KH_0/H_0$. Since $H/H_0$ is finite, so is $F$. Thus $G\simeq E/F$, and the final step reduces to a finite list of conjugacy tests in $E$.

We compute, once and for all, representatives of $F$. Since membership in both $K$ and $H_0$ is decidable, both subgroups are recursively enumerable.  Hence the fixed subgroup $KH_0$ is recursively enumerable, by enumerating products $kh_0$ with $k\in K$ and $h_0\in H_0$. Thus, \cite[Theorem~9.3.4]{LR04} decides membership in $KH_0$.  Since $H/H_0$ is finite, inspect its finitely many cosets and determine the image
\begin{equation}
 KH_0/H_0\leq H/H_0.
 \label{eq:7.2}
\end{equation}
For every coset in \eqref{eq:7.2}, enumerate all words in the fixed generators of $H$, retain those which lie in $K$ using the decidable membership test for $K$, and stop when a retained element belongs to the chosen coset of $H_0$.  The definition of \eqref{eq:7.2} guarantees that such an element exists, so each of these finitely many searches terminates. By injectivity of \eqref{eq:7.1}, the resulting finite list
\begin{equation}
 f_1,\ldots,f_m\in K
 \label{eq:7.3}
\end{equation}
has the property that $f_1K_0,\ldots,f_mK_0$ are precisely the elements of $F$, without repetition.

Let $x,y\in G$ be given by words and use fixed lifts of the generators of $G$ to compute lifts $\widetilde x,\widetilde y\in E$.  We show that $x$ and $y$ are conjugate in $G$ exactly when, for some $j\in\{1,\ldots,m\}$, there exists $c\in E$ such that
\begin{equation}
 \widetilde x^c=(f_jK_0)\widetilde y.
 \label{eq:7.5}
\end{equation}

Suppose first that $x^g=y$ for some $g\in G$, and choose a lift $c\in E$ of $g$. Then $\bar\phi(\widetilde x^c\widetilde y^{-1})=1$, so $\widetilde x^c\widetilde y^{-1}\in F=K/K_0$.  Hence $\widetilde x^c\widetilde y^{-1}=f_jK_0$ for some~$j$, and therefore \eqref{eq:7.5} holds.  Conversely, if \eqref{eq:7.5} holds for some $j$ and $c\in E$, then applying $\bar\phi$ gives $x^{\bar\phi(c)}=y$.  Thus the claimed equivalence follows.

The list \eqref{eq:7.3} is finite and
Theorem~\ref{thm:lifting} decides, for each $j=1,\ldots,m$, whether \eqref{eq:7.5} has a solution $c\in E$.  Therefore conjugacy in $G$ is decidable. When the answer is positive, Theorem~\ref{thm:lifting} returns a word representing an element $c\in E$ satisfying \eqref{eq:7.5}. Replacing each generator of $E$ occurring in this word by its fixed image in the generators of $G$ produces a word representing $\bar\phi(c)\in G$, and the preceding equivalence gives
$$
x^{\bar\phi(c)}=y.
$$
This proves the Main Theorem.

\section*{Acknowledgements}
The authors are members of the non-profit association ``AGTA -- Advances in Group Theory and Applications'' (www.advgrouptheory.com) and are supported by GNSAGA (INdAM).

During the preparation of this manuscript, the authors used OpenAI's ChatGPT as an interactive tool for checking mathematical arguments, exploring proof strategies, and improving the exposition. The statements, arguments, and proofs presented in the paper remain primarily the work of the authors, with ChatGPT used in a supporting role. The authors take full responsibility for the content.

\begin{flushleft}
\rule{8cm}{0.4pt}\\
\end{flushleft}

{
\sloppy
\noindent
Mattia Brescia, Ernesto Ingrosso, Marco Trombetti

\noindent 
Dipartimento di Matematica e Applicazioni ``Renato Caccioppoli''

\noindent
Università di Napoli Federico II

\noindent
Complesso Universitario Monte S. Angelo

\noindent
Via Cintia, Napoli (Italy)

\noindent
e-mail: mattia.brescia@unina.it; ernesto.ingrosso2@unina.it\\ \phantom{e-mail:} marco.trombetti@unina.it 

}

\end{document}